\documentclass[reqno,a4paper]{amsart}

\usepackage{amsmath, amsthm, amscd, amsfonts, amssymb, graphicx, color,mathrsfs,mathtools,enumerate}
\usepackage{lineno}
\usepackage{dsfont}
\modulolinenumbers[5]
\usepackage{bm}

\usepackage{comment}
\usepackage{graphicx}%
\usepackage{multirow}%
\usepackage[title]{appendix}%
\usepackage{xcolor}%
\usepackage{textcomp}%
\usepackage{manyfoot}%
\usepackage{booktabs}%
\usepackage{algorithm}%
\usepackage{algorithmicx}%
\usepackage{algpseudocode}%
\usepackage{listings}%
\usepackage{xcolor}

\usepackage{enumerate, enumitem}
\usepackage{hyperref}
\usepackage{float}

\newtheorem{theorem}{Theorem}

\newtheorem{thm}{Theorem}[section]
\newtheorem{lemma}[thm]{Lemma}
\newtheorem{proposition}[thm]{Proposition}

\newtheorem{rmk}[thm]{Remark}
\newtheorem{remark}[thm]{Remark}

\def\lc#1{{\color{blue}#1}}
\def\lc#1{#1}

\newtheorem{defn}[thm]{Definition}

\newcommand{\N}{\mathbb N}
\newcommand{\enne}{\mathbb N}
\newcommand{\R}{\mathbb R}

\newcommand{\OO}{\mathcal{O}}
\newcommand{\Id}{{\operatorname{Id}}}
\newcommand{\norm}[1]{{\left\|#1\right\|}}
\newcommand{\scal}[2]{{\left\langle #1,#2\right\rangle}}

\newcommand{\cL}{{\mathscr L}}
\renewcommand{\d}{{\mathrm d}}

\renewcommand{\P}{{\mathbb{P}}}
\newcommand{\Tr}{\operatorname{Tr}}
\newcommand{\E}{\mathbb E}

\allowdisplaybreaks

\makeatletter
\@namedef{subjclassname@2020}{%
  \textup{2020} Mathematics Subject Classification}
\makeatother

\begin{document}

\title[Pathwise uniqueness by noise for singular stochastic PDEs]
{Pathwise uniqueness by noise for singular stochastic PDEs}

\author{Davide Addona}
\address[Davide Addona]{Dipartimento di Scienze Matematiche, Fisiche e Informatiche, Universit\`a degli Studi di Parma, Parco Area delle Scienze 53/a (Campus), 43124 Parma, Italy}
\email{davide.addona@unipr.it}
\urladdr{}

\author{Davide A. Bignamini}
\address[Davide Bignamini]{Dipartimento di Scienza e Alta Tecnologia (DISAT), Universit\`a degli Studi dell'Insubria, Via Valleggio 11, 22100 Como, Italy}
\email{da.bignamini@uninsubria.it}
\urladdr{}

\author{Carlo Orrieri}
\address[Carlo Orrieri]{Department of Mathematics, Universit\`a di Pavia, Via Ferrata 1, 27100 Pavia, Italy}
\email{carlo.orrieri@unipv.it}
\urladdr{https://mate.unipv.it/orrieri}

\author{Luca Scarpa}
\address[Luca Scarpa]{Department of Mathematics, Politecnico di Milano, 
Via E.~Bonardi 9, 20133 Milano, Italy.}
\email{luca.scarpa@polimi.it}
\urladdr{https://sites.google.com/view/lucascarpa}



\begin{abstract}
Uniqueness for stochastic PDEs
with drift in differential form is a main 
open problem in the recent literature on regularisation by noise.
This paper establishes a self-contained theory
in the framework 
of stochastic evolution equations on separable Hilbert spaces
and provides a first 
result to address such an issue,
\lc{in the direction of both pathwise and weak uniqueness, as well as continuous dependence on the initial data.}
The singularity of the drift allows to 
achieve novel uniqueness results for 
several classes of examples, ranging from 
fluid-dynamics to phase-separation models.
\end{abstract}

\maketitle

\tableofcontents

\section{Introduction}
We are interested in pathwise 
regularisation-by-noise phenomena
for classes of SPDEs in the form 
\begin{align}
\label{SDE}
\d X + AX \,\d t=B(X)\,\d t+A^{-\delta}\,\d W\,, \qquad X(0)=x\,,   
\end{align}
where $A$ is a positive self-adjoint linear operator on 
a separable Hilbert space $H$ with domain $D(A)$,
$W$ is a $H$-cylindrical Wiener process and the parameter 
\lc{$\delta\in\left(-\frac12,\frac12\right]$} tunes the color of the noise.
The initial datum $x$ will be assumed to be either in $H$
or in $D(A^\alpha)$.
Motivated by numerous PDE examples, 
the nonlinearity $B$ is 
defined on
a Sobolev-type subspace of $H$
and is allowed to take values in larger dual spaces, namely 
\[
B:D(A^\alpha)\to D(A^{-\beta})
\]
\lc{for some $\alpha,\beta\in[0,1)$ with $\alpha+\beta<1$.}
Moreover, $B$ is only required to be
locally $\theta$-H\"older-continuous,
for some $\theta\in(0,1)$, so  that
the corresponding deterministic equation 
in general lacks uniqueness.

Pathwise uniqueness by noise for SPDEs with 
singular perturbations in such a form has been an open problem in 
the recent literature. This paper successfully establishes a self-contained framework to address it. 
Our strategy is twofold.
First, we construct the novel abstract theory 
for uniqueness by noise in such setting, by providing 
the precise range of parameters $\alpha,\beta,\delta$ for which 
weak/pathwise uniqueness or continuous dependence on the initial data 
hold. Secondly, we present some notable applications 
to specific examples of interest coming from 
fluid-dynamics and phase-separation, 
such as singular perturbations
of fractional heat PDEs, (possibly) hyper-viscous Burgers
and Navier-Stokes PDEs, and 
Cahn-Hilliard-type PDEs.

\subsection{Recent literature}
The general theory on regularisation by noise has reached 
a long history, starting from the contribution by Zvonkin \cite{Zv74} and Veretennikov \cite{V1980} for the finite-dimensional case,
and with the work \cite{GP1993}
for the infinite-dimensional setting.
Let us now briefly comment on the state of the art 
in the infinite-dimensional case.
Weak uniqueness by noise has been studied in 
\cite{ABGP2006, D2003, Z2000} in the context of equations with nondegenerate multiplicative noise and H\"older-continuous bounded drift. Still concerning weak uniqueness,
\cite{Pri2015, Pri2021, Pri2021c} tackled the
relevant case with drift of the form 
$B=A^{\frac12}F$, where $F$ is locally H\"older-continuous on $H$,
while \cite{BOS} considered the sub-critical case
with no H\"older assumption on the drift.
Finally, we also mention \cite{AB-weak}, where weak uniqueness for a class of stochastic differential equations evolving in a separable Hilbert space, which includes stochastic damped equations, is studied.
As far as pathwise uniqueness is concerned, 
we point out the results \cite{AddBig, AddBig2, Cer-Dap-Fla2013, 
Dap-Fla2010, Dap-Fla2014}, dealing with the particular setting $\alpha=\beta=0$. 
Moreover, the case of measurable drift from $H$ to $H$
has been studied in
\cite{Dap1, DPFPR2}, where pathwise uniqueness is obtained only 
for almost every initial datum with respect to a specific Gaussian measure.
\lc{It is worthwhile mentioning also the pioneering work \cite{FGP2010}, where the It\^o-Tanaka trick was employed to establish pathwise uniqueness for stochastic transport equations.
Lastly, we point out the recent contribution \cite{But},
where pathwise uniqueness is obtained in the case $\alpha=\beta=0$ and $B$ bounded and H\"older-continuous by using 
novel techniques based on the stochastic sewing.} For completeness, we also refer to the works 
\cite{AddMasPri23, KR2005, MP2017, Mas-Pri2024}
for further techniques for strong uniqueness.

\subsection{Discussion on the main results}
The paper provides three main results, 
on weak uniqueness, pathwise uniqueness and 
continuous dependence on the initial data, 
namely Theorems~\ref{th:1}, \ref{th:3}, and \ref{th:2} below.
The main requirement on the parameters $\alpha, \beta, \delta$
is that 
\[
  \alpha\in[0,1)\,, \qquad
  \delta\in \left(-\frac12+\alpha, \frac12\right]\,,\qquad
  \lc{\beta+\delta\in\left[0,\frac12\right]}\,.
\]
The range of the parameter $\alpha$ allows to 
consider drifts $B$ defined up to the effective domain 
$D(A)$ excluded. The range for $\delta$ 
may include both positive and negative values, corresponding 
to the case of coloured noise and rougher-than-cylindrical noise,
respectively. Eventually, the choice of $\beta$ has to be 
calibrated in terms of $\delta$: in particular, 
we stress the bound from above
\[
  \beta+\delta\leq \frac12\,,
\]
with $\beta+\delta=\frac12$ representing a critical case
for the theory.
A rough interpretation of the balance between the parameters
is that on the one hand, smaller values of $\delta$
provide better regularisation by noise, 
while on the other hand higher values of $\delta$
allow to consider wider range of equations (e.g.~with 
higher $\alpha$ or higher dimension). This will be 
discussed in detail when dealing with the examples.
In this framework, Theorem~\ref{th:1} provides 
a weak uniqueness result without any restriction on the 
range of the H\"older coefficient $\theta$,
and for initial data either in $D(A^\alpha)$ or in $H$
depending on whether $B$ is unbounded or bounded, respectively.
The pathwise uniqueness result is given in
Theorem~\ref{th:2}: here, one needs some further joint conditions
on all the coefficients, in particular either that 
\[
  (1-\theta)(\delta-\alpha)<\frac\theta2
\]
if the initial datum is taken in $D(A^\alpha)$, or even that 
\[
\theta\alpha + (1-\theta)\delta<\frac\theta2
\]
if the initial datum is taken in $H$ and $B$ is \lc{bounded and globally H\"older-continuous}. In the critical case $\beta+\delta=\frac{1}{2}$, \lc{$\delta\geq\alpha$ and} an additional smallness condition on $B$ \lc{are} required, see Remark \ref{rmk:pic} for details.
We note that the closer $\theta$ is to $1$
, the less restrictive such inequalities become, whereas 
for $\theta$ close to $0$ they add a constraint in the range of 
$\delta$ and $\alpha$.
Eventually, \lc{Theorem~\ref{th:3} contains sufficient conditions
for continuous dependence with respect to the initial data
when $B$ is bounded and globally H\"older-continuous,
in the topology of $D(A^s)$ for $s\in[-\beta,\alpha]$,
under restrictions on the coefficients analogous to the ones above.}

\subsection{Main novelties}
The main achievement of the paper is
the pathwise uniqueness result presented in Theorem~\ref{th:2}.
This is the first contribution in the literature on pathwise regularisation 
for SPDEs with singular perturbations in the form \eqref{SDE}.
Indeed, the only results available in the literature
are \cite{AddBig, AddBig2, AddMasPri23, Dap-Fla2010, Dap-Fla2014},
which deal only with the case $\alpha = \beta = 0$,
i.e., nonlinear perturbations of order zero.
Instead, Theorems~\ref{th:3} and \ref{th:2} below 
address both pathwise uniqueness and 
continuous dependence
for genuine singular perturbations in differential form, 
by including 
a wide range of the coefficients $\alpha$ and $\beta$.
These results can be viewed as a pathwise uniqueness 
counterpart of the most recent
weak uniqueness results \cite{BOS, Pri2021, Pri2021c}:
see Section~\ref{Exam} for a detailed discussion of the applications.

\lc{A further main novelty, which is relevant already in the case $\alpha=\beta=0$ of nonlinear perturbations of order zero, is the removal of the trace-class-type condition required in the previous literature on pathwise uniqueness.
Indeed, in \cite{AddBig, AddBig2, AddMasPri23, Dap-Fla2010, Dap-Fla2014}, pathwise uniqueness is proved under the assumption
   \begin{equation}\label{CDF}
\sum_{k\in\N}\lambda_k^{-1}\norm{\scal{B}{e_k}}_{C^\theta_b(H)}^2<+\infty,
\end{equation}
where $(e_k)_k$ are eigenvectors of $A$ with corresponding eigenvalues
$(\lambda_k)_k$.
This significantly restricts the class of H\"older continuous perturbations \( B \) that can be considered. For example, if \( A \) is a realisation of the negative Laplace operator, then \eqref{CDF} is satisfied for every \( \theta \)-H\"older continuous function \( B \) only in space dimension $1$. Instead, in Theorem~\ref{th:2},
condition \eqref{CDF} is considerably improved with the assumption \eqref{trace1} below, 
that in the case $\alpha=\beta=0$ 
with the negative Laplace operator
also allows for higher dimensions and every \( \theta \)-H\"older continuous function \( B \).
Condition \eqref{CDF} arises in the estimate of the stochastic integral produced by the It\^o-Tanaka trick (namely the term $I_7$ in the proof of Theorem~\ref{th:3} below).
The main idea to remove condition \eqref{CDF}
is to scale the $k$-th approximated Kolmogorov equation by a zero-order coefficient proportional to $\lambda_k$, see \eqref{eq:kolmI} below. This provides an additional decay in $k$ of the second derivatives of its solution, so that the term $I_7$ can be controlled and \eqref{CDF} can be replaced by a condition which does not depend on the structure of $B$, but only on its H\"older exponent $\theta$. We refer to Section~\ref{sec:Kolmo-sol} for details.
}


As a by-product of the results of the present work,
we also slightly extend the current state of the art 
of regularisation by noise 
in the direction of weak uniqueness, see \cite{AB-weak,BOS,Cho-Gol1995,Kun2013,Pri2021,Pri2021c}.
First, in \cite{Cho-Gol1995, Kun2013} the authors prove weak uniqueness in an abstract framework, but only 
in the case $\alpha = \beta = 0$.
Secondly, the seminal works \cite{Pri2021,Pri2021c}
addresses the interesting critical case $\alpha=0$ and $\beta=\frac12$, 
under a local H\"older condition on the nonlinearity
and using a space-time white noise ($\delta=0$).
Eventually, in the range $\alpha\geq 0$ and $\delta\geq0$, \cite{BOS} deals with the sub-critical case $\beta+\delta<\frac12$ but without H\"older assumptions on the nonlinearity. Under a local H\"older condition on the nonlinearity, Theorem \ref{th:1} unifies and extends the results in \cite{BOS,Pri2021,Pri2021c} allowing the study of the critical case even when $\alpha>0$
and when the noise is white or coloured.
The possibility of using a coloured noise 
allows to obtain uniqueness in higher dimensions when, e.g., 
$A$ is the realisation 
of the negative Laplace operator on a bounded domain.
In order to prove Theorem \ref{th:1}, we simplify the techniques proposed in \cite{BOS}, by exploiting the approximation method introduced in \cite{AddBig} and, for the critical case, we follow the idea in \cite{Pri2021,Pri2021c}.
Let us point out that in this context, proving weak uniqueness in the critical cases without H\"older assumptions on the nonlinearity remains an open problem.

\subsection{Applications}
From the application point of view, 
the main focus of Theorems~\ref{th:1},
\ref{th:2}, and \ref{th:3} concerns
stability-by-noise effects with respect to 
singular perturbations for a wide class of 
specific PDE examples. This includes 
physical models ranging from fluid-dynamics to phase-separation.
We briefly discuss here 
some main applications and examples, 
as detailed in Section~\ref{Exam}.

The paradigmatic example to start with is the stochastic heat equation with H\"older continuous perturbations.
This has been surely the most studied model in the recent literature 
of regularization by noise, see e.g.~\cite{AddBig, AddBig2, AB-weak,
AddMasPri23, Dap-Fla2010, Dap1, DPFPR2, MP2017}
where the nonlinear perturbation is of order zero, as pointed out before.
Here, we consider the stochastic heat equation even under 
more generality, by including both possible fractional diffusion 
effects and genuinely singular nonlinear perturbations in differential form, namely
  \begin{align*}
  \d X +(- \Delta)^\gamma X\,\d t = 
  (-\Delta)^\nu F((-\Delta)^\mu X)\,\d t 
  + (-\Delta)^{-\rho}\,\d W
  \end{align*}
  for some $\gamma>0$, $\mu,\nu\geq0$, $\rho>-\frac12$,
  and where $F:\R\to\R$ is locally $\theta$-H\"older continuous.
Our results provide weak and pathwise uniqueness by noise
for such a class of models, even up to space dimension $3$, and 
for a wide range of the parameters. For details 
we refer to Subsections~\ref{ssec:heat}--\ref{ssec:heat_fr}. 

The second case of examples that we provide concern with singular H\"older-type perturbations of fluid-dynamical models,
such as the Burgers and the Navier-Stokes equations,
which are considered both in their classical and 
hyper-viscous form, namely
\begin{align*}
\d X +(-\Delta)^\gamma X \,\d t + 
(X\cdot \nabla)X\,\d t= 
(-\Delta)^\nu F((-\Delta)^\mu X)\,\d t+
(-\Delta)^{-\rho}\,\d W 
\end{align*}
and 
\begin{align*}
\begin{cases}
\d X +(-\Delta)^\gamma X \,\d t + 
(X\cdot \nabla)X\,\d t + \nabla p\,\d t
= (-\Delta)^\nu F((-\Delta)^\mu X)\,\d t+
(-\Delta)^{-\rho}\,\d W\,,\\ 
\nabla\cdot X = 0 \,,
\end{cases}
\end{align*}
for some $\gamma>0$, $\mu,\nu\geq0$, $\rho>-\frac12$, and where $F:\R\to\R$ is locally $\theta$-H\"older continuous.
The results of the paper
establish weak and pathwise uniqueness under 
specific tuning of the hyperviscosity parameter, the H\"older exponent and the regularity of the initial datum.
In this direction, weak uniqueness is obtained up to dimension $3$ for every choice of the H\"older exponent,
under a suitable balance between the hyperviscosity parameter and 
the regularity of the initial datum. Intuitively, the more regular the initial datum is, the less hyperviscosity 
is required for uniqueness. The precise range of the parameters is detailed in Table~\ref{tab1}.
As far as pathwise uniqueness is concerned, 
the situation is more delicate, since also the H\"older exponent comes into play. In this case, we present two main ways
of optimising the choice of the parameters:
either we take extremely rough perturbations (i.e., with H\"older exponent close to $0$) or we consider more gentle perturbations (i.e., with H\"older exponent close to $1$).
In both cases, the hyperviscosity is optimised depending 
on the space dimension and the initial datum:
this is detailed in Tables~\ref{tab2} and \ref{tab3}.
The specific technique employed in the paper results
in an analogous range of parameters both 
for the Burgers and the Navier-Stokes equations.

As far as the Burgers equation is concerned, when $\gamma\geq 1$ 
our results provide a pathwise uniqueness counterpart to 
\cite{Pri2021, Pri2021c}, also including higher dimensions.
Moreover, 
we also point out that our framework 
covers the perturbed Burgers equation in the range $\gamma\in\left(\textcolor{blue}{\frac34},1\right)$. The well-posedness of Burgers equation in $L^2([0,1])$ when $\gamma\in\left(\frac34,1\right)$ is studied in the contribution \cite{Brz-Deb2007}
in dimension $d=1$ with multiplicative noise. 

Let us emphasise here the well-studied problem of the three-dimensional stochastic Navier-Stokes equation. In this case, our results establish pathwise uniqueness 
in the following regimes:
\begin{itemize}
\item hyperviscosity $\gamma=\frac74^+$, \quad
H\"older coefficient $\theta=1^-$, \quad
    initial datum in $H^{\frac12}_{00}$;
    \item hyperviscosity $\gamma=\frac32^+$, \quad
    H\"older coefficient $\theta=1^-$, \quad
    initial datum in $H^{1}_{0}$;
    \item hyperviscosity $\gamma=\frac94^+$, \quad
    H\"older coefficient $\theta=0^+$, \quad
    initial datum in $H^{1}_{0}$.
\end{itemize}
Let us stress that the novelty of our approach 
is given by the presence of the singular H\"older perturbation 
term $F$ in differential form
in the equations.
The presence of such an irregular perturbation forces
the hyperviscosity $\gamma$ to be slightly larger 
with respect to the expected ones that appear in the literature
without perturbation (i.e.~ with $F\equiv0$).
In this direction, the literature on the classical Navier-Stokes 
equation has witnessed an enormous development, starting from 
the foundational papers by Lions \cite{Lions} on the deterministic case
and Flandoli-Gatarek \cite{FG1995} on the stochastic case.
More recently, important improvements on 
regularisation by noise
for the three-dimensional Navier-Stokes
have been obtained in \cite{FL2021} in
the context of vorticity blow-up of solutions
and in \cite{HZZ2023,MZZ} in the direction of non-uniqueness.
In \cite{Agresti} uniqueness with high probability and 
large initial data has been obtained in the case $\gamma>1$,
whereas in \cite{HLL2025} global well-posedness 
is achieved also in the case $\gamma=1$
for any initial data in the critical space $H^{\frac12}$.

Eventually, the last class of examples concerns with phase-separation models such as reaction-diffusion equations and Cahn-Hilliard-type equations. In this direction, we point out that in \cite{Cer-Dap-Fla2013} pathwise uniqueness for stochastic reaction-diffusion equations was studied in the case $\delta=0$ and when the nonlinearity is bounded and H\"older continuous on $C([0,1])$. In our setting, by tuning the parameters $\alpha$ and $\beta$ we are able to cover also the case of polynomial growth by exploiting classical Sobolev embeddings and working in $D(A^\alpha)$. For Cahn-Hilliard models, only weak uniqueness results were available in the literature \cite{BOS, Pri2021, Pri2021c}: here, we provide the pathwise uniqueness counterpart.

\subsection{Technical strategy}\label{ssec:tech}
Let us briefly comment on the 
main technical aspects of the arguments.
One of the classical approaches to regularization by noise consists in studying the associated Kolmogorov equation
to the SPDE and to exploit the so-called It\^o-Tanaka trick, i.e., to replace the nonlinear term $B$ by means of addends which depends on the (more regular) solution to the Kolmogorov equation itself. Here, we propose a variation of such procedure: instead of working directly on the infinite dimensional Kolmogorov equation (as was done e.g.~in \cite{BOS}), we introduce a family of finite-dimensional approximations on the line of \cite{AddBig} calibrated with a suitable scaling in the form of
\begin{equation}
\label{eq:kolmI}
\bar c\lambda_k u_{n,k}(x)
-\frac{1}{2}\Tr\left[A_n^{-2\delta}D^2 u_{n,k}(x)\right]
+\scal{A_nx}{D u_{n,k}(x)}
=\langle B_n(x), Du_{n,k}(x)\rangle +\scal{B(x)}{e_k},\, 
\end{equation}
where $(e_i)_{i\in\N_+}$ is a complete orthonormal system of $H$ consisting of eigenvectors of $A$ with corresponding eigenvalues $(\lambda_i)_{i\in\N_+}$, 
$n\in\N_+$, $k\in\{1,\ldots,n\}$, $x\in{\rm Span}\{e_1,\ldots,e_n\}$, $B_n(x):=\sum_{i=1}^n\scal{B(x)}{e_i}e_i$, $A_n x:=\sum_{i=1}^n\lambda_i\scal{x}{e_i}e_i$, and $\overline{c}$ is a suitable constant.
\lc{This approach generalises the one in \cite{FGP2010},
where the coefficient of the zero-order term 
was taken large but independent of $k$, and 
is exactly the idea that allows to remove condition \eqref{CDF}.}

On the one hand, 
the proof of weak uniqueness requires uniform estimates only with 
respect to $n$, with $k$ being fixed (e.g.~$k=1$). On the other hand, the study of pathwise uniqueness is based on uniform estimates both with respect to $n$ and $k$ jointly, involving suitable norms of the derivatives of
$u_n=\sum_{k=1}^nu_{n,k}e_k$ up to second order. The It\^o-Tanaka argument is applied at the approximation level with $k$ and $n$ fixed, and the passage to the limit is performed at the very end by summing over $k$ and letting $n\to\infty$. In this case, the It\^o-Tanaka trick does not consist in replacing the nonlinear term $B$ but, exploiting the approximation $B_n$ in \eqref{eq:kolmI}, we add some correction terms which eventually compensate the bad behavior of the drift part. The advantage of such procedure is that one avoids to 
treat the infinite dimensional vector-valued Kolmogorov-type equation associated to \eqref{eq:kolmI}, for which a general theory is not directly available. Indeed, this would read formally as
\begin{equation}\label{eq:quasi-kolmo}
\overline c A[u(x)]+(\mathcal{L}u)(x)= Du(x)B(x) +B(x)\,,
\quad x\in H\,,
\end{equation}
where $\mathcal{L}$ is an Ornstein-Uhlenbeck-type operator.
\lc{We stress that our approach does not rely on the convergence of the functions $(u_{n,k})_{n,k}$, but only on uniform estimates
in $n$ and $k$, see Proposition \ref{prop:est}. The choice of the operator $A$ in front of $u(x)$ in the (formal) vector-valued Kolmogorov equation \eqref{eq:quasi-kolmo} (and, analogously, of $\lambda_k$ in front of $u_{n,k}$ in \eqref{eq:kolmI}) is optimal in the sense that, if one replaces $Au(x)$
in \eqref{eq:quasi-kolmo} by 
$A^\gamma u(x)$ (and, analogously, $\lambda_{n,k}$ in \eqref{eq:kolmI} by $\lambda_{k,n}^\gamma$) for some power $\gamma>0$, then the weakest assumptions on the coefficients of \eqref{SDE} ensuring and pathwise uniqueness are obtained for $\gamma=1$ (note that the proof of weak uniqueness only involves the equation with $k=1$, so that it is not affected by this choice).
Roughly speaking, a larger zero-order coefficient improves the decay in $\gamma$ of the derivatives of $u_{n,k}$, which is crucial to control the stochastic term $I_7$ in the proof of Theorem \ref{th:3}, but it also yields a worse decay for the term $I_2$. Indeed, if $\gamma>1$ then we obtain a worst estimate of $I_2$, while $\gamma<1$ would give a worst estimate of $I_7$. Since $I_2$ and $I_7$ are the terms with the critical behaviour, the optimal choice is $\kappa=1$. For details we refer to Section~\ref{sec:Kolmo-sol}.}

\subsection{Structure of the paper}
We finally present the main contents of the sections.
In Section~\ref{sec:main} we set the 
assumptions and state the main results of the paper, namely 
Theorems~\ref{th:1}, \ref{th:2}, and \ref{th:3}.
In Section~\ref{sec:Kolmo-sol}, we study 
the elliptic Kolmogorov equation in the special case $\alpha=0$
and when $B$ is bounded. 
Section~\ref{sec:U0} deals with the proofs of weak uniqueness, pathwise uniqueness and continuous dependence in the general case. Section~\ref{Exam} is devoted to discussing the
application of the theory to 
specific examples. Eventually, 
the Appendix \ref{sec:app} contains a technical result 
used in Section~\ref{sec:U0}
that allows to consider more general initial data.

\section{Assumptions and main results}
\label{sec:main}
\subsection{Notation}
We set $\R_+:=[0,+\infty)$ and $\N_+:=\N\setminus\{0\}$. If $K_1, K_2$ are two separable Hilbert spaces with inner product $(\cdot,\cdot)_{K_1}$ and $(\cdot,\cdot)_{K_2}$, respectively, and $\psi:K_1\to K_2$ is a twice Fr\'echet-differentiable function, we denote the first and second differentials of $\psi$
by $D\psi:K_1\to\cL(K_1,K_2)$ and
$D^2\psi: K_1\to \cL(K_1,\cL(K_1,K_2))
\cong \cL_{(2)}(K_1\times K_1;K_2)$, respectively, where $\cL(K_1,K_2)$ denotes the space of bounded linear operators and $\cL_{(2)}(K_1\times K_1;K_2)$ denotes the space of bilinear forms. If $K_1=K_2$, then we simply write $\cL(K_1)$.

Typically, we will deal with the choice $K_2=\R$; as usual, in this case we use the classical identifications $K_1^*\cong K_1$ and $\cL_{(2)}(K_1\times K_1, \R)\cong\cL(K_1)$.
This means that for a twice-differentiable function $\psi:K_1\to\R$, we can identify the first and second differentials at $x\in K_1$ with an element of $K_1$ and $\cL(K_1)$, respectively, still denoted by $D\psi(x)$ and $D^2\psi(x)$. More precisely, for any increments $v,w\in K_1$ and $x\in K_1$ the above identification reads as
\[
D\psi(x)[v]=\left(D\psi(x), v\right)_{K_1}
\]
and
\[
  D^2\psi(x)[v,w]=\left(D^2\psi(x)v, w\right)_{K_1}
  =\left(D^2\psi(x)w, v\right)_{K_1}.
\]
Let $k\in\N$, we denote by $C^k_b(K_1;K_2)$ the space of 
continuously $k$-differentiable functions from $K_1$ into $K_2$ endowed with their natural norms. If $\vartheta\in(0,1)$, then we denote by $C_b^\vartheta(K_1;K_2)$ the space of bounded 
$\vartheta$-H\"older continuous functions from $K_1$ to $K_2$ with associated norm
\[
\|\psi\|_{C^\vartheta_b(K_1;K_2)}
:=\|\psi\|_{C^0_b(K_1;K_2)}+[\psi]_{C^\vartheta_b(K_1;K_2)}\,,
\]
where
\[    [\psi]_{C^\vartheta_b(K_1;K_2)}:=
\sup_{x,y\in K_1, \, x\neq y}
\frac{\|\psi(x)-\psi(y)\|_{K_2}}{\|x-y\|_{K_1}^\vartheta}\,,
\quad \psi\in C_b^\vartheta(K_1;K_2)\,.
\]
If $K_2=\R$, then we simply write 
$C^k_b(K_1)$ and $C_b^\vartheta(K_1)$, respectively. Further, we denote by $C_{\rm loc}^\vartheta(K_1;K_2)$ the space of locally bounded and $\vartheta$-H\"older continuous functions, i.e., the space of continuous functions $\psi$ such that for every $r>0$ the function $\psi$ is bounded on $B_{K_1}(r)$ (the ball in $K_1$ centered at $0$ with radius $r$) and
\begin{align*}
\sup_{x,y\in B_{K_1}(r), x\neq y}\frac{\|\psi(x)-\psi(y)\|_{K_2}}{\|x-y\|_{K_1}^\vartheta}<\infty.    
\end{align*}

\subsection{Assumptions}
We introduce the following assumptions.
\begin{enumerate}[start=0,label={{(H\arabic*})}]
\item \label{H0} $H$ is a real separable Hilbert space
with scalar product $(\cdot, \cdot)_H$. 
\item \label{H1} $A: D(A)\subseteq H\rightarrow H$ is a linear self-adjoint operator with dense domain. Moreover there exists an orthonormal basis $\{e_k:k\in\N_+\}$ of $H$ consisting of eigenvectors of $A$ and an increasing sequence $(\lambda_k)_{k\in\N_+}$ such that $\lambda_1>0$, $\lambda_k\to+\infty$ as $k\to\infty$ and
\[
Ae_k=\lambda_k e_k \quad\forall\,k\in\N_+\,.
\]
\end{enumerate}
\begin{enumerate}[start=2,label={{(H\arabic*})}]
\item \label{H2} 
\lc{$\alpha,\beta\in[0,1)$,
$\delta\in(-\frac12+\alpha, \frac12]$, with 
\[
  \beta+\delta\in
  \begin{cases}
      [0,\frac12] \quad&\text{if } \delta\geq\alpha\,,\\
      [0,\frac12) \quad&\text{if } \delta<\alpha\,.
  \end{cases}
\]}
\item \label{H3} There exists $\eta\in (0,1)$ such that $A^{-(1+2\delta)+2\alpha+\eta}\in\cL^1(H,H)$.
\item \label{H4} 
$B\in C^\theta_{{\rm loc}}(D(A^\alpha);D(A^{-\beta}))$ for some $\theta\in (0,1)$.
\item \label{H5} 
In the regime $\beta+\delta=\frac12$,
there exists $z_0\in D(A^{-\beta})$ such that 
\[
\sup_{x\in D(A^{\alpha})}\norm{B(x)-z_0}_{D(A^{-\beta})}=:C_{\tilde B}<
\frac{\theta(1-\theta)(1-\beta-\alpha)}
{4M_{\alpha+\beta,\delta-\alpha,\theta}(2-\theta)}\,,
\]
where $M_{\alpha+\beta,\delta-\alpha,\theta}$ is the implicit constant 
appearing in Proposition~\ref{stimeHolder}.
\end{enumerate}

We note that condition \ref{H1} implies that $-A$ is the infinitesimal generator of a strongly continuous analytic symmetric semigroup of contractions $e^{-tA}$ on $H$. Moreover, by standard spectral theory arguments one can define the fractional spaces $D(A^\sigma)$ for every $\sigma \in (0,1]$, as well as the corresponding dual spaces $D(A^{-\sigma}):=D(A^\sigma)^*$, so that the condition \ref{H4} makes sense. The duality between any 
pairing $(D(A^\sigma), D(A^{-\sigma}))$ is denoted by $\langle\cdot,\cdot \rangle$, and the norm in $D(A^\sigma)$ is denoted by $\|\cdot\|_\sigma$.
Furthermore, we recall that for every $\gamma\in\R$ the semigroup $e^{-tA}$
extends to a strongly continuous analytic symmetric semigroup on $D(A^{\gamma})$: such an extension will be denoted by the same symbol, as usual.
Note that Proposition~\ref{stimeHolder}
only relies on \ref{H0}--\ref{H3}, 
so that assumption \ref{H5} is well-defined.
\lc{Furthermore, note that, in the critical regime $\beta+\delta=\frac12$, assumption \ref{H5} implies that $B$ is bounded. We stress that \ref{H5} is strictly stronger than the boundedness of $B$, since it requires $B$ to be uniformly close to a constant $z_0\in D(A^{-\beta})$, with an explicit threshold.}
\subsection{Main results}

In this paper we will consider the mild formulation of \eqref{SDE}.

\begin{defn}\label{def-sol}
Assume \ref{H0}--\ref{H4}.
A weak mild solution to \eqref{SDE} is a pair 
$(X,W)$ where $W$ is a $H$-cylindrical Wiener process defined on a filtered probability space 
$(\Omega,\mathcal{F},\{\mathcal{F}_t\}_{t\geq0},\mathbb{P})$ and 
$X$ is a $H$-valued progressively measurable process such that 
\begin{alignat*}{2}
X &\in C^0(\mathbb R_+; H)\cap C^0((0,+\infty); D(A^\alpha))
\quad \P\text{-a.s.} \qquad&&\text{if } x\in H\,,\\
X &\in C^0(\mathbb R_+; D(A^\alpha))
\quad \P\text{-a.s.} \qquad&&\text{if } x\in D(A^\alpha)\,,
\end{alignat*}
for every $t>0$ it holds $s\mapsto e^{-(t-s)A}B(X(s)) \in L^1(0,t;H)$, $\P\text{-a.s.}$, and
\begin{align}\label{mild_al}
X(t)=e^{-tA}x+\int_0^te^{-(t-s)A}B(X(s))\,\d s+W_A(t)
\quad\forall\,t\geq0\,, \quad \mathbb{P}\text{-a.s.},   
\end{align}
where $W_A$ is the stochastic convolution process given by 
\begin{equation*}
W_A(t):=\int^t_0e^{-(t-s)A}A^{-\delta}\,\d W(s)\,, \quad t\geq0\,.
\end{equation*}
\end{defn}

\begin{rmk}\label{rmk:conv}
We stress that thanks to assumption \ref{H2}, 
one has $\alpha+\beta\in[0,1)$ so that
Definition~\ref{def-sol} is well-posed. Let us show that 
all the three terms in \eqref{mild_al} are well-defined:
\begin{itemize}
\item The map $t\mapsto e^{-tA }x$ belongs to  $C^0(\mathbb R_+; H)\cap C^0((0,+\infty); D(A^\alpha))$ if $x\in H$, and also to
$C^0(\mathbb{R}_+; D(A^\alpha))$ if $x\in D(A^\alpha)$.
\item If $x\in H$, then condition 
$s\mapsto e^{-(t-s)A}B(X(s)) \in L^1(0,t; H)$
ensures the well-definition of the deterministic convolution term. For example, this is satisfied when
$B$ is bounded thanks to deterministic maximal regularity.
Moreover, if also $x\in D(A^\alpha)$, \ref{H4} implies that 
$B(X)\in C^0(\mathbb R_+; D(A^{-\beta}))$ $\mathbb P$-a.s. and it even holds that $s\mapsto e^{-(t-s)A}B(X(s)) \in L^1(0,t;D(A^\alpha))$, $\P\text{-a.s.}$ since $\alpha+\beta\in[0,1)$. 
Deterministic maximal regularity ensure that  
the deterministic convolution process in \eqref{mild_al} is continuous with values in $D(A^\alpha)$.
\item  Condition \ref{H3} implies that
\[
\int_0^{t}\frac{1}{s^{2\alpha+\eta}} \Tr\left[e^{-2sA}A^{-2\delta}\right]\,\d s<+\infty
\quad\forall\,t\geq0\,,
\]
for some $\eta>0$, which guarantees that 
$W_A\in C^0(\R_+;D(A^\alpha))$, see \cite[Chapter 4]{Dap-Zab14}.
\end{itemize}
\end{rmk}

\begin{rmk}\label{rmk:ex_eq}
We point out that if $x\in D(A^\alpha)$ and $B$ is bounded, by arguing as in \cite[Thm.~A1]{BOS}, then the existence of a weak mild solution for \eqref{SDE} holds under \ref{H0}--\ref{H4}.
If $x\in H$, existence of solutions can be shown 
in each specific example according to the explicit form of $B$.
We refer to \cite[Thm.~3.6 and Prop.~6.8]{Kun2013} for the equivalence between martingale, analytic weak, analytic weak mild and mild solutions.  
\end{rmk}

Here we are interested in studying the weak and pathwise uniqueness of the mild solutions to \eqref{SDE}.

\begin{defn}\label{uniqueness}
Assume \ref{H0}--\ref{H4} and let $H_0\subseteq H$. 
\begin{itemize}
\item[(Weak)] We say that weak uniqueness 
holds for \eqref{SDE} with initial data in $H_0$
if for every $x\in H_0$ and
for every pair $(X_1,W_1)$ and $(X_2,W_2)$ of weak mild solutions to \eqref{SDE} defined on the same probability space $(\Omega,\mathcal F,\{\mathcal F_t\}_{t\geq0},\mathbb P)$ with $X_1(0)=X_2(0)=x$, 
it holds that $X_1$ and $X_2$ have the same law on $C^0(\R_+; H)$, namely that for every measurable bounded $\psi:C^0(\R_+; H)\to \R$
\[
\E\left[\psi(X_1)\right] =\E\left[\psi(X_2)\right]\,.
\]
\item[(Pathwise)] We say that pathwise $($or strong$)$ uniqueness holds for \eqref{SDE} with initial data in $H_0$
if for every $x\in H_0$ and for every pair $(X_1,W)$ and $(X_2,W)$ of weak mild solutions to \eqref{SDE}
with the same $H$-cylindrical process $W$ defined on the same probability space $(\Omega,\mathcal F,\{\mathcal F_t\}_{t\geq0},\mathbb P)$ with $X_1(0)=X_2(0)=x$, it holds that
\[
\mathbb{P}\left(X_1(t)=X_2(t), \;\;\forall\,t\geq0 \right)=1\,.
\]
\end{itemize}
\end{defn}

We state now the three main results of the paper. 
These concern, respectively, weak uniqueness, pathwise uniqueness, and continuous dependence on the initial data.

\begin{theorem}[Weak uniqueness]
\label{th:1}
Assume \ref{H0}--\ref{H4}.
Then, weak uniqueness holds for \eqref{SDE}
with initial data in $D(A^\alpha)$,
in the sense of
Definition~\ref{uniqueness}.
Moreover, if $B$ is bounded, then
weak uniqueness holds for \eqref{SDE}
with initial data in $H$,
in the sense of
Definition~\ref{uniqueness}.
\end{theorem}

\begin{theorem}[Continuous dependence]
\label{th:3}
Assume \ref{H0}--\ref{H5}, \lc{$B\in C^\theta_b(D(A^\alpha); D(A^{-\beta}))$, 
let $s\in[-\beta,\alpha]$, and suppose that
\begin{align}
\label{trace_s}
&\exists\,\varepsilon \in (0,1):\quad
A^{-(1+\theta)+2\beta+2(1-\theta)\max\{s,\delta\}+2\theta\alpha+2\varepsilon}\in\cL^1(H,H)\,,\\
\label{ineq_s}
&\theta\alpha+(1-\theta)\max\{s,\delta\}<s+\frac\theta2\,.
\end{align}
Then, there exists a constant $\mathfrak L>0$, depending only on the structural data $\alpha,\beta,\delta,\theta,T,B,s$, such that, for every weak mild solutions $(X,W)$ and $(Y,W)$ to \eqref{SDE} with initial data $x,y\in H\cap D(A^s)$, respectively, in the sense of Definition~\ref{def-sol}, with the same cylindrical Wiener process $W$ and defined on the same probability space, it holds that 
\begin{equation}
\label{cont_dep_s}
\|X-Y\|_{C^0([0,T];L^2(\Omega; D(A^s)))}\leq \mathfrak L\|x-y\|_{D(A^s)}\,.
\end{equation}
}
\end{theorem}

\begin{remark}\label{rmk:conv2}
    Note that the boundedness assumption on $B$
    in Theorem~\ref{th:3} ensures by deterministic
    and stochastic maximal regularity that 
    $X,Y\in L^2(\Omega; C^0([0,T]; D(A^s)))$
    for all $s\leq \alpha$. See also Remark~\ref{rmk:conv}.
\end{remark}

\begin{theorem}[Pathwise uniqueness]
\label{th:2}
Assume \ref{H0}--\ref{H5} and that
\begin{align}
\label{trace1}
&\exists\,\varepsilon \in (0,1):\quad
A^{-(1+\theta)+2\beta+2(1-\theta)\delta+2\theta\alpha+2\varepsilon}\in\cL^1(H,H)\,,\\
&\label{ineq2_al}
\theta\alpha+(1-\theta)\delta<\alpha + \frac\theta{2}\,.
\end{align}
Then, pathwise uniqueness holds for \eqref{SDE}
with initial data in $D(A^\alpha)$, in the sense of Definition~\ref{uniqueness}.
Furthermore, if also \lc{$B\in C^\theta_b(D(A^\alpha);D(A^{-\beta}))$} and 
\begin{equation}\label{ineq3_al}
\theta\alpha+(1-\theta)\delta< \frac\theta{2}
\quad\text{if } \delta\geq0\,,
\end{equation}
then, pathwise uniqueness holds for \eqref{SDE}
with initial data in $H$, in the sense of Definition~\ref{uniqueness}.
\end{theorem}

\begin{rmk}\label{rmk:pic}
Notice that assumption \ref{H5} is required only in Theorems~\ref{th:2} and~\ref{th:3} in the critical regime $\beta + \delta = \frac{1}{2}$.
The techniques from 
 \cite[Thm.~26]{Pri2015} and \cite[Lem.~22]{Pri2021} permit to remove it in the weak uniqueness case only, and do not apply to the pathwise setting.
 Nevertheless, this is not restrictive from the  point of view of the applications, where suitable tuning of the parameters takes place. The open problem of removing assumption \ref{H5} in the pathwise uniqueness case is surely interesting and
 currently under investigation.
\end{rmk}


\section{The Kolmogorov equation}
\label{sec:Kolmo-sol}

In this section we 
study the Kolmogorov equation
associated to \eqref{SDE}: this will be done 
only in the case $\alpha=0$.
For this reason, 
we reformulate for clarity of presentation
the assumptions \ref{H2}, \ref{H3}, and \ref{H4} in the 
special case $\alpha=0$.
\begin{enumerate}[start=2,label={{(H\arabic*}$_0$)}]
\item \label{H2_0} \lc{$\alpha=0$, $\beta\in[0,1)$,
$\delta\in(-\frac12, \frac12]$, with 
\[
  \beta+\delta\in
  \begin{cases}
      [0,\frac12] \quad&\text{if } \delta\geq0\,,\\
      [0,\frac12) \quad&\text{if } \delta<0\,.
  \end{cases}
\]}
\item \label{H3_0} There exists $\eta\in (0,1)$ such that $A^{-(1+2\delta)+\eta}\in\cL^1(H,H)$.
\item \label{H4_0} 
$B\in C^\theta_{{\rm loc}}(H;D(A^{-\beta}))$ for some $\theta\in (0,1)$.
\item \label{H5_0} 
In the regime $\beta+\delta=\frac12$, there exists $z_0\in D(A^{-\beta})$ such that 
\[
\sup_{x\in H}\norm{B(x)-z_0}_{D(A^{-\beta})}=:C_{\tilde  B}<
\frac{\theta(1-\theta)(1-\beta)}
{4M_{\beta,\delta,\theta}(2-\theta)}\,,
\]
where $M_{\beta,\delta,\vartheta}$ is the implicit constant 
appearing in Proposition~\ref{stimeHolder}.
\end{enumerate}

\subsection{The Ornstein-Uhlenbeck semigroup}
Let $R_t$ be the (scalar) Ornstein-Uhlenbeck semigroup defined by $R_0\varphi=\varphi$ and
\begin{align}
\label{OU}
(R_t\varphi)(x)=\int_{H}
\varphi(e^{-tA}x+y)\,\mu_t(\d y)\,, \quad x\in H\,, \quad t>0\,,
\quad \varphi\in C^0_b(H)\,,
\end{align}
where, for every $t>0$, $\mu_t$ is the Gaussian measure on $H$
with mean $0$ and covariance operator 
\[
Q_t:=\int_0^te^{-2sA}A^{-2\delta}\,\d s
=\frac12
A^{-1-2\delta}(\Id_H-e^{-2tA})
\]
and $\Id_H$ is the identity operator on $H$.
For every $n\in\N_+$ we set
\[
A_n:=P_nAP_n=AP_n=P_nA\,,
\]
where $P_n$ is the orthogonal projection on $H_n:=\operatorname{span}\{e_1,\ldots,e_n\}$. The operator ${A_n}\in\cL(H_n,H_n)$ is the generator of a uniformly continuous semigroup $e^{-tA_n}$ which fulfills
$e^{-tA_n}=e^{-tA} P_n=P_n e^{-tA}$ for every $t\geq0$.
We underline that $P_n$, $A_n$ and $e^{-tA_n}$ commute for every $n\in\N_+$ 
and $t\geq0$.

Analogously, for every $n\in\N_+$, we define the (scalar) finite-dimensional Ornstein-Uhlenbeck semigroup $R_{n,t}$ by setting
\[
(R_{n,t}\varphi)(x)=\int_{H_n}\varphi(e^{-tA_n}x+y)\,\mu_{n,t}(\d y)
\,,\quad x\in H\,, \quad t>0\,,
\quad \varphi\in C^0_b(H_n)\,,
\]
where, for every $t>0$ and $n\in\N_+$,
$\mu_{n,t}$ is the Gaussian measure on $H_n$ with mean $0$ and covariance operator 
\[
Q_{n,t}:=\int_0^te^{-2sA_n}A_n^{-2\delta}\,\d s=
\frac{1}{2}A_n^{-1-2\delta}(\Id_{H_n}-e^{-2tA_n})\,.
\]
For every $t>0$ and $n\in\N_+$ we set
\[
\Gamma_t:=Q^{-\frac12}_t e^{-tA}, \qquad\Gamma_{n,t}:=Q_{n,t}^{-\frac12}e^{-tA_n}
\]
and we note that
\begin{equation}\label{proiezioni}
Q_{n,t}=Q_{t}P_n=P_nQ_t\,,
\qquad \Gamma_{n,t}=\Gamma_{t}P_n=P_n\Gamma_t\,.
\end{equation}

By hypothesis~\ref{H1} and
the fact that $P_n$  and $A$ commute, standard computations (see \cite[Lemma 3.1]{BOS}) show that for all \lc{$\gamma\in\R$} there exist constants $M_\gamma, M_{\gamma,\delta}>0$ such that,
for every $n\in\N_+$ and $t>0$ we have
\begin{align}
\norm{e^{-tA_n}}_{\mathscr{L}(H_n)}&\leq 1\,,\notag\\
\norm{A_n^\gamma e^{-tA_n}}_{\mathscr{L}(H_n)}
&\leq \frac{M_\gamma}{t^{\max\{0,\gamma\}}}\,,\label{semiS}\\
\label{stimeproiezione}
\norm{A_n^{\gamma}\Gamma_{n,t}}_{\mathscr{L}(H_n)}&\leq 
\frac{M_{\gamma,\delta}}{t^{\frac12+\lc{\max\{0,\gamma+\delta\}}}}\,.
\end{align}
In particular, we stress that for every $\gamma\in\R$ the constants $M_{\gamma,\delta}$ and $M_\gamma$, are independent of $n\in\N_+$.

We recall the main properties of the Ornstein-Uhlenbeck semigroup 
in the following proposition.
\begin{proposition}
    \label{prop:OU}
    Assume \ref{H0}, \ref{H1}, \ref{H2_0}, \ref{H3_0} and let $t>0$. 
    Then, for every $n\in\N_+$ it holds
    \[
    R_{n,t}(C^0_b(H_n))\subseteq C_b^{\infty}(H_n)\,.
    \]
    \lc{Moreover, for every \lc{$\gamma, \gamma_1, \gamma_2\in\R$} there exist 
    $M_{\gamma,\delta}>0$, $M_{\gamma_1, \gamma_2,\delta}>0$, and 
    $M_\gamma>0$
    such that
    for every $n\in\N_+$ it holds
    \begin{alignat*}{2}
        \norm{A_n^\gamma D(R_{n,t}\varphi)}_{C^0_b(H_n; H_n)}
        &\leq \frac{M_{\gamma,\delta}}{t^{\frac12+\lc{\max\{0,\gamma+\delta\}}}}
        \norm{\varphi}_{C^0_b(H_n)}
        \qquad&&\forall\,\varphi\in C^0_b(H_n)\,,\\
        \norm{A_n^\gamma D(R_{n,t}\varphi)}_{C^0_b(H_n; H_n)}
        &\leq \frac{M_{\gamma}}{t^{\max\{0,\gamma\}}}
        \norm{\varphi}_{C^1_b(H_n)}
        \qquad&&\forall\,\varphi\in C^1_b(H_n)\,,\\
        \norm{D(A_n^{\gamma_2}D(R_{n,t}\varphi))A_n^{\gamma_1}}_{C^0_b(H_n; \cL(H_n))}
        &\leq\lc{\frac{M_{\gamma_1,\gamma_2,\delta}}{t^{1+\max\{0,\gamma_1+\delta\}+
        \max\{0,\gamma_2+\delta\}}}\|\varphi\|_{C^0_b(H_n)}
        }
        \qquad&&\forall\,\varphi\in C^0_b(H_n)\,,\\
        \norm{D(A_n^{\gamma_2}D(R_{n,t}\varphi))A_n^{\gamma_1}}_{C^0_b(H_n; \cL(H_n))}
        &\leq\frac{M_{\gamma_1,\gamma_2,\delta}}{t^{\frac12+\lc{\max\{0,\gamma_1+\delta\}
        +\max\{0,\gamma_2\}
        }}}\|\varphi\|_{C^1_b(H_n)}
        \qquad&&\forall\,\varphi\in C^1_b(H_n)\,,\\
        \norm{D(A_n^{\gamma_1}D(R_{n,t}\varphi))A_n^{\gamma_2}}_{C^0_b(H_n; \cL(H_n))}
        &\leq\frac{M_{\gamma_1,\gamma_2,\delta}}{t^{\frac12+\lc{\max\{0,\gamma_1+\delta\}
        +\max\{0,\gamma_2\}
        }}}\|\varphi\|_{C^1_b(H_n)}
        \qquad&&\forall\,\varphi\in C^1_b(H_n)\,.
        \end{alignat*}
        }
\end{proposition}
\begin{proof}
    By \cite[Thm.~6.2.2 \& Prop.~6.2.9]{Dap-Zab2002},
    for every $\varphi\in C^0_b(H_n)$ and $x,v,w,z\in H_n$ it holds that
    \begin{align}
    &D(R_{n,t}\varphi)(x)[v]=\int_{H_n}
    \left(\Gamma_{n,t}v,
    Q_{n,t}^{-\frac12}y\right)_H
    \varphi(e^{-tA_{n}}x+y)\,\mu_{n,t}(\d y)\label{D10} \,,\\
    &D^2(R_{n,t}\varphi)(x)[w,z]\notag\\
    &\qquad=\int_{H_n}
    \left[\left(\Gamma_{n,t}w,Q_{n,t}^{-\frac12}y\right)_H
    \left(\Gamma_{n,t}z,Q_{n,t}^{-\frac12}y\right)_H
    -\left(\Gamma_{n,t}w,\Gamma_{n,t}z\right)_H\right]
    \varphi(e^{-tA_{n}}x+y)\,\mu_{n,t}(\d y)\label{D20},
\end{align}
    if in addition $\varphi\in C^1_b(H_n)$ then it holds that
\begin{align}
    &D(R_{n,t}\varphi)(x)[v]=\int_{H_n}
    \left(D\varphi(e^{-tA_{n}}x+y),
    e^{-tA_n}v\right)_H\,\mu_{n,t}(\d y)\label{D11} \,,\\
    &D^2(R_{n,t}\varphi)(x)[w,z]
    =\int_{H_n}
    \left(\Gamma_{n,t}w,Q_{n,t}^{-\frac12}y\right)_H
    \left(D\varphi(e^{-tA_{n}}x+y),e^{-tA_n}z\right)_H
    \,\mu_{n,t}(\d y)\label{D21}\,.
\end{align}
Noting that
\begin{alignat*}{2}
    A_n^\gamma D(R_{n,t}\varphi)(x)[h]&=
    D(R_{n,t}\varphi)(x)[A_n^\gamma h] \quad&&\forall\,h\in H_n\,,\\
   \left(D(A_n^{\gamma_2} D(R_{n,t}\varphi))(x)A_n^{\gamma_1}h_1,h_2\right)_H
   &=\lc{D^2(R_{n,t}\varphi)(x)[A_n^{\gamma_1}h_1, A_n^{\gamma_2}h_2] }\quad&&\forall\,h_1,h_2\in H_n\,,
\end{alignat*}
exploiting \eqref{proiezioni}, recalling that $A_n$, $P_n$, and $e^{-tA_{n}}$ commute, and
choosing $v=A_n^{\gamma}h$ 
in \eqref{D10} and \eqref{D11}, the first two estimates follow by the H\"older inequality, 
\cite[Sec.~1.2.4]{Dap-Zab2002} and
\eqref{stimeproiezione}. \lc{Moreover, by choosing
$w=A_n^{\gamma_1}h_1$ and $z=A_n^{\gamma_2}h_2$ for $h_1,h_2\in H_n$ with $\|h_1\|_{H_n}
=\|h_2\|_{H_n}=1$ in \eqref{D20} and \eqref{D21}, 
we obtain by the H\"older inequality,
\cite[Sec.~1.2.4, Prop.~1.2.4]{Dap-Zab2002}, and
\eqref{stimeproiezione}, that
\begin{align*}
    &|D^2(R_{n,t}\varphi)(x)[A_n^{\gamma_1}h_1,A_n^{\gamma_2}h_2]|\\
    &\leq\norm{\varphi}_{C^0_b(H_n)}
    \left(\int_{H_n}
    \left(\Gamma_{n,t}A_n^{\gamma_1}h_1,Q_{n,t}^{-\frac12}y\right)_H^2
    \,\mu_{n,t}(\d y)
    \right)^{\frac12}
    \left(\int_{H_n}
    \left(\Gamma_{n,t}A_n^{\gamma_2}h_2,Q_{n,t}^{-\frac12}y\right)_H^2
    \,\mu_{n,t}(\d y)
    \right)^{\frac12}\\
    &\qquad+\norm{\varphi}_{C^0_b(H_n)}
    \norm{\Gamma_{n,t}A_n^{\gamma_1}h_1}_{H_n}
    \norm{\Gamma_{n,t}A_n^{\gamma_2}h_2}_{H_n}\\
    &= 2\norm{\varphi}_{C^0_b(H_n)}
    \norm{\Gamma_{n,t}A_n^{\gamma_1}h_1}_{H_n}
    \norm{\Gamma_{n,t}A_n^{\gamma_2}h_2}_{H_n}
    \leq 2\norm{\varphi}_{C^0_b(H_n)}
    \frac{M_{\gamma_1,\delta}M_{\gamma_2,\delta}}{t^{1+\max\{0,\gamma_1+\delta\}+\max\{0,\gamma_2+\delta\}}}
\end{align*}
and 
\begin{align*}
    |D^2(R_{n,t}\varphi)(x)
    [A_n^{\gamma_1}h_1,A_n^{\gamma_2}h_2]|^2
    &\leq \norm{\varphi}_{C^1_b(H_n)}^2
    \norm{e^{-tA_n}A_n^{\gamma_2}h_2}_{H_n}^2
    \int_{H_n}
    \left(\Gamma_{n,t}A_n^{\gamma_1}h_1,Q_{n,t}^{-\frac12}y\right)_H^2
    \,\mu_{n,t}(\d y)\\
    &\leq \norm{\varphi}_{C^1_b(H_n)}^2
    \frac{M_{\gamma_2}^2}{t^{2\max\{0,\gamma_2\}}}
    \frac{M_{\gamma_1,\delta}^2}{t^{1+2\max\{0,\gamma_1+\delta\}}}.
\end{align*}
This proves the third and fourth estimates. The fifth estimate follows by symmetry of $D^2(R_{n,t}\varphi)(x)$
by exchanging the role of $A_n^{\gamma_1}h_1$ and $A_n^{\gamma_2}h_2$.}
\end{proof}

By using Proposition \ref{prop:OU}, standard interpolation arguments (see \cite[Thm.~1.12 of Ch.~5]{CR88}, and \cite[Prop.~2.3.3]{Dap-Zab2002}), we obtain the following result.

\begin{proposition}\label{stimeHolder}
Assume \ref{H0}, \ref{H1}, \ref{H2_0}, \ref{H3_0}, let $t>0$ and $\vartheta\in(0,1)$.
Then for every \lc{$\gamma, \gamma_1, \gamma_2\in\R$} there exists
    $M_{\gamma,\delta, \vartheta}>0$
    and $M_{\gamma_1,\gamma_2,\delta, \vartheta}>0$
    such that
    for every $n\in\N_+$ and 
    $\varphi\in C_b^\vartheta(H_n)$ it holds that
\begin{align}
\|A_n^\gamma D(R_{n,t}\varphi)\|_{C^0_b(H_n;H_n)}&\leq \frac{M_{\gamma,\delta,\vartheta}}{t^{(1-\vartheta)(\frac12+\max\{0,\gamma+\delta\})
+\vartheta\max\{0,\gamma\}}}
\|\varphi\|_{C^\vartheta_b(H_n)}\,,\label{Hstima1var}\\
\|A_n^\gamma D(R_{n,t}\varphi)\|_{C^\vartheta_b(H_n;H_n)}&\leq  
\frac{M_{\gamma,\delta,\vartheta}}{t^{\frac12+\max\{0,\gamma+\delta\}}}
\|\varphi\|_{C^\vartheta_b(H_n)}\,,\label{Hstima1.5var}\\
\|D(A_n^{\gamma_2}D(R_{n,t}\varphi))A_n^{\gamma_1}
\|_{C^0_b(H_n;\mathscr{L}(H_n))}
&\leq 
\lc{\frac{M_{\gamma_1,\gamma_2,\delta,\vartheta}}{t^{1-\frac{\vartheta}{2}+
\max\{0,\gamma_1+\delta\}+
(1-\vartheta)\max\{0,\gamma_2+\delta\}
+\vartheta\max\{0,\gamma_2\}
}}}\|\varphi\|_{C^\vartheta_b(H_n)}\,,\label{Hstima2var}\\
\|D(A_n^{\gamma_1}D(R_{n,t}\varphi))A_n^{\gamma_2}
\|_{C^0_b(H_n;\mathscr{L}(H_n))}
&\leq 
\lc{\frac{M_{\gamma_1,\gamma_2,\delta,\vartheta}}{t^{1-\frac{\vartheta}{2}+
\max\{0,\gamma_1+\delta\}+
(1-\vartheta)\max\{0,\gamma_2+\delta\}
+\vartheta\max\{0,\gamma_2\}
}}}\|\varphi\|_{C^\vartheta_b(H_n)}\,.\label{Hstima2var_bis}
\end{align}
\end{proposition}

\begin{rmk} 
Noting that $A_n$ is bijective for every $n\in\N$, it is immediate to see that for every $n\in\N_+$
there exists a constant $M_n>0$ such that, for all $t>0$,
\begin{equation*}
\norm{\Gamma_{n,t}}_{\mathscr{L}(H_n)}
\leq  M_n\frac{1}{t^{\frac12}}\,.
\end{equation*}
Hence, by \cite[Thm.~6.2.2 \& Prop.~6.2.9]{Dap-Zab2002}
and standard interpolation arguments 
for every $n\in\N_+$ and $\vartheta\in (0,1)$ there exists a constant $M_{n,\vartheta}>0$ such that for every $\varphi\in C^\vartheta_b(H_n)$ and $t>0$ we have 
\begin{equation}\label{Sn}
\norm{R_{n,t}\varphi}_{C^2_b(H_n)}\leq 
M_{n,\vartheta}\left(1+\frac{1}{t^{1-\frac{\vartheta}{2}}}\right)\norm{\varphi}_{C^\vartheta_b(H_n)}\,.
\end{equation}
\end{rmk}

\subsection{The elliptic Kolmogorov equation}
We assume here \ref{H0}, \ref{H1}, \ref{H2_0}--\ref{H5_0}, as well as the following 
additional boundedness assumption.
\begin{enumerate}[start=6,label={{(H\arabic*}$_0$)}]
\item \label{H6_0}
$B\in C^\theta_b(H;D(A^{-\beta}))$ and set 
$C_B:=\max\{1,\|B\|_{C^\theta_b(H;D(A^{-\beta}))}\}$.
\end{enumerate}
The additional assumptions \ref{H5_0}--\ref{H6_0} are crucial to solve the Kolmogorov equation,
but will be suitably weakened in the general uniqueness results: see Theorems~\ref{th:1}, \ref{th:2}, and \ref{th:3}.

For every $n\in\N_+$ we define $B_n:H\to H_n$ by
\[
B_n(x):=P_nB(x):=\sum_{k=1}^n\langle B(x), e_k\rangle e_k\,, \quad x\in H\,,
\]
and we note that thanks to \ref{H6_0} one has, for every $n\in\N_+$,
\begin{align}\label{est:Bn2}
\norm{B_{n|H_n}}_{C^0_b(H_n;D(A_n^{-\beta}))}\leq \norm{B_{n|H_n}}_{C^\theta_b(H_n;D(A_n^{-\beta}))}
\leq \norm{B}_{C^\theta_b(H;D(A^{-\beta}))}&\leq C_B\,.
\end{align}
Indeed, using the fact that $P_n$ is non-expansive on $D(A^{-\beta})$ and $P_n$ and $A$ commute, 
for all $x,y\in H_n$ one has
\begin{align*}
\norm{B_n(x)}_{D(A_n^{-\beta})}&=
\norm{B_n(x)}_{-\beta}
=\sup_{\|z\|_\beta\leq1}
|\langle P_nB(x), z\rangle|
\leq\norm{B(x)}_{-\beta}
\end{align*}
and
\begin{align*}
\norm{B_n(x)-B_n(y)}_{D(A_n^{-\beta})}&=
\norm{B_n(x)-B_n(y)}_{-\beta}
=\sup_{\|z\|_\beta\leq1}
|\langle P_n(B(x)-B(y)), z\rangle|\leq\norm{B(x)-B(y)}_{-\beta}
\,.
\end{align*}

We denote now by $\bar c$ a generic positive constant satisfying the following constraints:
\begin{itemize}
\item if $\beta+\delta\in[0,\frac12)$, then
\begin{equation}
\label{c}
\bar c>\max\left\{
\frac{4C_B}{\lambda_1},
\frac1{\lambda_1}\left[
4M_{\beta,\delta,\theta}C_B 
\Gamma\left(\frac12-\beta-\delta\right)
\right]^{\frac1{\frac12-\beta-\delta}}
\right\}\,;
\end{equation}
\item if $\beta+\delta=\frac12$, then
\begin{equation}
\begin{aligned}
\label{c'}
\bar c>
    &\frac1{\lambda_1}\left[
    16M_{\beta,\delta,\theta}\Gamma(\theta(1-\beta))(C_{\tilde B}+C_B) \left( 1 + 
  \frac{2M_{\beta,\delta,\theta}(2-\theta)}{\theta(1-\theta)(1-\beta)}(C_{\tilde B}+C_B)
      \right)\right]^{\frac1{\theta(1-\beta)}}
      +\frac{8C_B}{\lambda_1}\,;
\end{aligned}
\end{equation}
\end{itemize}
where $\Gamma$ is the Eulero Gamma function.
\begin{rmk}
    Conditions \eqref{c}--\eqref{c'} on $\bar c$ will be used in order 
    to prove well-posedness of the Kolmogorov equation and deduce 
    suitable estimates on the solutions. In the proof of pathwise uniqueness only (see Subsection~\ref{ssec:path}), 
    the possible values of $\bar c$ will be further restricted. 
\end{rmk}

Let $n\in\N_+$, $k\in\{1,\ldots,n\}$, and $g\in C^\theta_b(H)$. We consider the finite-dimensional Kolmogorov equation 
\begin{equation}
    \label{eq:kolm}
    \bar c\lambda_k u_{n,k}(x)
    +L_n u_{n,k}(x)
    =\langle B_n(x), Du_{n,k}(x)\rangle +g(x)\,, \quad x\in H_n\,,
\end{equation}
where
\[
L_n\varphi(x):=
-\frac{1}{2}\Tr\left[A_n^{-2\delta}D^2\varphi(x)\right]
+\scal{A_nx}{D\varphi(x)}\,, \quad x\in H_n\,, \quad \varphi\in C^2_b(H_n)\,.
\]
Let us note that $g_{|H_n}\in C^\theta_b(H_n)$ and 
\begin{equation}
\label{est:g}
\|g_{|H_n}\|_{C^\theta_b(H_n)}\leq\|g\|_{C^\theta_b(H)}\,.
\end{equation}

\begin{proposition}
\label{prop:kolm}
Assume \ref{H0}, \ref{H1}, \ref{H2_0}--\ref{H6_0}
and let $n\in\N_+$, $k\in\{1,\ldots,n\}$ and $g\in C^\theta_b(H)$.
Then, there exists a unique $u_{n,k}\in C^{1+\theta}_b(H_n)$ such that 
\begin{equation}
\label{eq:kolm_mild}
u_{n,k}(x)=
\int_0^{+\infty} 
e^{-\bar c\lambda_k t}
R_{n,t}\left[\scal{B_n}{Du_{n,k}} + g_{|H_n}\right](x)\,\d t
\qquad\forall\,x\in H_n\,.
\end{equation}
Moreover, it holds that $u_{n,k}\in C^2_b(H_n)$ and satisfies the Kolmogorov equation \eqref{eq:kolm}.
\end{proposition}
\begin{proof}
We use a fixed point argument on the operator $\mathcal V_{n,k}: C^{1+\theta}_b(H_n)\to C^{2}_b(H_n)$
defined as
\[
\mathcal V_{n,k} \varphi(x):=
\int_0^{+\infty}e^{-\bar c\lambda_k t} 
R_{n,t}\left[\scal{B_n}{D\varphi} + g_{|H_n}\right](x)\,\d t\,,
\quad x\in H_n\,, \quad \varphi\in C^{1+\theta}_b(H_n)\,.
\]
Note that the operator $\mathcal V_{n,k}$ is well-defined: indeed, for every $\varphi\in C^{1+\theta}_b(H_n)$,
from the estimates \eqref{Sn}, \eqref{est:Bn2}
and \eqref{est:g}, and recalling that $A_n^\beta\in\cL(H_n)$, we have
\begin{align*}
\norm{\mathcal V_{n,k} \varphi}_{C^2_b(H_n)}&\leq 
M_{n,\theta}
\left(C_B\norm{A_n^\beta D\varphi}_{C^\theta_b(H_n;H_n)}
+\|g\|_{C^\theta_b(H)}\right)
\int_0^{+\infty} e^{-\bar c\lambda_k t}
\left(1+\frac{1}{t^{1-\frac{\theta}{2}}}\right)\,\d t
<+\infty\,.  
\end{align*}
Let us show now that the operator $\mathcal V_{n,k}$ is a contraction on $C^{1+\theta}_b(H_n)$. To this end, let $\varphi,\psi\in C^{1+\theta}_b(H_n)$. We preliminary note that by proceeding as above and noting that the forcing terms $g_{|H_n}$ compensate when taking the difference, from \eqref{est:Bn2} one has
\begin{align}
\nonumber
\norm{\mathcal V_{n,k}\varphi - \mathcal V_{n,k}\psi}_{C^0_b(H_n)}
&\leq \int_0^{\infty}e^{-\bar c\lambda_k t}
\norm{B_n}_{C^0_b(H_n; D(A_n^{-\beta}))}
\norm{D(\varphi-\psi)}_{C^0_b(H_n; D(A_n^{\beta}))}\,\d t  \\
\nonumber
&\leq \frac1{\bar c\lambda_k}C_B
\norm{A_n^\beta D(\varphi-\psi)}_{C^0_b(H_n;H_n)}\\
\label{V_aux1}
&\leq \frac1{\bar c\lambda_1}C_B
\norm{A_n^\beta D(\varphi-\psi)}_{C^\theta_b(H_n;H_n)}\,,
\end{align}
where in the last inequality we have used \ref{H1}, specifically $\lambda_k>\lambda_1>0$
for every $k>1$.
Now, we divide the proof into the two different cases $\beta+\delta\in[0,\frac12)$ and $\beta+\delta=\frac12$.\\
{\textbf Case \boldsymbol{$\beta+\delta\in[0,\frac12)$}.}
By using \eqref{Hstima1.5var} with $\gamma=\beta$ one gets via similar computations as above that 
\begin{align}
&\norm{A_n^\beta D(\mathcal V_{n,k}\varphi - \mathcal V_{n,k}\psi)}_{C^\theta_b(H_n; H_n)}
\leq M_{\beta,\delta,\theta}
\int_0^{+\infty}e^{-\bar c\lambda_k t}
\frac1{t^{\frac12+\beta+\delta}}
\norm{\scal{B_n}{D(\varphi-\psi)}}_{C^\theta_b(H_n)}\,\d t\notag  \\
&\qquad\leq \frac1{(\bar c\lambda_k)^{\frac12-\beta-\delta}}M_{\beta,\delta,\theta}
\Gamma\left(\frac12-\beta-\delta\right)
C_B
\norm{D(\varphi-\psi)}_{C^\theta_b(H_n;D(A_n^\beta))}\notag
\\
&\qquad\leq \frac1{\bar c^{\frac12-\beta-\delta}}
\left[\frac{M_{\beta,\delta,\theta}C_B}{\lambda_1^{\frac12-\beta-\delta}}
\Gamma\left(\frac12-\beta-\delta\right)\right]
\norm{A_n^\beta D(\varphi-\psi)}_{C^\theta_b(H_n;H_n)}\,\label{V_aux0}.
\end{align}
By condition \eqref{c}, \eqref{V_aux1} and \eqref{V_aux0} we infer that 
\begin{align*}
&\norm{\mathcal V_{n,k}\varphi - \mathcal V_{n,k}\psi}_{C^0_b(H_n)}
+\norm{A_n^\beta D(\mathcal V_{n,k}\varphi - \mathcal V_{n,k}\psi)}_{C^\theta_b(H_n; H_n)}\\
&\qquad\leq 
\left[\frac1{\bar c}\frac{C_B}{\lambda_1}+
\frac{1}{\bar c^{\frac12-\beta-\delta}}
\frac{M_{\beta,\delta,\theta}}{\lambda_1^{\frac12-\beta-\delta}}
\Gamma\left(\frac12-\beta-\delta\right)
\right]
\norm{A_n^\beta D(\varphi-\psi)}_{C^\theta_b(H_n;H_n)}
\\
&\qquad\leq\frac12\left[
\norm{\varphi - \psi}_{C^0_b(H_n)}+
\norm{A_n^\beta D(\varphi-\psi)}_{C^\theta_b(H_n;H_n)}
\right]\,,
\end{align*}
so that $\mathcal V_{n,k}$ is a contraction on $C_b^{1+\theta}(H_n)$ with respect to the equivalent norm
\[
\varphi\mapsto \norm{\varphi}_{C^0_b(H_n)} + 
\norm{A_n^\beta D\varphi}_{C^\theta_b(H_n; H_n)}\,.
\]
{\textbf Case \boldsymbol{$\beta+\delta=\frac12$}.}
Recall that $\delta\geq0$ by \ref{H2_0}.
Let us define $\tilde B:H\to D(A^{-\beta})$ as
$\tilde B(x):=B(x)-z_0$, $x\in H$, where $z_0$ is given by assumption \ref{H5_0}. Analogously, for every $n\in\N_+$ we define $\tilde B_n:H\to H_n$ as $\tilde B_n(x):=P_n\tilde B(x)=B_n(x)-P_n z_0$, 
$x\in H$. We note that \ref{H5_0} guarantees that, for all $n\in\N_+$,
\begin{align}
\label{est_Bn_tilde1}
\|\tilde B_n\|_{C^0_b(H;D(A_{n}^{-\beta}))}\leq
\|\tilde B\|_{C^0_b(H;D(A^{-\beta}))}&=C_{\tilde B}\,,\\
\label{est_Bn_tilde2}
\|\tilde B_n\|_{C^\theta_b(H;D(A_{n}^{-\beta}))}\leq
\|\tilde B\|_{C^\theta_b(H;D(A^{-\beta}))}&\leq
C_{\tilde B}+ C_B\,.
\end{align}
Now, note that one can formally rewrite the Kolmogorov equation 
\eqref{eq:kolm} as 
\[
\bar c\lambda_k u_{n,k}(x)+L_n^{z_0} u_{n,k}(x)
=\langle \tilde B_n(x), Du_{n,k}(x)\rangle +g(x)\,, \quad x\in H_n\,,
\]
where
\[
  L_n^{z_0}\varphi(x):=
  -\frac{1}{2}\Tr\left[A_n^{-2\delta}D^2\varphi(x)\right]
  +\scal{A_nx-z_0}{D\varphi(x)}\,, \quad x\in H_n\,, \quad \varphi\in C^2(H_n)\,.
\]
Hence, we can equivalently represent the operator $\mathcal V_{n,k}$
as
\begin{equation}\label{V_aux2}
  \mathcal V_{n,k} \varphi(x):=
  \int_0^{+\infty}e^{-\bar c\lambda_k t} 
  R^{z_0}_{n,t}\left[\langle\tilde B_n,D\varphi\rangle
  + g_{|H_n}\right](x)\,\d t\,,
  \quad x\in H_n\,, \quad \varphi\in C^{1+\theta}_b(H_n)\,,
\end{equation}
with the semigroup $R_{n,t}^{z_0}$ being defined as
$R_{n,0}^{z_0}\varphi=\varphi$ and 
\[
(R^{z_0}_{n,t}\varphi)(x):=
\int_{H_n}\varphi(e^{-tA_n}x+\zeta_t+y)\,\mu_{n,t}(\d y)
\,,\quad x\in H\,, \quad t>0\,,
\quad \varphi\in C^0_b(H_n)\,,
\]
where
\[
\zeta_t:=\int_0^te^{-(t-s)A_n}z_0\, \d s, \qquad t>0\,.
\]
It is possible to prove that $R_{n,t}^{z_0}$ verifies the estimates \eqref{Hstima1var}--\eqref{Hstima2var} (see \cite{CerLun2021})
by possibly updating the implicit constants with dependence on $z_0$ but not on $n$. In particular,
for every $n\in\N_+$, $t>0$ and $\varphi\in C_b^\theta(H_n)$, it holds that (recall that $\beta+\delta=\frac12$)
\begin{align}
\label{V_aux3}
&\|A_n^\beta D(R^{z_0}_{n,t}\varphi)\|_{C^0_b(H_n;H_n)}
\leq M_{\beta,\delta,\theta}\frac{1}{t^{(1-\theta)(\frac12+\beta+\delta)+\theta\beta}}
\|\varphi\|_{C^\theta_b(H_n)}
= M_{\beta,\delta,\theta}\frac{1}{t^{1-\theta(1-\beta)}}
\|\varphi\|_{C^\theta_b(H_n)}\,,\\
\label{V_aux4}
&\|D(A_n^{\beta}D(R_{n,t}^{z_0}\varphi))
\|_{C^0_b(H_n;\mathscr{L}(H_n))}
\leq M_{\beta,0,\delta,\theta}
\frac{1}{t^{\frac32-\frac{\theta}{2}+(1-\theta)\delta}}\|\varphi\|_{C^\theta_b(H_n)}\,.
\end{align}
Note that $1-\theta(1-\beta)<1$
and $\frac32-\frac{\theta}{2}+(1-\theta)\delta>1$
as a consequence of assumption \ref{H2_0}.

By  \eqref{est_Bn_tilde1}, \eqref{est_Bn_tilde2}, \eqref{V_aux2} and \eqref{V_aux3} we get
\begin{align}
&\norm{A_n^\beta D(\mathcal V_{n,k}\varphi - \mathcal V_{n,k}\psi)}_{C^0_b(H_n; H_n)} \notag \\
&\leq M_{\beta,\delta,\theta}
\norm{\langle\tilde B_n,D(\varphi-\psi)\rangle}_{C^\theta_b(H_n)}
\int_0^{+\infty}e^{-\bar c\lambda_k t}
\frac{1}{t^{1-\theta(1-\beta)}}\,\d t \notag \\
&\leq \frac{M_{\beta,\delta,\theta}\Gamma(\theta(1-\beta))}
{(\bar c\lambda_k)^{\theta(1-\beta)}}
\left[\|\tilde B_n\|_{C^0_b(H_n; D(A_{n}^{-\beta}))}
\norm{D(\varphi-\psi)}_{C^\theta_b(H_n; D(A_{n}^\beta))}\right. \notag \\
&\qquad\left.+
\|\tilde B_n\|_{C^\theta_b(H_n; D(A_{n}^{-\beta}))}
\norm{D(\varphi-\psi)}_{C^0_b(H_n; D(A_{n}^\beta))}
\right] \notag \\
&\leq \frac{M_{\beta,\delta,\theta}\Gamma(\theta(1-\beta))}
{(\bar c\lambda_1)^{\theta(1-\beta)}}
\left[C_{\tilde B}\norm{A_n^\beta D(\varphi-\psi)}_{C^\theta_b(H_n; H_n)}\right. \notag \\
&\qquad\left.+
(C_{\tilde B}+C_B)
\norm{A_n^\beta D(\varphi-\psi)}_{C^0_b(H_n; H_n)}
\right] \notag \\
&\leq\frac1{\bar c^{\theta(1-\beta)}}
\left[\frac{2M_{\beta,\delta,\theta}
\Gamma(\theta(1-\beta))(C_{\tilde B}+C_B)}{\lambda_1^{\theta(1-\beta)}}
\right]\norm{A_n^\beta D(\varphi-\psi)}_{C^\theta_b(H_n; H_n)}\,.
\label{normaCb}
\end{align}
Let us estimate now the $\theta$-H\"older seminorm of $A_n^\beta D(\mathcal V_{n,k}\varphi-\mathcal V_{n,k}\psi)$: to this end, we separately consider the two cases $\norm{h}_{H_n}\geq 1$ and $\|h\|_{H_n}\leq 1$ with $h\neq0$.\\
The case $\norm{h}_{H_n}\geq 1$ is a direct consequence of \eqref{V_aux3}: for all $x,h\in H_n$ with $\norm{h}_{H_n}\geq 1$, by \eqref{V_aux3} we have
\begin{align*}
\nonumber
&\|A_n^\beta D(\mathcal V_{n,k}\varphi-\mathcal V_{n,k}\psi)(x+h)
-A_n^\beta D(\mathcal V_{n,k}\varphi-\mathcal V_{n,k}\psi)(x)\|_{H_n}\\
\nonumber
&\leq\int_0^{+\infty}e^{-\bar c\lambda_k t}
\norm{A_n^\beta D R_{n,t}^{z_0}
\left[\langle\tilde B_n,D(\varphi-\psi)\rangle \right](x+h)-  A_n^\beta D R_{n,t}^{z_0}
\left[\langle\tilde B_n,D(\varphi-\psi)\rangle \right](x)}_{H_n}\,\d t\\
\nonumber
&\leq 2 M_{\beta,\delta,\theta}
\norm{\langle\tilde B_n,D(\varphi-\psi)\rangle}_{C^\theta_b(H_n)}
\int_0^{+\infty}e^{-\bar c\lambda_k t}
\frac{1}{t^{1-\theta(1-\beta)}}\,\d t\\
&\leq \frac{2M_{\beta,\delta,\theta}\Gamma(\theta(1-\beta))}
{(\bar c\lambda_k)^{\theta(1-\beta)}}
\|\langle \tilde B_n, D(\varphi-\psi)\rangle\|_{C^\theta_b(H_n)} \\
&\leq \frac{2M_{\beta,\delta,\theta}\Gamma(\theta(1-\beta)) (C_{\tilde B}+C_B)}
{(\bar c\lambda_1)^{\theta(1-\beta)}}
\|\langle \tilde B_n, D(\varphi-\psi)\rangle\|_{C^\theta_b(H_n)}
\,.
\end{align*}
Since $\norm{h}_{H_n}\geq 1$, we infer that
\begin{align}
\nonumber
&\|A_n^\beta D(\mathcal V_{n,k}\varphi-\mathcal V_{n,k}\psi)(x+h)
-A_n^\beta D(\mathcal V_{n,k}\varphi-\mathcal V_{n,k}\psi)(x)\|_{H_n}\\
&\leq \frac{2M_{\beta,\delta,\theta}\Gamma(\theta(1-\beta)) (C_{\tilde B}+C_B)}
{(\bar c\lambda_1)^{\theta(1-\beta)}}
\|\langle \tilde B_n, D(\varphi-\psi)\rangle\|_{C^\theta_b(H_n)}
\label{normCt>1}.
\end{align}
We now consider the case $h\neq0$ and $\|h\|_{H_n}\leq 1$. We set $\varsigma:=\frac1{1-\beta}$.
For all $x,h\in H_n$ with $h\neq0$ and $\|h\|_{H_n}\leq 1$,
we have, thanks to \eqref{V_aux3}--\eqref{V_aux4},
\begin{align}
\nonumber
&\|A_n^\beta D(\mathcal V_{n,k}\varphi-\mathcal V_{n,k}\psi)(x+h)
-A_n^\beta D(\mathcal V_{n,k}\varphi-\mathcal V_{n,k}\psi)(x)\|_{H_n}\\
\nonumber
&\leq\int_0^{+\infty}e^{-\bar c\lambda_k t}
\norm{A_n^\beta D R_{n,t}^{z_0}
\left[\langle\tilde B_n,D(\varphi-\psi)\rangle \right](x+h)-  A_n^\beta D R_{n,t}^{z_0}
\left[\langle\tilde B_n,D(\varphi-\psi)\rangle \right](x)}_{H_n}\,\d t\\
\nonumber
&=\int_{0}^{\|h\|_{H_n}^{\varsigma}}
e^{-\bar c\lambda_k t}
\norm{A_n^\beta
D R_{n,t}^{z_0}
\left[\langle\tilde B_n,D(\varphi-\psi)\rangle \right](x+h)-A_n^\beta D R_{n,t}^{z_0}
\left[\langle\tilde B_n,D(\varphi-\psi)\rangle \right](x)}_{H_n}\,\d t\\
\nonumber
&\quad+\int_{\|h\|_{H_n}^{\varsigma}}^{+\infty}
e^{-\bar c\lambda_k t}\norm{A_n^\beta
D R_{n,t}^{z_0}
\left[\langle\tilde B_n,D(\varphi-\psi)\rangle \right](x+h)-A_n^\beta D R_{n,t}^{z_0}
\left[\langle\tilde B_n,D(\varphi-\psi)\rangle \right](x)}_{H_n}\,\d t\\
\nonumber
&\leq M_{\beta,\delta,\theta}
\|\langle \tilde B_n, D(\varphi-\psi)\rangle\|_{C^\theta_b(H_n)}
\left[2\int_{0}^{\|h\|_{H_n}^{\varsigma}}
\frac{1}{t^{1-\theta(1-\beta)}}\,\d t
+\|h\|_{H_n}\int_{\|h\|_{H_n}^{\varsigma}}^{+\infty}
\frac{1}{t^{\frac32-\frac{\theta}{2}+(1-\theta)\delta}}\,\d t
\right]\\
\label{exp:h}
&=M_{\beta,\delta,\theta}
\|\langle \tilde B_n, D(\varphi-\psi)\rangle\|_{C^\theta_b(H_n)}
\left[\frac{2\|h\|_{H_n}^\theta}
{\theta(1-\beta)}
+\frac{
\|h\|_{H_n}^{
1-\frac{\frac12-\frac\theta2+(1-\theta)\delta}
{1-\beta}
}}{\frac12-\frac\theta2+(1-\theta)\delta}
\right]\,.
\end{align}
Now, by taking into account that $\beta+\delta=\frac12$,
one has $1-\frac{\frac12-\frac\theta2+(1-\theta)\delta}
{1-\beta}=\theta$, so that 
\begin{align}
\nonumber
&\|A_n^\beta D(\mathcal V_{n,k}\varphi-\mathcal V_{n,k}\psi)(x+h)
-A_n^\beta D(\mathcal V_{n,k}\varphi-\mathcal V_{n,k}\psi)(x)\|_{H_n}\\
\label{normCt<1}
&\leq M_{\beta,\delta,\theta}
\|\langle \tilde B_n, D(\varphi-\psi)\rangle\|_{C^\theta_b(H_n)}
\left[\frac{2}
{\theta(1-\beta)}
+\frac{1}
{\frac12-\frac\theta2+(1-\theta)\delta}
\right]\|h\|_{H_n}^\theta\,.
\end{align}
Taking into account \eqref{normCt>1} and \eqref{normCt<1}, and noting that by \eqref{c'} we can assume $\overline{c}$ arbitrarily large in  \eqref{normCt>1}, we obtain
\begin{align}
\nonumber
&[A_n^\beta D(\mathcal V_{n,k}\varphi-
\mathcal V_{n,k}\psi)]_{C^\theta_b(H_n;H_n)}\\
\nonumber
&\leq M_{\beta,\delta,\theta}
\frac{2-\theta}{\theta(1-\theta)(1-\beta)}
\|\langle \tilde B_n, D(\varphi-\psi)\rangle\|_{C^\theta_b(H_n)}\\
\label{normaCt}
&\leq
\frac{M_{\beta,\delta,\theta}(2-\theta)}{\theta(1-\theta)(1-\beta)}
\left[C_{\tilde B}\norm{A_n^\beta D(\varphi-\psi)}_{C^\theta_b(H_n; H_n)}+
(C_{\tilde B}+C_B)\norm{A_n^\beta D(\varphi-\psi)}_{C^0_b(H_n; H_n)}\right]\,.
\end{align}
By multiplying \eqref{normaCb} by the constant 
$1+\frac{2M_{\beta,\delta,\theta}(2-\theta)}{\theta(1-\theta)(1-\beta)}(C_{\tilde B}+C_B)$, by summing up with \eqref{normaCt} and
by defining the constants
\begin{alignat*}{2}
&L_{\beta,\delta,\theta}
:=\frac{2M_{\beta,\delta,\theta}(2-\theta)}{\theta(1-\theta)(1-\beta)}(C_{\tilde B}+C_B)\,,
\qquad
&&L'_{\beta,\delta,\theta}
:=\frac{2M_{\beta,\delta,\theta}}{\lambda_1^{\theta(1-\beta)}}
\Gamma(\theta(1-\beta))(C_{\tilde B}+C_B)(1+L_{\beta,\delta,\theta})\,,\\
&L''_{\beta,\delta,\theta}:=
\frac{M_{\beta,\delta,\theta}(2-\theta)}{\theta(1-\theta)(1-\beta)}\,,
\qquad
&&L'''_{\beta,\delta,\theta}:=
\frac{C_B}{\lambda_1}\,,
\end{alignat*}
we obtain 
\begin{align*}
&L_{\beta,\delta,\theta}
\norm{A_n^\beta D(\mathcal V_{n,k}\varphi - \mathcal V_{n,k}\psi)}_{C^0_b(H_n; H_n)}
+\|A_n^\beta D(\mathcal V_{n,k}\varphi-
\mathcal V_{n,k}\psi)\|_{C^\theta_b(H_n;H_n)}\\
&\leq
\frac{L_{\beta,\delta,\theta}'}{\bar c^{\theta(1-\beta)}}
\|A^\beta_n D(\varphi-\psi)\|_{C^\theta_b(H_n;H_n)}
+C_{\tilde B}L_{\beta,\delta,\theta}''
\|A^\beta_n D(\varphi-\psi)\|_{C^\theta_b(H_n;H_n)}
\\
&+\frac{L_{\beta,\delta,\theta}}{2}
\norm{A_n^\beta D(\varphi-\psi)}_{C^0_b(H_n; H_n)}\,.
\end{align*}
Together with \eqref{V_aux1}, this implies
that 
\begin{align*}
&\norm{\mathcal V_{n,k}\varphi - \mathcal V_{n,k}\psi}_{C^0_b(H_n)}+L_{\beta,\delta,\theta}
\norm{A_n^\beta D(\mathcal V_{n,k}\varphi - \mathcal V_{n,k}\psi)}_{C^0_b(H_n; H_n)}
+\|A_n^\beta D(\mathcal V_{n,k}\varphi-
\mathcal V_{n,k}\psi)\|_{C^\theta_b(H_n;H_n)}\\
&\leq \frac{L_{\beta,\delta,\theta}}{2}\norm{A_n^\beta D(\varphi-\psi)}_{C^0_b(H_n; H_n)}+
\left(\frac{L'''_{\beta,\delta,\theta}}{\bar c}+
\frac{L'_{\beta,\delta,\theta}}{\bar c^{\theta(1-\beta)}}+C_{\tilde B}L''_{\beta,\delta,\theta}
\right)\norm{A_n^\beta D(\varphi-\psi)}_{C^\theta_b(H_n;H_n)}\,.
\end{align*}
By recalling assumption \ref{H5_0} on $C_{\tilde B}$
and the condition \eqref{c'}, we infer that 
  \begin{align*}
    &\norm{\mathcal V_{n,k}\varphi - \mathcal V_{n,k}\psi}_{C^0_b(H_n)}
    +L_{\beta,\delta,\theta}
    \norm{A_n^\beta D(\mathcal V_{n,k}\varphi - \mathcal V_{n,k}\psi)}_{C^0_b(H_n; H_n)}
    +\|A_n^\beta D(\mathcal V_{n,k}\varphi-
    \mathcal V_{n,k}\psi)\|_{C^\theta_b(H_n;H_n)}\\
    &\leq \frac12
    \left[\|\varphi-\psi\|_{C^0_b(H_n)}+
    L_{\beta,\delta,\theta}
    \norm{A_n^\beta D(\varphi-\psi)}_{C^0_b(H_n; H_n)}+
    \norm{A_n^\beta D(\varphi-\psi)}_{C^\theta_b(H_n;H_n)}
    \right]\,,
\end{align*}
so that $\mathcal V_{n,k}$ is a contraction on $C^{1+\theta}_b(H_n)$ with respect to the (equivalent) norm
\[
\psi\mapsto 
\|\psi\|_{C^0_b(H_n)}+
    L_{\beta,\delta,\theta}
    \norm{A_n^\beta D\psi}_{C^0_b(H_n; H_n)}+
    \norm{A_n^\beta D\psi}_{C^\theta_b(H_n;H_n)}\,.
\]

To summarize, we have shown 
that, in both cases $\beta+\delta\in[0,\frac12)$
and $\beta+\delta=\frac12$,
there exists a unique $u_{n,k}\in C^{1+\theta}_b(H_n)$ solving \eqref{eq:kolm_mild}. Moreover, since $\mathcal V_{n,k}: C^{1+\theta}_b(H_n)\to C^{2}_b(H_n)$ it also holds that  $u_{n,k}\in C^2_b(H_n)$.
This concludes the proof.
\end{proof}

\begin{rmk}
Let us point out that for the fixed-point method in $C^{1+\theta}_b$ proposed in the proof of Proposition~\ref{prop:kolm} the regime $\beta+\delta=\frac12$ turns out to be critical. 
Indeed, the technique exploits the \textcolor{blue}{H\"older}-continuity of $A_n^\beta
D R_{n,t}^{z_0}$, which 
crucially relies in turn on the inequality $1-\frac{\theta(\beta-\frac\theta2+\delta(2-\theta))}{1-\beta-(1-\theta)(\frac12+\delta)}\geq\theta$ for the last exponent appearing in \eqref{exp:h} $($notice that $\theta(1-\beta)=1-\beta-(1-\theta)(\frac12+\delta)$ if $\beta+\delta=\frac12)$: this condition is guaranteed if and only if $\beta+\delta\leq\frac12$.  We stress, however, that this notion of criticality seems to be related to the specific fixed-point method employed. Investigation whether such criticality is instead intrinsic in the equation is currently carried out.
\end{rmk}

\begin{rmk}
We point out that the estimates in Proposition~\ref{stimeHolder}
are uniform with respect to $n$ and also hold for the infinite-dimensional Ornstein-Uhlenbeck semigroup $R$ defined in \eqref{OU}.
As a consequence, the same computations performed in the proof of Proposition~\ref{prop:kolm} yield, for every $g\in C^\theta_b(H)$ and every $k\in\N_+$, also the existence and uniqueness of a solution $u_k\in C^{1+\theta}_b(H)$ with 
$Du_k\in C^\theta_b(H; D(A^\beta))$ to the mild infinite-dimensional Kolmogorov equation
\begin{equation}\label{eq-infinita}
u_k(x):=\int_0^{\infty}e^{-\bar c\lambda_k t}
R_t\left[\scal{B}{Du_k}+g\right](x)\,\d t, \qquad x\in H\,.
\end{equation}
We stress, however, that the arguments 
for the $C^2$-regularity are not uniform in $n$, 
and one cannot expect in general that $u_k$ is also 
a classical solution to the Kolmogorov equation.
Indeed, in the infinite-dimensional case, 
additional structural assumptions are needed, see e.g.~\cite{BOS, Dap-Zab2002}.  In the present paper, we will only 
use the sequence $(u_{n,k})_{n,k}$ and we will not need to deal with the solutions $(u_k)_k$ to the infinite-dimensional equation \eqref{eq-infinita}.
\end{rmk}

\subsection{Uniform estimates}
We focus here on uniform estimates for the solution $u_{n,k}$
to the Kolmogorov equation \eqref{eq:kolm}, 
independent of $n$.
We collect the results in the following statement.
\begin{proposition}
\label{prop:est}
Assume \ref{H0}, \ref{H1}, \ref{H2_0}--\ref{H6_0} and, for every $n\in\N_+$ and $k\in\{1,\ldots,n\}$,
let $u_{n,k}$ be the solution to the Kolmogorov equation \eqref{eq:kolm} as given by Proposition~\ref{prop:kolm}. Then, the following statements hold.
\begin{itemize}
\item 
For all \lc{$\gamma\in[0,\beta]$}, there exist constants $\mathfrak C_1, \mathfrak C_2>0$,
depending only on the data $\beta,\gamma,\delta,\theta,B$, such that, for every $\bar c$ satisfying \eqref{c}--\eqref{c'}, for every $n\in\N_+$ and $k\in\{1,\ldots,n\}$, it holds
\begin{align}
\label{est0}
\|u_{n,k}\|_{C^0_b(H_n)} &\leq \mathfrak C_1        \|g\|_{C^\theta_b(H)}\,,\\ 
\label{est1}
\|A_n^\gamma Du_{n,k}\|_{C^0_b(H_n; H_n)}&\leq \frac{\mathfrak C_2}{(\bar c\lambda_k)^{\frac{1+\theta}{2}-(1-\theta)\lc{\max\{0,\gamma+\delta\}-\theta\gamma}}}\|g\|_{C^\theta_b(H)}\,.
\end{align}
\item \lc{For all $\gamma_1,\gamma_2\in\R$
with $\max\{0,\gamma_1+\delta\}+
(1-\theta)\max\{0,\gamma_2+\delta\}
+\theta\max\{0,\gamma_2\}<\frac{\theta}{2}$,}
there exists a constant $\mathfrak C_3>0$,
depending only on the data $\beta,\gamma_1,\gamma_2,\delta,\theta,B$, such that, for every $\bar c$ satisfying \eqref{c}--\eqref{c'},
for every $n\in\N_+$ and $k\in\{1,\ldots,n\}$, it holds \begin{align}
\label{est2}
\|D(A_n^{\gamma_2}Du_{n,k})A_n^{\gamma_1}\|_{C^0_b(H_n; \cL(H_n))}
\leq \frac{\mathfrak C_3}
{(\bar c\lambda_k)^{\lc{
\frac{\theta}{2}-
\max\{0,\gamma_1+\delta\}-
(1-\theta)\max\{0,\gamma_2+\delta\}
-\theta\max\{0,\gamma_2\}
}}}
\|g\|_{C^\theta_b(H)}
\end{align}
and
\begin{align}
\label{est2_bis}
\|D(A_n^{\gamma_1}Du_{n,k})A_n^{\gamma_2}\|_{C^0_b(H_n; \cL(H_n))}
\leq \frac{\mathfrak C_3}
{(\bar c\lambda_k)^{\lc{
\frac{\theta}{2}-
\max\{0,\gamma_1+\delta\}-
(1-\theta)\max\{0,\gamma_2+\delta\}
-\theta\max\{0,\gamma_2\}
}}}
\|g\|_{C^\theta_b(H)}\,.
\end{align}
\end{itemize}
\end{proposition}
\begin{rmk}
Note that the constants $\mathfrak C_1, \mathfrak C_2, \mathfrak C_3$ do not depend on the values of $\bar c$.
Moreover, the exponent appearing in \eqref{est1} is positive 
for every $\gamma\in[0,\beta]$
thanks to \ref{H2_0}. 
\end{rmk}

\begin{proof}
\underline{\sc Step 1.} We first prove that there exists a constant $\mathfrak C>0$, depending only on 
$\beta,\delta,\theta, B$, such that 
\begin{equation}
\label{est_aux1}
\|A_n^\beta Du_{n,k}\|_{C^\theta_b(H_n;H_n)}
\leq
\mathfrak C \|g\|_{C^\theta_b(H)}\,.
\end{equation}
This will be shown separately in the two cases 
$\beta+\delta\in[0,\frac12)$ and $\beta+\delta=\frac12$.\\
{\bf Case \boldsymbol{$\beta+\delta\in[0,\frac12)$}.}
We use the mild formulation \eqref{eq:kolm_mild}
and estimate \eqref{Hstima1.5var} with $\gamma=\beta$
to infer that 
\begin{align*}
\|A_n^\beta Du_{n,k}\|_{C^\theta_b(H_n;H_n)}
\leq M_{\beta,\delta,\theta}
\int_0^{\infty}
e^{-\bar c\lambda_k t}
\frac1{t^{\frac12+\beta+\delta}}
\left[\|\scal{B_n}{Du_{n,k}}\|_{C^\theta_b(H_n)}
+\|g_{|H_n}\|_{C^\theta_b(H_n)}\right]\,\d t\,.
\end{align*}
Hence, by using the estimates \eqref{est:Bn2} and \eqref{est:g} and proceeding as in the proof of Proposition~\ref{prop:kolm}, we infer that 
\begin{align*}
\|A_n^\beta Du_{n,k}\|_{C^\theta_b(H_n;H_n)}
&\leq M_{\beta,\delta,\theta}
\int_0^{\infty}
e^{-\bar c\lambda_k t}
\frac1{t^{\frac12+\beta+\delta}}
\left[
C_B\|A_n^\beta Du_{n,k}\|_{C^\theta_b(H_n;H_n)}
+\|g\|_{C^\theta_b(H)}\right]\,\d t\\
&\leq 
\frac1{\bar c^{\frac12-\beta-\delta}}
\left[
\frac1{\lambda_1^{\frac12-\beta-\delta}}M_{\beta,\delta,\theta}C_B\Gamma\left(\frac12-\beta-\delta\right)\right]
\|A_n^\beta Du_{n,k}\|_{C^\theta_b(H_n;H_n)}\\
&\qquad+
\frac{M_{\beta,\delta,\theta}}
{\bar c^{\frac12-\beta-\delta}}
\Gamma\left(\frac12-\beta-\delta\right)
\frac1{\lambda_1^{\frac12-\beta-\delta}}\|g\|_{C^\theta_b(H)}
\,.
\end{align*}
By exploiting now the definition of $\bar c$ in \eqref{c}, we obtain 
\[
\|A_n^\beta Du_{n,k}\|_{C^\theta_b(H_n;H_n)}
\leq \left[2\frac{M_{\beta,\delta,\theta}}{\bar c^{\frac12-\beta-\delta}}
\Gamma\left(\frac12-\beta-\delta\right)\right]
\frac1{\lambda_1^{\frac12-\beta-\delta}}
\|g\|_{C^\theta_b(H)}\,,
\]
in particular condition \eqref{c} guarantees that \eqref{est_aux1} holds in the case $\beta+\delta\in[0,\frac12)$ with a constant $\mathfrak C$ independent of the value of $\bar c$.\\
{\bf Case \boldsymbol{$\beta+\delta=\frac12$}.}
    We use here the same notation as in the proof of Proposition~\ref{prop:kolm} in the critical case $\beta+\delta=\frac12$ (recall that $\delta\geq0$). The equivalent mild formulation of the Kolmogorov equation written in terms of $\tilde B_n$ and
    $R_{n}^{z_0}$ (see \eqref{V_aux2}) reads as
    \[
  u_{n,k}(x):=
  \int_0^{+\infty}e^{-\bar c\lambda_k t} 
  R^{z_0}_{n,t}\left[\langle\tilde B_n, Du_{n,k}\rangle 
  + g_{|H_n}\right](x)\,\d t\,,
  \quad x\in H_n\,.
  \]
  By using analogous arguments as the ones to deduce inequality \eqref{normaCb} in the proof of Proposition~\ref{prop:kolm} and by recalling \eqref{est:g}, we get
  \begin{align}
      \nonumber
      \|A_n^\beta Du_{n,k}\|_{C^0_b(H_n;H_n)}
      &\leq\frac1{\bar c^{\theta(1-\beta)}}
  \left[
  \frac{2M_{\beta,\delta,\theta}}{\lambda_1^{\theta(1-\beta)}}
  \Gamma(\theta(1-\beta))(C_{\tilde B}+C_B)
  \right]
  \norm{A_n^\beta Du_{n,k}}_{C^\theta_b(H_n; H_n)}\\
  \label{est_aux2}
  &\qquad+
  \frac{M_{\beta,\delta,\theta}}{(\bar c\lambda_1)^{\theta(1-\beta)}}
  \Gamma(\theta(1-\beta))\|g\|_{C^\theta_b(H)}\,.
  \end{align}
  The same arguments used to prove inequality \eqref{normaCt} in the proof of Proposition~\ref{prop:kolm} and \eqref{est:g} yield
  \begin{align}
    \nonumber
    &[A_n^\beta Du_{n,k}]_{C^\theta_b(H_n;H_n)}\\
    \label{est_aux3}
    &\leq
\frac{M_{\beta,\delta,\theta}(2-\theta)}{\theta(1-\theta)(1-\beta)}
    \left[
  C_{\tilde B}
  \norm{A_n^\beta Du_{n,k}}_{C^\theta_b(H_n; H_n)}+
  (C_{\tilde B}+C_B)
  \norm{A_n^\beta Du_{n,k}}_{C^0_b(H_n; H_n)}
  +\|g\|_{C^\theta_b(H)}\right]\,.
\end{align}
Hence, by multiplying \eqref{est_aux2} by the constant 
$1+L_{\beta,\delta,\theta}$,
by summing up with \eqref{est_aux3} and by defining the constant 
\[ L^{''''}_{\beta, \delta, \theta}:= 
\frac{M_{\beta,\delta,\theta}}{(\bar c\lambda_1)^{\theta(1-\beta)}}
  \Gamma(\theta(1-\beta))(1+L_{\beta,\delta,\theta})
  +L_{\beta,\delta,\theta}''\,,
  \]
we obtain 
\begin{align*}
    &L_{\beta,\delta,\theta}
    \norm{A_n^\beta Du_{n,k}}_{C^0_b(H_n; H_n)}
    +\|A_n^\beta Du_{n,k}\|_{C^\theta_b(H_n;H_n)}\\
    &\leq
    \left[
    \frac{L_{\beta,\delta,\theta}'}{\bar c^{\theta(1-\beta)}}
    +C_{\tilde B}L_{\beta,\delta,\theta}''
    \right]
  \|A^\beta_n Du_{n,k}\|_{C^\theta_b(H_n;H_n)}
  +
\frac{L_{\beta,\delta,\theta}}2
\norm{A_n^\beta Du_{n,k}}_{C^0_b(H_n; H_n)}
+L^{''''}_{\beta, \delta, \theta} \|g\|_{C^\theta_b(H)}\,. 
\end{align*}
Hence, thanks to condition \eqref{c'} 
\begin{align*}
\frac{L_{\beta,\delta,\theta}}2
\norm{A_n^\beta Du_{n,k}}_{C^0_b(H_n; H_n)}
+\frac12\|A_n^\beta Du_{n,k}\|_{C^\theta_b(H_n;H_n)}
\leq L^{''''}_{\beta, \delta, \theta}\|g\|_{C^\theta_b(H)}\,.
\end{align*}
In particular condition \eqref{c'} guarantees that
\eqref{est_aux1} holds also in the case 
$\beta+\delta=\frac12$ with a constant $\mathfrak C$ independent of the value of $\bar c$.

\underline{\sc Step 2.}
By proceeding as in \eqref{V_aux1}, we have
\begin{align*}
\norm{u_{n,k}}_{C^0_b(H_n)}
&\leq \frac1{\bar c\lambda_1}
\left(C_B
\norm{A_n^\beta Du_{n,k}}_{C^\theta_b(H_n;H_n)}
+\|g\|_{C^\theta_b(H)}\right)\,,
\end{align*}
so that \eqref{est0} follows from \eqref{est:g} and 
\eqref{est_aux1}. The constant $\mathfrak C_1$ is independent of the value
of $\bar c$ thanks to \eqref{c}--\eqref{c'}.

\underline{\sc Step 3.}
Now, we use the mild formulation \eqref{eq:kolm_mild} and estimate \eqref{Hstima1var} with 
\lc{$\gamma\in[0,\beta]$} to infer that 
\begin{align*}
&\|A_n^\gamma Du_{n,k}\|_{C^0_b(H_n;H_n)}\\
&\leq M_{\gamma,\delta,\theta}
\int_0^{\infty}
e^{-\bar c\lambda_k t}
\frac1{t^{(1-\theta)(\frac12+
\lc{\max\{0,\gamma+\delta\}})+\theta\gamma}}
\left[\|\scal{B_n}{Du_{n,k}}\|_{C^\theta_b(H_n)}
+\|g_{|H_n}\|_{C^\theta_b(H_n)}\right]\,\d t\\
&\leq M_{\gamma,\delta,\theta}
\int_0^{\infty}
e^{-\bar c\lambda_k t}
\frac1{t^{(1-\theta)(\frac12+
\lc{\max\{0,\gamma+\delta\}})+\theta\gamma}} \left[2C_B\|A_n^\beta Du_{n,k}\|_{C^\theta_b(H_n;H_n)}
+\|g\|_{C^\theta_b(H)}\right]\,\d t\,.
\end{align*}
Hence, by taking \eqref{est_aux1}
into account we get 
\begin{align*}
\|A_n^\gamma Du_{n,k}\|_{C^0_b(H_n;H_n)}
&\leq M_{\gamma,\delta,\theta}
\left(
2C_B\mathfrak C+1
\right)\|g\|_{C^\theta_b(H)}
\int_0^{+\infty}
e^{-\bar c\lambda_k t}
\frac1{t^{(1-\theta)(\frac12+\lc{\max\{0,\gamma+\delta\}})+\theta\gamma}}\,\d t\\
&\leq \frac{M_{\gamma,\delta,\theta}\left(
2C_B\mathfrak C+1
\right)} {\bar c^{\frac{1+\theta}{2}-(1-\theta)
\lc{\max\{0,\gamma+\delta\}-\theta\gamma}}}\Gamma\left(\frac{1+\theta}{2}-(1-\theta)\lc{\max\{0,\gamma+\delta\}-\theta\gamma}\right)\\
&\qquad\times\frac{1}
{\lambda_k^{\frac{1+\theta}{2}-(1-\theta)\lc{\max\{0,\gamma+\delta\}-\theta\gamma}}}
\|g\|_{C^\theta_b(H)}\,,
\end{align*}
so that \eqref{est1} is proved. Furthermore, by using the estimate \eqref{Hstima2var} in the mild formulation \eqref{eq:kolm_mild}, by \eqref{est_aux1} we obtain
\begin{align*}
&\|D(A_n^{\gamma_2}Du_{n,k})A_n^{\gamma_1}\|_{C^0_b(H_n; \cL(H_n))}\\
&\leq 
M_{\gamma_1,\gamma_2,\delta,\theta}
\left[2C_B\|A_n^\beta Du_{n,k}\|_{C^\theta_b(H_n;H_n)}
+\|g\|_{C^\theta_b(H)}\right]\\
&\qquad\times\int_0^{\infty}
e^{-\bar c\lambda_k t}
\frac{1}{t^\lc{1-\frac{\theta}{2}+
\max\{0,\gamma_1+\delta\}+
(1-\theta)\max\{0,\gamma_2+\delta\}
+\theta\max\{0,\gamma_2\}
}}\,\d t\\
&\leq M_{\gamma_1,\gamma_2,\delta,\theta}
\left(2C_B\mathfrak C+1\right)\|g\|_{C^\theta_b(H)}
\int_0^{\infty}
e^{-\bar c\lambda_k t}
\frac{1}{t^\lc{1-\frac{\theta}{2}+
\max\{0,\gamma_1+\delta\}+
(1-\theta)\max\{0,\gamma_2+\delta\}
+\theta\max\{0,\gamma_2\}
}}\,\d t\\
&\leq \frac{M_{\gamma_1,\gamma_2,\delta,\theta}\left(
2C_B\mathfrak C+1
\right)}
{\bar c^{\lc{\frac{\theta}{2}-
\max\{0,\gamma_1+\delta\}-
(1-\theta)\max\{0,\gamma_2+\delta\}
-\theta\max\{0,\gamma_2\}}}}\\
&\qquad\times\Gamma\left(\lc{\frac{\theta}{2}-
\max\{0,\gamma_1+\delta\}-
(1-\theta)\max\{0,\gamma_2+\delta\}
-\theta\max\{0,\gamma_2\}}\right)\\
&\qquad\times
\frac{1}
{\lambda_k^{\lc{\frac{\theta}{2}-
\max\{0,\gamma_1+\delta\}-
(1-\theta)\max\{0,\gamma_2+\delta\}
-\theta\max\{0,\gamma_2\}}}}
\|g\|_{C^\theta_b(H)}\,,
\end{align*}
so that \eqref{est2} is proved. 
Analogously, \eqref{est2_bis} follows by \eqref{Hstima2var_bis}.
We note that the constants $\mathfrak C_2, \mathfrak C_3$ are independent of the value of $\bar c$  since so is $\mathfrak C$.
\end{proof}


\section{Proof of the main results}
\label{sec:U0}
In this section we prove the main 
Theorems~\ref{th:1}, \ref{th:2}, 
and \ref{th:3}. To this aim, we first introduce a finite dimensional approximation procedure. 

\subsection{Finite dimensional approximation of the SPDE}
\label{ssec:approx}
Let $x\in H$ be fixed and let
$(X,W)$ be a weak mild solution to \eqref{SDE}
in the sense of Definition~\ref{def-sol} with initial datum $x$. 
For every $n\in\N_+$, 
we consider the $H_n$-valued stochastic process $X_n:=P_nX$,
that satisfies
\begin{equation*}
X_n(t):=e^{-tA_n} P_nx
+\int^t_0e^{-(t-s)A_n}B_n(X(s))\,\d s
+W_{A_n}(t)\,,
\quad t\geq0\,, 
\end{equation*}
where $W_{A_n}$ is the stochastic convolution process given by 
\begin{equation*}
W_{A_n}(t):=\int^t_0e^{-(t-s)A_n}A_n^{-\delta}\,\d W(s)\,, \quad t\geq0\,.
\end{equation*}
In particular, for every $n\in\N_+$ the process $X_n$ solves on $H_n$
the following It\^o equation
\begin{align}\label{eq-approssimata}
X_n(t)+\int_0^tA_nX_n(s)\,\d s
=P_nx + \int_0^tB_n(X(s))\,\d s
+\int_0^tA_n^{-\delta}\,\d W(s)\,, \quad\forall\,t\geq0\,,
\quad\P\text{-a.s.}
\end{align}

\begin{proposition}\label{prop:Xn}
Assume \ref{H0}, \ref{H1},
\ref{H2_0}, \ref{H3_0}, \ref{H6_0}, and let $x\in H$. Then, 
for every $T>0$ we have
\begin{equation}\label{supE}
\lim_{n \to \infty}\sup_{t\in [0,T]}\mathbb{E}
\|X_{n}(t)-X(t)\|_H=0\,.
\end{equation} 
\end{proposition}

\begin{proof}
\lc{The proof is similar to the one of \cite[Proposition 4.2]{AddBig}, where the case $\beta=0$ is considered: the main difference is that here $B$ takes values in the larger space $D(A^{-\beta})$, which is handled by means of the smoothing estimate \eqref{semiS}. Hence, we only sketch the argument.
Firstly, since $A_n=AP_n$ and $e^{-tA}$ is a contraction, we have
$\sup_{t\geq0}\norm{e^{-tA_n} P_nx-e^{-tA}x}_H\leq\norm{P_nx-x}_H\to0$ as $n\to\infty$.
Secondly, since $\beta<1$, let us fix $q\in(1,\frac1\beta)$
and $p:=\frac{q}{q-1}$. By \eqref{semiS} with $\gamma=\beta$ and the H\"older inequality,
for every $n\in\N_+$ and $t\in[0,T]$ we get
\begin{align*}
&\E\norm{\int_0^te^{-(t-s)A_n}B_n(X(s))\,\d s
-\int_0^te^{-(t-s)A}B(X(s))\,\d s}^2_H\\
&\qquad\leq M^2_\beta
\left(\int_0^T\frac{1}{s^{q\beta}}\,\d s\right)^{\frac2q}
\E\norm{P_nB(X)-B(X)}_{L^p(0,T;D(A^{-\beta}))}^2\,,
\end{align*}
Since $B(X)$ is bounded with values in $D(A^{-\beta})$ by \ref{H6_0}, the right-hand side converges to $0$ by the dominated convergence theorem.
Finally, by the It\^o isometry, for every $t\geq0$ we have
\[
\mathbb{E}\|W_A(t)-W_{A_n}(t)\|^2_H
=\int_0^t\sum_{k=n+1}^\infty e^{-2s\lambda_k}
\lambda_k^{-2\delta}\,\d s
\leq \sum_{k=n+1}^\infty
\frac1{\lambda_k^{1+2\delta}}\,,
\]
which converges to $0$ uniformly in $t$ as $n\to\infty$ thanks to \ref{H3_0}.}
\end{proof}

\subsection{Weak uniqueness}\label{ssec:weak}
We prove here Theorem~\ref{th:1}:
the proof is divided in multiple steps. First,
we focus on the case $\alpha=0$, by further assuming 
\ref{H5_0}--\ref{H6_0}. Secondly, we remove assumption 
\ref{H5_0} and \ref{H6_0}. Eventually, 
we extend to the general case $\alpha>0$ 
and general initial data. 

\begin{proof}[Proof of Theorem~\ref{th:1}]
Let $(X,W)$ and $(Y,W)$ be two weak mild solutions to \eqref{SDE} with same initial datum $x$ in the sense of Definition~\ref{def-sol}, defined on the same probability space. 
    
\underline{\sc Step 1.} Here, we also assume that $\alpha=0$ and that \ref{H5_0}--\ref{H6_0} hold: 
\ref{H5_0}--\ref{H6_0}
will be removed in {\sc Step 2} while 
the case $\alpha>0$ will be treated in {\sc Step 3.}
Let $(X_n)_n$ and $(Y_n)_n$ be the respective approximated solutions as constructed in Subsection~\ref{ssec:approx}.
Let $g\in C^\theta_b(H)$ be arbitrary but fixed and let $(u_{n,1})_{n\in\N_+}$ be the family of solutions to the Kolmogorov equation \eqref{eq:kolm} with the choice $n\in\N_+$
and $k=1$, as given by Proposition~\ref{prop:kolm}. Then, an application of the It\^o formula yields,
together with the Kolmogorov equation \eqref{eq:kolm}, for all $t\geq0$,
\begin{align}
\nonumber &\E\left[e^{-\bar c\lambda_1 t}(u_{n,1}(X_n(t))-u_{n,1}(Y_n(t)))\right]
+\E\int_0^te^{-\bar c\lambda_1 s}\left[g(X_n(s))-g(Y_n(s))\right]\,\d s\\
\label{ito_aux}
= & \lc{\E\int_0^te^{-\bar c\lambda_1 s}\langle D u_{n,1}(X_n(s)),
B_n(X(s))-B_n(X_n(s))\rangle\,\d s} \notag \\
& \lc{-\E\int_0^te^{-\bar c\lambda_1 s}\langle Du_{n,1}(Y_n(s)),
B_n(Y(s))-B_n(Y_n(s))\rangle\,\d s\,.}
\end{align}
We let now $n\to\infty$ in all the terms separately, with $t>0$ fixed.
First, thanks to \eqref{supE} and the dominated convergence theorem, 
we have that 
\begin{align*}
&\left|\E\int_0^te^{-\bar c\lambda_1 s}\left[g(X_n(s))
-g(Y_n(s))\right]\,\d s
-\E\int_0^te^{-\bar c\lambda_1s}\left[g(X(s))
-g(Y(s))\right]\,\d s\right|\\
&\qquad\leq\|g\|_{C^\theta_b(H)}
\E\int_0^te^{-\bar c\lambda_1 s}
\left[\|X_n(s)-X(s)\|_H^\theta+\|Y_n(s)-Y(s)\|_H^\theta\right]\,\d s\to 0\,.
\end{align*}
Furthermore, thanks to \eqref{est0} we have 
\begin{align*}
\limsup_{n\to\infty}\left|
\E\left[e^{-\bar c\lambda_1 t}(u_{n,1}(X_n(t))-u_{n,1}(Y_n(t)))\right]\right|
\leq 2\mathfrak C_1e^{-\bar c\lambda_1 t}\|g\|_{C^\theta_b(H)}\,.
\end{align*}
Finally, thanks to \eqref{est:Bn2} and \eqref{supE}, together with the dominated convergence theorem, we infer that
\begin{align}\nonumber
&\left|\E\int_0^te^{-\bar c\lambda_1 s}\lc{\langle Du_{n,1}(X_n(s)),}B_n(X(s))-B_n(X_n(s))\lc{\rangle}\,\d s\right.\\
\nonumber
&\qquad\left.-\E\int_0^te^{-\bar c\lambda_1 s}\lc{\langle Du_{n,1}(Y_n(s)),}B_n(Y(s))-B_n(Y_n(s))\lc{\rangle}\,\d s\right|\\
\label{eq:aux}
&\qquad\leq \lc{\sup_{m\in\N_+}\|A_m^\beta Du_{m,1}\|_{C^0_b(H_m;H_m)}}\,C_B\,\E\int_0^t\left[\|X_n(s)-X(s)\|_H^\theta+\|Y_n(s)-Y(s)\|_H^\theta\right]\,\d s\to 0\,,
\end{align}
\lc{where the supremum is finite by \eqref{est1} with $\gamma=\beta$.}
Hence, by letting $n\to\infty$ in \eqref{ito_aux}, we get
\[
\left|\E\int_0^te^{-\bar c\lambda_1 s}\left[g(X(s))
-g(Y(s))\right]\,\d s\right|\leq 
2\mathfrak C_1e^{-\bar c\lambda_1 t}\|g\|_{C^\theta_b(H)}\,
\]
and, by letting now $t\to\infty$, we infer that 
\[
\E\int_0^{+\infty}e^{-\bar c\lambda_1 s}g(X(s))\,\d s=\E\int_0^{+\infty}e^{-\bar c\lambda_1 s}g(Y(s))\,\d s\,.
\]
Hence, by the properties of the Laplace transform, weak uniqueness for \eqref{SDE} holds under the additional assumptions \ref{H5_0}--\ref{H6_0},  $\alpha =0$, in the sense of Definition~\ref{uniqueness}.

\underline{\sc Step 2.} First, we show how to remove assumption \ref{H5_0}: the idea is to exploit \cite[Thm.~26]{Pri2015} by arguing as in \cite[Lem.~22]{Pri2021}.
To this aim, we consider the set of smooth and cylindrical functions $C^2_{cyl}(H)$ on $H$ as the set of functions $g:H\to\mathbb R$ such that 
there exist ${i_1},\ldots,{i_N}\in\N$ and 
$\tilde g\in C^2(\mathbb R^N)$ with compact support such that 
$g(x)=\tilde g((x,e_{i_1})_H, \ldots, (x,e_{i_N})_H)$
for all $x\in H$.
Let $\mathcal{N}$ be the second order Kolmogorov operator associated to \eqref{SDE}, namely
\[
\mathcal N\varphi(x):=
-\frac{1}{2}\Tr\left[A^{-2\delta}D^2\varphi(x)\right]+\scal{Ax}{D\varphi(x)}+\scal{B(x)}{D\varphi(x)}\,, \quad x\in H\,, \quad \varphi\in C^2_{cyl}(H)\,.
\]
We note that $\mathcal{N}$ verifies \cite[Hyp.~17]{Pri2015} 
as required in 
\cite[Thm.~26]{Pri2015} (see \cite[Rmk.~8]{Pri2015}).
Furthermore, 
by proceeding as in the proof of \cite[Lem.~22]{Pri2021}, 
exploiting the continuity of $B$, 
there exist two sequences $\{x_k\}_{k\in\N_+}\subseteq H$ and $\{r_k\}_{k\in\N_+}\subseteq \R^+$ such that
\begin{itemize}
\item setting $U_k:=\{x\in H\; :\; \norm{x-x_k}< r_k/2\}$, $k\in\mathbb N$, then $H=\cup_{k\in\N_+}U_k$;
\item setting $\tilde U_k:=\{x\in H\; :\; \norm{x-x_k}< r_k\}$, $k\in\mathbb N$, then
\[
\sup_{x\in \tilde U_k}\norm{B(x)-B(x_k)}_{-\beta}<
\frac{\theta(1-\theta)(1-\beta)}
{4M_{\beta,\delta,\theta}(2-\theta)}\,.
\]
\end{itemize}
Let $\rho\in C^\infty_0(\R^+;[0,1])$ be such that 
\begin{align*}
\rho(s)=1\quad \forall\,s\in [0,1]\,, \qquad \rho(s)=0\quad \forall\,s\in [2,\infty)\,.
\end{align*}
For every $k\in\N_+$ we define
\[
B_k(x):=\rho\left(4r_k^{-2}\norm{x-x_k}^2\right)B(x)+\left(1-\rho\left(4r_k^{-2}\norm{x-x_k}^2\right)\right)B(x_k),\qquad x\in H\,.
\]
We note that for every $k\in\N_+$ we have 
$B_k\in C^\theta_b(H; D(A^{-\beta}))$ and 
\[
B_k(x)-B(x_k)=
\begin{cases}
\rho(4r_k^{-2}\norm{x-x_k}^2)(B(x)-B(x_k)),\quad&\text{if }x\in \tilde U_k\,,\\
0 \quad&\text{if } x\not\in \tilde U_k\,,
\end{cases}
\]
so that $B_k$ verifies \ref{H5_0} with $z_0=B(x_k)$.
For every $k\in\N_+$, we define
\[
\mathcal N_k\varphi(x):=-\frac{1}{2}\Tr\left[A^{-2\delta}D^2\varphi(x)\right]+\scal{Ax}{D\varphi(x)}
+\scal{B_k(x)}{D\varphi(x)}\,, \quad x\in H\,, \quad\varphi\in C^2_{cyl}(H)\,,
\]
so that by definition of $B_k$ we have
\[
\mathcal{N}_k\varphi(x)=\mathcal{N}\varphi(x) 
\quad\forall\, x\in U_k\,,
\]
which implies that condition (ii) of \cite[Thm.~26]{Pri2015} is satisfied.
Moreover, for every $k\in\N$, the equation 
\begin{align*}
\d X + AX \,\d t=B_k(X)\,\d t+A^{-\delta}\,\d W\,, \qquad X(0)=x\,,
\end{align*}
admits a weak mild solution
(see Remark~\ref{rmk:ex_eq}), which is also unique by {\sc Step 1}.
Hence, thanks to the 
equivalence between martingale and weak mild solutions
given in \cite[Thm.~3.6 and Prop.~6.8]{Kun2013},
we infer that also condition (i) of \cite[Thm.~26]{Pri2015} holds.
We can then conclude by applying 
\cite[Thm.~26]{Pri2015} and exploiting again the equivalence between martingale and weak mild solutions given in \cite[Thm.~3.6 and Prop.~6.8]{Kun2013}.
Eventually, assumption \ref{H6_0} can be removed
by following directly the procedure presented in \cite[Sec.~4.2]{BOS} (with $\alpha=0$). This proves  weak uniqueness for \eqref{SDE} in the case $\alpha =0$, in the sense of Definition~\ref{uniqueness}.

\underline{\sc Step 3.} We consider here the general case $\alpha>0$ and $x\in D(A^\alpha)$. We set
\[
  \mathcal H:=D(A^{\alpha})
\]
and $\tilde e_k:=\lambda_k^{-\alpha} e_k$ for all $k\in\enne$. 
It follows that $\{\tilde e_k:k\in\N_+\}$ is an orthonormal basis of
$\mathcal H$, with $A\tilde e_k=\lambda_k\tilde e_k$ for all $k\in\enne_+$. By denoting by $\mathcal A$ the restriction of
the operator $A$ to $\mathcal H$, with effective domain
given by $D(\mathcal A)=D(A^{1+\alpha})$, one has 
\[ 
  D(\mathcal A^{\sigma}) = D(A^{\sigma+\alpha}) \quad\forall\,\sigma\in\mathbb R\,.
\]
With this notation, $\mathcal W:=A^{-\alpha} W$
is a cylindrical Wiener process on $\mathcal H$, so that,
by setting $\tilde\delta:=\delta-\alpha$, one has formally
that $A^{-\delta}\,\d W=\mathcal A^{-\tilde\delta}\d \mathcal  W$.
Analogously, by setting $\tilde\beta:=\alpha+\beta$, 
the operator $B$ can be reformulated 
as $B:\mathcal  H\to D(\mathcal  A^{-\tilde\beta})$.

It is immediate to check that   assumptions \ref{H1}, 
\ref{H2_0}--\ref{H4_0} are satisfied by
the space $\mathcal  H$,
the parameters $\tilde\beta$, $\tilde \delta$, 
and the operators $\mathcal  A$, $B$. Indeed, note that
\lc{$\tilde\beta=\alpha+\beta\in[0,1)$, $\tilde\delta=\delta-\alpha\in(-\frac12, \frac12]$, and
$\tilde\beta+\tilde\delta=\beta+\delta$, so that by \ref{H2} we have
$\tilde\beta+\tilde\delta\in[0,\frac12]$ if $\tilde\delta\geq0$ and $\tilde\beta+\tilde\delta\in[0,\frac12)$ if $\tilde\delta<0$; moreover,
$\mathcal A^{-(1+2\tilde\delta)+\eta}=\mathcal A^{-(1+2\delta)+2\alpha+\eta}$, so that \ref{H3_0} follows from \ref{H3},}
as well as the operator $B$ belongs to
$C^\theta_{{\rm loc}}(\mathcal H; D(\mathcal A^{-\tilde\beta}))$ for some $\theta\in(0,1)$. Hence, 
as a consequence of {\sc Step 1--2} (with $H,\beta, \delta, A, W$ replaced by $\mathcal  H, \tilde\beta$, $\tilde \delta,\mathcal  A,\mathcal W$) 
we have proved that for every bounded and measurable $\psi:C^0(\R_+;\mathcal  H)\to \R$, it holds that
\[
\mathbb E[\psi(X)]=\mathbb E[\psi(Y)]\,. 
\]
 It remains to show that this holds also 
 for every bounded and measurable $\psi:C^0(\R_+;H)\to \R$. To this aim, we note that if $\psi:C^0(\R_+;H)\to \R$ is bounded and measurable, then also $\tilde \psi:=\psi_{|C^0(\R_+;\mathcal  H)}:C^0(\R_+;\mathcal  H)\to \R$ is bounded and measurable: this follows from the fact that $C^0((0,+\infty); D(A^\alpha))$ is a Borel subset of $C^0((0,+\infty);H)$ $($see \cite[Thm.~15.1]{Kec1995}$)$.
 By Definition~\ref{def-sol} 
 one has $X,Y\in C^0(\R_+;\mathcal  H)$ $\mathbb P$-a.s. so that  for every bounded and measurable $\psi:C^0(\R_+;H)\to \R$
\begin{align*}
\mathbb E[\psi(X)]
=\mathbb E[\tilde \psi(X)]
= \mathbb E[\tilde \psi(Y)]
= \mathbb E[\psi(Y)].
\end{align*}
 This proves that weak uniqueness holds for \eqref{SDE}
 with initial data $D(A^\alpha)$, in the sense of Definition~\ref{uniqueness}.
 
 \underline{\sc Step 4.} We consider here the case
 where $B$ is bounded, $\alpha>0$, and $x\in H$.
 By performing the same construction as in {\sc Step 3},
 it holds that $x\in D(\mathcal  A^{-\alpha})$.
We can then replicate the arguments of {\sc Step 1} and {\sc Step 2}, 
by using the refined convergence results of Proposition~\ref{prop:Xn_alpha} in appendix instead of 
Proposition~\ref{prop:Xn}, with the choices $\mathcal H=D(A^\alpha)$,
$\tilde \alpha=\alpha$, $\tilde \beta=\alpha+\beta$,
and $\tilde\delta=\delta-\alpha$.
The only point to adapt is the convergence of
the term corresponding to \eqref{eq:aux}, namely 
\begin{align*}
&\E\int_0^t
\left[\|X_n(s)-X(s)\|_{\mathcal  H}^\theta+\|Y_n(s)-Y(s)\|_{\mathcal  H}^\theta\right]\,\d s\\
&\leq \sup_{s\in(0,t]}\left(s^{{\alpha}\theta}\E\|X_n(s)-X(s)\|_{\mathcal  H}^\theta
+s^{{\alpha}\theta}\E\|Y_n(s)-Y(s)\|_{\mathcal  H}^\theta
\right)
\int_0^ts^{-{\alpha}\theta}\,\d s\to 0
\end{align*}
as $n\to\infty$ thanks to Proposition~\ref{prop:Xn_alpha}
and the fact that $\alpha\theta<1$.
\end{proof}

\subsection{Pathwise uniqueness and continuous dependence}\label{ssec:path}
We prove now Theorems~\ref{th:3}--\ref{th:2}. 
\lc{We first prove the continuous dependence result of Theorem~\ref{th:3}: we deal with the case $\alpha=s=0$ under the additional assumption \ref{H6_0} by means of the Kolmogorov equations studied in Section~\ref{sec:Kolmo-sol}, and then we extend it to the general case through the abstract result of Appendix~\ref{sec:app}.
Secondly, we deduce the pathwise uniqueness result of Theorem~\ref{th:2} from Theorem~\ref{th:3}, and we remove the global H\"older-continuity of $B$ via localisation in the case of initial data in $D(A^\alpha)$.}

\begin{proof}[Proof of Theorem \ref{th:3}]
Let $(X,W)$ and $(Y,W)$ be two weak mild solutions to \eqref{SDE} with initial data $x$ and $y$, respectively, in the sense of Definition~\ref{def-sol}, with the same cylindrical Wiener process $W$ and defined on the same probability space.
Let $(X_n)_n$ and $(Y_n)_n$ be the respective approximated solutions as constructed in Subsection~\ref{ssec:approx}.

\underline{\sc Step 1.} We first prove the result in the simple case where  
$\alpha=s=0$ and \ref{H6_0} holds. To ease the reading of the proof, 
we recall here the formulation of assumptions 
\eqref{trace_s}--\eqref{ineq_s} in the case $\alpha=s=0$:
\begin{align}
\label{trace}
&\exists\,\varepsilon \in (0,1):\quad 
A^{-(1+\theta)+2\beta+2(1-\theta)\lc{\max\{0,\delta\}}+2\varepsilon}\in\cL^1(H,H)\,,\\
\label{ineq2}
&(1-\theta)\lc{\max\{0,\delta\}}<\frac\theta{2}\,.
\end{align}
We further introduce, for every $k\in\N_+$, the function
\begin{equation}
\label{def:gk}
g_{k}:H\to\R\,, \qquad
g_{k}(x):=\scal{B(x)}{e_k}\,, \quad x\in H\,.
\end{equation}
By the definition \eqref{def:gk} and assumption \ref{H6_0}, it follows that $g_k\in C_b^\theta(H)$ for every $k\in\N$ and 
\begin{equation}
\label{est:gk}
\norm{g_k}_{C^0_b(H)}\leq
\norm{g_k}_{C^\theta_b(H)}\leq 
C_B\lambda_k^\beta
\qquad \forall\,k\in\N\,.
\end{equation}
Indeed, by assumptions \ref{H1} and \ref{H6_0}, 
for every $x,y\in H$ one has that 
\[
|g_k(x)|=|\scal{B(x)}{e_k}|
\leq \norm{B}_{C^0_b(H; D(A^{-\beta}))}\norm{e_k}_\beta
= \norm{B}_{C^0_b(H; D(A^{-\beta}))}\norm{A^\beta e_k}_H=\norm{B}_{C^0_b(H; D(A^{-\beta}))}\lambda_k^\beta
\]
and
\begin{align*}
|g_k(x)-g_k(y)|&=
|\scal{B(x)-B(y)}{e_k}|
\leq \norm{B(x)-B(y)}_{-\beta}\norm{e_k}_\beta\\
&=\lambda_k^\beta\norm{B(x)-B(y)}_{-\beta}
\leq [B]_{C_b^\theta(H;D(A^{-\beta}))}\lambda_k^\beta\norm{x-y}_H^\theta\,.
\end{align*}
Note also that for every $n\in\N_+$ and $x\in H$ we have
\[
g_k(x)=\begin{cases}
    &\scal{B_n(x)}{e_k},\qquad k\in\{1,\ldots,n\}, \\
    &\phantom{aaa}0, \phantom{aaaaa}\qquad k>n.
\end{cases}
\]
For $n\in\N_+$ and $k\in\{1,\ldots,n\}$,
we consider now the Kolmogorov equation \eqref{eq:kolm}
with forcing term $g_k$ instead of $g$, namely 
\begin{equation}\label{eq:kolm_k}
\bar c\lambda_k u_{n,k}(x)+L_n u_{n,k}(x)
=\langle B_n(x), Du_{n,k}(x)\rangle +g_k(x)\,, \quad x\in H_n\,,
\end{equation}
where the constant $\bar c$ is given as in \eqref{c}--\eqref{c'}.
We recall that $u_{n,k}\in C^2_b(H_n)$ for all $n\in\N$ and $k\in\{1,\ldots,n\}$.
Thanks to \eqref{ineq2}, 
Proposition~\ref{prop:est} and estimate \eqref{est:gk},
we infer that for every \lc{$\gamma\in[0,\beta]$,
for every $\gamma'\in\R$ with 
$\max\{0,\delta\}+
(1-\theta)\max\{0,\gamma'+\delta\}
+\theta\max\{0,\gamma'\}<\frac\theta2$,
and for every $\gamma''\in\R$ with 
$\max\{0,\gamma''+\delta\}+
(1-\theta)\max\{0,\delta\}
<\frac\theta2$,
}
there exist constants $\mathfrak L_1, \mathfrak L_2,
\mathfrak L_3>0$, independent of $\bar c$, $n$ and $k$, such that 
\begin{align}
\|u_{n,k}\|_{C^0_b(H_n)} &\leq \mathfrak L_1\lambda_k^\beta\,,\notag\\
\label{est1'}
\|A_n^{\gamma} Du_{n,k}\|_{C^0_b(H_n; H_n)}&\leq\frac{\mathfrak L_2\lambda_k^\beta}{(\bar c\lambda_k)^{\frac{1+\theta}{2}-(1-\theta)\lc{\max\{0,\gamma+\delta\}-\theta\gamma}}}\,,\\
\label{est2'}
\|D(A_n^{\gamma'}Du_{n,k})\|_{C^0_b(H_n; \cL(H_n,H_n))}
&\leq \lc{\frac{\mathfrak L_3\lambda_k^\beta}
{(\bar c\lambda_k)^{\lc{
\frac{\theta}{2}-
\max\{0,\delta\}-
(1-\theta)\max\{0,\gamma'+\delta\}
-\theta\max\{0,\gamma'\}
}}}\,,}\\
\label{est2'_bis}
\|D(A_n^{\gamma''}Du_{n,k})\|_{C^0_b(H_n; \cL(H_n,H_n))}
&\leq \lc{\frac{\mathfrak L_3\lambda_k^\beta}
{(\bar c\lambda_k)^{\lc{
\frac{\theta}{2}-
\max\{0,\gamma''+\delta\}-
(1-\theta)\max\{0,\delta\}
}}}\,.}
\end{align}
for every $n\in\N$ and $k\in\{1,\ldots,n\}$.
For every $n\in\N_+$, we define the vector-valued function
\[
u_n:H_n\to H_n\,, \qquad
u_n(x):=\sum_{k=1}^nu_{n,k}(x)e_k\,, \quad x\in H_n\,,
\]
so that $u_n\in C^2_b(H_n;H_n)$ and
\begin{align*}
Du_n(x)[h]&=\sum_{k=1}^n (Du_{n,k}(x), h)_H e_k \quad\forall\,x,h\in H_n\,,\\
D^2u_n(x)[h_1,h_2]&=\sum_{k=1}^n (D^2u_{n,k}(x)h_1, h_2)_H e_k \quad\forall\,x,h_1,h_2\in H_n\,.
\end{align*}

We are now ready to show the Lipschitz-continuity with respect to the initial data.
The It\^o formula for $u_{n,k}(X_n)$ yields,
    for all $t\geq0$, $\P$-almost surely,
    \begin{align*}
        &u_{n,k}(X_n(t)) 
        +\int_0^t L_nu_{n,k}(X_n(s))\,\d s\\
        &= u_{n,k}(P_nx)
        +\int_0^t (B_n(X(s)), Du_{n,k}(X_n(s)))_{H}\,\d s
        +\int_0^t(Du_{n,k}(X_n(s)), A^{-\delta}\,\d W(s))_H\,,
    \end{align*}
    and thanks to the Kolmogorov equation \eqref{eq:kolm_k} we have
    \[
    \int_0^t L_nu_{n,k}(X_n(s))\,\d s
    =\int_0^t (B_n(X_n(s)), Du_{n,k}(X_n(s)))_{H}\,\d s
    +\int_0^t \left[g_k(X_n(s))
    -\bar c\lambda_k u_{n,k}(X_n(s))\right]\,\d s\,.
    \]
    Putting everything together, we infer that 
    \begin{align*}
        \int_0^t g_k(X_n(s))\,\d s&=
        u_{n,k}(P_nx) -u_{n,k}(X_n(t))
        +\int_0^t (B_n(X(s))-B_n(X_n(s)), 
        Du_{n,k}(X_n(s)))_{H}\,\d s\\
        &+\bar c\lambda_k\int_0^t u_{n,k}(X_n(s))\,\d s
        +\int_0^t(Du_{n,k}(X_n(s)), A^{-\delta}\,\d W(s))_H\,.
    \end{align*}
    By multiplying the previous identity
    by $e_k$ and summing over $k=1,\ldots,n$, we get 
    \begin{align*}
        \int_0^t B_n(X_n(s))\,\d s&=
        u_{n}(P_nx) -u_{n}(X_n(t))
        +\int_0^t Du_n(X_n(s))[B_n(X(s))-B_n(X_n(s))] \,\d s\\
        &+\bar c \int_0^t A_n u_{n}(X_n(s))\,\d s
        +\int_0^tDu_{n}(X_n(s))[A^{-\delta}\,\d W(s)]\,,
    \end{align*}
    which yields, together with \eqref{eq-approssimata},
    \begin{align*}
        &X_n(t) + u_{n}(X_n(t))
        +\int_0^tA_n[X_n(s)+u_n(X_n(s))]\,\d s\\
&=P_nx + u_{n}(P_nx) + (\bar c+1)\int_0^t A_n u_{n}(X_n(s))\,\d s
+\int_0^t Du_n(X_n(s))[B_n(X(s))-B_n(X_n(s))] \,\d s\\
&\qquad+ \int_0^t[B_n(X(s))-B_n(X_n(s))]\,\d s
+\int_0^tA_n^{-\delta}\,\d W(s)
+\int_0^tDu_{n}(X_n(s))[A^{-\delta}\,\d W(s)]\,.
    \end{align*}
Now, passing to the mild formulation, we obtain 
\begin{align*}
X_n(t)&= e^{-tA_n}(P_nx+u_n(P_nx))-u_{n}(X_n(t))+(\overline{c}+1)A_n\int^t_0e^{-(t-s)A_n}u_n(X_n(s))\,\d s\\
&+\int_0^te^{-(t-s)A_n}Du_n(X_n(s))\left[B_n(X(s))-B_n(X_n(s))\right]\,\d s\\
&+\int_0^te^{-(t-s)A_n}\left[B_n(X(s))-B_n(X_n(s))\right]\,\d s\\
&+\int_0^te^{-(t-s)A_n}Du_n(X_n(s))\left[A^{-\delta}\d W(s)\right]
+\int_0^te^{-(t-s)A_n}A_n^{-\delta}\d W(s).
\end{align*}
By arguing in the same way for $Y$ we infer analogously that 
\begin{align*}
Y_n(t)&= e^{-tA_n}(P_ny+u_n(P_ny))-u_{n}(Y_n(t))+(\overline{c}+1)A_n\int^t_0e^{-(t-s)A_n}u_n(Y_n(s))\,\d s\\
&+\int_0^te^{-(t-s)A_n}Du_n(Y_n(s))\left[B_n(Y(s))-B_n(Y_n(s))\right]\,\d s\\
&+\int_0^te^{-(t-s)A_n}\left[B_n(Y(s))-B_n(Y_n(s))\right]\,\d s\\
&+\int_0^te^{-(t-s)A_n}Du_n(Y_n(s))\left[A^{-\delta}dW(s)\right]+\int_0^te^{-(t-s)A_n}A_n^{-\delta}\d W(s).
\end{align*}
By taking the difference $X_n-Y_n$
we finally deduce that 
 \begin{equation}
    \label{diff}
        X_n(t) - Y_n(t) = I_0+ \sum_{j=1}^7I_j(t)
\end{equation}
for every $t\in[0,T]$, where
\begin{align*}
I_0&:=e^{-tA_n}(P_n(x-y) + u_n(P_nx)-u_n(P_ny))\,,\\
I_1(t)&:=u_{n}(Y_n(t))- u_{n}(X_n(t))\,,\\
I_2(t)&:= (\bar c+1)A_n\int_0^t  e^{-(t-s)A_n}
[u_{n}(X_n(s))-u_{n}(Y_n(s))]\,\d s\,,\\
I_3(t)&:=\int_0^t e^{-(t-s)A_n} Du_n(X_n(s))[B_n(X(s))-B_n(X_n(s))] \,\d s\,,\\
I_4(t)&:=-\int_0^t e^{-(t-s)A_n} Du_n(Y_n(s))[B_n(Y_n(s))-B_n(Y(s))] \,\d s\,,\\
I_5(t)&:=\int_0^te^{-(t-s)A_n}[B_n(X(s))-B_n(X_n(s))]\,\d s\,,\\
I_6(t)&:=-\int_0^te^{-(t-s)A_n}[B_n(Y_n(s))-B_n(Y(s))]\,\d s \,,\\
I_7(t)&:=\int_0^te^{-(t-s)A_n}
(Du_{n}(X_n(s))-Du_{n}(Y_n(s)))[A^{-\delta}\,\d W(s)].
\end{align*}
We estimate now the left-hand side of \eqref{diff}
in the space $C^0([0,T]; L^2(\Omega; H))$
by analyzing all the terms $I_0,\ldots,I_7$ separately.
Let $T_0\in(0,T]$ be fixed, whose value will be specified later.

For the term $I_0$, we have 
\begin{align*}
I_0&=e^{-tA_n}\bigg(P_n(x-y)+
\int_0^1Du_n(P_ny+rP_n(x-y))[P_n(x-y)]\,\d r\bigg)\\
&=e^{-tA_n}\bigg(P_n(x-y)+\sum_{k=1}^n\int_0^1(Du_{n,k}(P_ny+rP_n(x-y)),  P_n(x-y))_He_k\,\d r\bigg)\,,
\end{align*}
so that, by using \eqref{est1'} with \lc{$\gamma=0$}
and the contraction property of $e^{-tA_n}$ and $P_n$, we infer that 
\begin{align}
\notag
&\|I_0\|^2_{C^0([0,T_0]; L^2(\Omega; H))}\\
\label{I0}
&\leq \left[1+\frac{\mathfrak L_2^2}
{{\bar c}^{1+\theta
-2(1-\theta)\lc{\max\{0,\delta\}}
}}
\sum_{k=1}^n
\frac1{\lambda_k^{1+\theta
-2\beta-2(1-\theta)\lc{\max\{0,\delta\}}
}
}\right]\|x-y\|_H^2\,.
\end{align}
Analogously, for the term $I_1$ we have
\begin{align*}
    I_1(t)&=\int_0^1Du_n(X_n(t)+r(Y_n(t)-X_n(t)))[Y_n(t)-X_n(t)]\,\d r\\
    &=\sum_{k=1}^n\int_0^1(Du_{n,k}(X_n(t)+r(Y_n(t)-X_n(t))), 
    Y_n(t)-X_n(t))_He_k\,\d r\,,
\end{align*}
so that, by using \eqref{est1'} with 
\lc{$\gamma=0$}, we infer that 
\begin{align*}
    &\|I_1\|^2_{C^0([0,T_0]; L^2(\Omega; H))}\\
    &\leq \frac{\mathfrak L_2^2}{\bar c^{1+\theta
-2(1-\theta)\lc{\max\{0,\delta\}}
    }}
    \left(\sum_{k=1}^n
    \frac1{\lambda_k^{1+\theta-2\beta
-2(1-\theta)\lc{\max\{0,\delta\}}
        }}\right)
        \|X_n-Y_n\|_{C^0([0,T_0]; L^2(\Omega; H))}^2\,.
\end{align*}
For the term $I_2$, for every $t\in[0,T_0]$, by \eqref{est1'} with \lc{$\gamma=0$} and 
the H\"older inequality we have
\begin{align*}
&\|I_2(t)\|^2_{L^2(\Omega;H)}\\
&=(\bar c +1)^2\sum_{k=1}^{n}\mathbb{E}
\left(\int_0^te^{-(t-s)A_n}A_n
\left[u_n(X_n(s))-u_n(Y_n(s))\right]\, \d s,e_k\right)_H^2\\
&=(\bar c +1)^2\sum_{k=1}^{n}
\mathbb{E}\left(\int_0^te^{-(t-s)\lambda_k}\lambda_k\left[u_{n,k}(X_n(s))-u_{n,k}(Y_n(s))\right]\,\d s\right)^2\\
&\leq \frac{\mathfrak{L}^2_2(\bar c +1)^2}{\bar{c}^{
1+\theta
-2(1-\theta)\lc{\max\{0,\tilde\alpha+\tilde \delta\}}
-\lc{2\theta\tilde\alpha}
}}\\
&\qquad\times\sum_{k=1}^{n}\mathbb{E}\left(\int_0^te^{-(t-s)\lambda_k}\lambda_k\frac{1}{\lambda_k^{\frac{1+\theta}{2}-{\tilde \beta}-(1-\theta)
\lc{\max\{-\tilde\alpha, \tilde\delta\}}
}}\norm{{\mathcal A}^{-{\tilde \alpha}}(X_n(s)-Y_n(s))}_{\mathcal H}\,\d s\right)^2 \notag \\
&\leq \frac{\mathfrak{L}^2_2(\bar c +1)^2}{\bar{c}^{1+\theta
-2(1-\theta)\lc{\max\{0,\delta\}}
}}\sum_{k=1}^{n}
\frac{1}{\lambda_k^{\theta-1-2\beta
-2(1-\theta)\lc{\max\{0,\delta\}}
}}\\
&\qquad\times\mathbb{E}
\left(\int_0^te^{-\frac{(t-s)\lambda_k}{2}}e^{-\frac{(t-s)\lambda_k}{2}}\norm{X_n(s)-Y_n(s)}_H\,\d s\right)^2\\
&\leq \frac{\mathfrak{L}^2_2(\bar c +1)^2}{\bar{c}^{1+\theta
-2(1-\theta)\lc{\max\{0,\delta\}}
}}\sum_{k=1}^{n}
\frac{1}{\lambda_k^{\theta-1-2\beta
-2(1-\theta)\lc{\max\{0,\delta\}}
}}\\
&\qquad\times
\int_0^te^{-(t-s)\lambda_k}\,\d s
\int_0^te^{-(t-s)\lambda_k}\mathbb{E}\norm{X_n(s)-Y_n(s)}^2_H\, \d s\,.
  \end{align*}  
Noting that for every $\varepsilon\in(0,1)$, by the H\"older inequality with
exponents $\frac{1}{1-\varepsilon}$ and $\frac{1}{\varepsilon}$,
for every $t\in [0,T]$ we have
\[
\int_0^te^{-(t-s)\lambda_k}\,\d s\leq T^{\varepsilon}\lambda_k^{\varepsilon-1}\,,
\]
we infer 
\begin{align*}
&\|I_2\|^2_{C^0([0,T_0]; L^2(\Omega; H))}\\
&\qquad\leq \frac{T_0^{2\varepsilon}\mathfrak{L}^2_2(\bar c +1)^2}{\bar{c}^{1+\theta-2(1-\theta)\lc{\max\{0,\delta\}}}}
\left(\sum_{k=1}^{n}\frac{1}{\lambda_k^{1+\theta-2\beta-2(1-\theta)\lc{\max\{0,\delta\}}-2\varepsilon}}\right)
\norm{X_n-Y_n}^2_{C^0([0,T]; L^2(\Omega; H))}\,.
  \end{align*}
For the term $I_3$, by \eqref{est1'} with $\gamma=\beta$ and assumption \ref{H6_0} for every $t\in[0,T_0]$ we get
\begin{align*}
&\|I_3(t)\|^2_{L^2(\Omega; H)}\\
&=
\sum_{k=1}^n\mathbb{E}
\left(\int_0^te^{-(t-s)A_n}Du_n(X_n(s))\left[B_n(X(s))-B_n(X_n(s))\right]\,\d s, e_k\right)_H^2\\
&=\sum_{k=1}^n\mathbb{E}\left(\int_0^te^{-(t-s)\lambda_k}
\left(Du_{n,k}(X_n(s)), B_n(X(s))-B_n(X_n(s))\right)_H\,\d s\right)^2\\
    &\leq \frac{C_B^2\mathfrak L_2^2}{\bar c^{1+\theta
        -2\beta-2(1-\theta)\delta}}\sum_{k=1}^n\frac{1}{\lambda_k^{1+\theta
        -4\beta-2(1-\theta)\delta}}
        \left(\int_0^t e^{-(t-s)\lambda_k}\mathbb{E}\norm{X_n(s)-X(s)}_H^\theta\,\d s\right)^2\\
    &\leq \frac{C_B^2\mathfrak L_2^2}{\bar c^{1+\theta
        -2\beta-2(1-\theta)\delta}}\sum_{k=1}^n\frac{1}{\lambda_k^{3+\theta
        -4\beta-2(1-\theta)\delta}}
        \sup_{s\in[0,T]}
        \mathbb{E}\norm{X_n(s)-X(s)}_H^{2\theta}\,,
\end{align*}
so that 
\begin{align*}
\|I_3\|^2_{C^0([0,T_0]; L^2(\Omega; H))}\leq 
\frac{C_B^2\mathfrak L_2^2}{\bar c^{1+\theta
        -2\beta-2(1-\theta)\delta}}
        \left(\sum_{k=1}^{n}\frac{1}{\lambda_k^{3+\theta
        -4\beta-2(1-\theta)\delta}}\right)
        \norm{X_n-X}_{C^0([0,T_0]; L^{2\theta}(\Omega; H))}^{2\theta}\,.
\end{align*}
In the same way, for $I_4$ we get
\begin{align*}
\|I_4\|^2_{C^0([0,T_0]; L^2(\Omega; H))}\leq 
\frac{C_B^2\mathfrak L_2^2}{\bar c^{1+\theta
        -2\beta-2(1-\theta)\delta}}
        \left(\sum_{k=1}^{n}\frac{1}{\lambda_k^{3+\theta
        -4\beta-2(1-\theta)\delta}}\right)
        \norm{Y_n-Y}_{C^0([0,T_0]; L^{2\theta}(\Omega; H))}^{2\theta}\,.
\end{align*}
For $I_5$, for every $t\in[0,T_0]$ we have
\begin{align*}
    I_5(t)
    =A_n^\beta\int_0^te^{-(t-s)A_n}
    A_n^{-\beta}[B_n(X(s))-B_n(X_n(s))]\,\d s\,,
\end{align*}
By \eqref{semiS} (with $\gamma=\beta$), \eqref{est:Bn2} and the H\"older inequality we get
\begin{align*}
\mathbb{E}\norm{I_5(t)}_H^2
    &\leq \mathbb{E}\left[\left(\int_0^t\norm{A_n^\beta e^{-(t-s)A_n}
    A_n^{-\beta}[B_n(X(s))-B_n(X_n(s))]}_H\,\d s\right)^2\right]\\
    &\leq \mathbb{E}\left[M_\beta^2C_B^2\left(\int_0^t\frac{1}{(t-s)^{\frac\beta2}}\frac{1}{(t-s)^{\frac\beta2}}\norm{X(s)-X_n(s)}^\theta_H\,\d s\right)^2\right]\\
    &\leq \mathbb{E}\left[M_\beta^2C_B^2\left(\int_0^t\frac{1}{(t-s)^{\beta}}\d s\right)\left(\int_0^t\frac{1}{(t-s)^{\beta}}\norm{X(s)-X_n(s)}^{2\theta}_H\,\d s\right)\right]\\
    &\leq \frac{M_\beta^2C_B^2 T^{1-\beta}}{1-\beta}
    \int_0^t\frac{1}{(t-s)^{\beta}}\mathbb{E}
    \norm{X(s)-X_n(s)}^{2\theta}_H\,\d s\\
     &\leq \frac{M_\beta^2C_B^2 T^{2-2\beta}}{(1-\beta)^2}
     \sup_{t\in [0,T]}\mathbb{E}\norm{X(t)-X_n(t)}^{2\theta}_H\,,
\end{align*}
so that
\begin{align*}
\sup_{t\in [0,T_0]}\mathbb{E}\norm{I_5(t)}_H^2
\leq \frac{M_\beta^2C_B^2 T^{2-2\beta}}{(1-\beta)^2}
     \sup_{t\in [0,T_0]}\mathbb{E}\norm{X(t)-X_n(t)}^{2\theta}_H\,.
\end{align*}
Analogously, we obtain for $I_6$ that 
\begin{align*}
\sup_{t\in [0,T_0]}\mathbb{E}\norm{I_6(t)}_H^2\leq \frac{M_\beta^2C_B^2 T^{2-2\beta}}{(1-\beta)^2}
     \sup_{t\in [0,T_0]}\mathbb{E}\norm{Y(t)-Y_n(t)}^{2\theta}_H\,.
\end{align*}
Eventually, for $I_7$, by the It\^o isometry for every $t\in[0,T_0]$ we get
\begin{align*}
    \|I_7(t)\|^2_{L^2(\Omega; H)}&=
    \E\sum_{k=1}^n\left(\int_0^te^{-(t-s)A_n}\left(Du_n(X_n(s))-Du_n(Y_n(s)\right)[A_n^{-\delta}\,\d W(s)], e_k\right)^2_H\\
&=\sum_{k=1}^n\mathbb{E}\left(\int_0^te^{-(t-s)\lambda_k}\left(Du_{n,k}(X_n(s))-Du_{n,k}(Y_n(s)), A_n^{-\delta}\,\d W(s)\right)_H\right)^2\\
&=\sum_{k=1}^n\sum_{j=1}^{n}\int_0^te^{-2(t-s)\lambda_k}\mathbb{E}
\left(Du_{n,k}(X_n(s))-Du_{n,k}(Y_n(s)),A_n^{-\delta}e_j\right)_H^2\,\d s\\
&=\sum_{k=1}^n\int_0^te^{-2(t-s)\lambda_k}\mathbb{E}
\norm{A_n^{-\delta}(Du_{n,k}(X_n(s))-Du_{n,k}(Y_n(s)))}_H^2\,\d s\,.
\end{align*}
By the estimate \eqref{est2'_bis} with $\gamma''=-\delta$
(\lc{which is admissible by \eqref{ineq2}}), we obtain
\begin{align*}
&\|I_7(t)\|^2_{L^2(\Omega; H)}\\
&\leq\frac{\mathfrak L_3^2}
{\bar c^{\theta-2(1-\theta)\lc{\max\{0,\delta\}}}}\sup_{s\in[0,T]}\mathbb{E}\norm{X_n(s)-Y_n(s)}_H^2\sum_{k=1}^n\frac{1}{\lambda_k^{\theta-2\beta-2(1-\theta)\lc{\max\{0,\delta\}}}}\int_0^te^{-2(t-s)\lambda_k}\,\d s\notag\\
&\leq\frac{\mathfrak L_3^2}
{2\bar c^{\theta-2(1-\theta)\lc{\max\{0,\delta\}}}}\sup_{s\in[0,T]}\mathbb{E}\norm{X_n(s)-Y_n(s)}_H^2\sum_{k=1}^n\frac{1}{\lambda_k^{1+\theta-2\beta-2(1-\theta)\lc{\max\{0,\delta\}}}}
\end{align*}
for every $t\in[0,T_0]$,
so that 
\begin{align}
\label{I7}
\|I_7\|^2_{C^0([0,T_0]; L^2(\Omega; H))}
\leq\frac{\mathfrak L_3^2}
{2\bar c^{\theta-2(1-\theta)\lc{\max\{0,\delta\}}}}
\left(\sum_{k=1}^n\frac{1}{\lambda_k^{1+\theta-2\beta-2(1-\theta)\lc{\max\{0,\delta\}}}}\right)
\norm{X_n-Y_n}^2_{C^0([0,T_0]; L^2(\Omega; H))}\,.
\end{align}
Now, choosing $\varepsilon \in (0,1)$ satisfying \eqref{trace} we have 
\[
  \sum_{k=1}^n
    \frac1{\lambda_k^{1+\theta
      -2\beta-2(1-\theta)\lc{\max\{0,\delta\}}-2\varepsilon}}\leq
  \sum_{k=1}^\infty
    \frac1{\lambda_k^{1+\theta
      -2\beta-2(1-\theta)\lc{\max\{0,\delta\}}-2\varepsilon}} =: s_1 <+\infty\,.
\]
Moreover, since the exponent $3+\theta-4\beta-2(1-\theta)\delta$
is greater than $1+\theta -2\beta-2(1-\theta)\lc{\max\{0,\delta\}}-2\varepsilon$, 
it also holds that 
\[
s_2:= \sum_{k=1}^\infty\frac{1}{\lambda_k^{3+\theta
-4\beta-2(1-\theta)\delta}}<+\infty\,.
\]
Hence, by taking \eqref{I0}--\eqref{I7} into account, we infer that
\begin{align*}
    &\|X_n-Y_n\|^2_{C^0([0,T_0]; L^2(\Omega; H))}\\
    &\leq \left[1+\frac{\mathfrak L_2^2 s_1}
{{\bar c}^{1+\theta
-2(1-\theta)\lc{\max\{0,\delta\}}}}
    \right]\|x-y\|_H^2\\
    &\quad+\left(
    \frac{\mathfrak L_2^2s_1}{\bar c^{1+\theta-2(1-\theta)\lc{\max\{0,\delta\}}}}
    +\frac{T_0^{2\varepsilon}\mathfrak{L}^2_2(\bar c +1)^2s_1}{\bar{c}^{1+\theta-2(1-\theta)\lc{\max\{0,\delta\}}}}
    +\frac{\mathfrak L_3^2s_1}
        {2\bar c^{\theta-2(1-\theta)\lc{\max\{0,\delta\}}}}
    \right)\|X_n-Y_n\|^2_{C^0([0,T_0]; L^2(\Omega; H))}\\
    &\quad+
    \left(\frac{C_B^2\mathfrak L_2^2s_2}{\bar c^{1+\theta
        -2\beta-2(1-\theta)\delta}}+ 
        \frac{M_\beta^2C_B^2 T^{2-2\beta}}{(1-\beta)^2}\right)
        \left[\norm{X_n-X}_{C^0([0,T_0]; L^{2\theta}(\Omega; H))}^{2\theta}
        +\norm{Y_n-Y}_{C^0([0,T_0]; L^{2\theta}(\Omega; H))}^{2\theta}
        \right].
\end{align*}
Now, we fix $\bar c$ satisfying \eqref{c}--\eqref{c'} large enough and $T_0\in(0,T]$ small enough so that 
\[
  \frac{\mathfrak L_2^2s_1}{\bar c^{1+\theta-2(1-\theta)\lc{\max\{0,\delta\}}}}+
  \frac{T_0^{2\varepsilon}\mathfrak{L}^2_2(\bar c +1)^2s_1}{\bar{c}^{1+\theta-2(1-\theta)\lc{\max\{0,\delta\}}}}+
  \frac{\mathfrak L_3^2s_1}
        {2\bar c^{\theta-2(1-\theta)\lc{\max\{0,\delta\}}}}<\frac12\,.
\]
Hence, by letting $n\to\infty$ and by exploiting Proposition~\ref{prop:Xn}
we get 
\begin{align*}
    &\|X-Y\|^2_{C^0([0,T_0]; L^2(\Omega; H))}
    \leq \left[1+\frac{\mathfrak L_2^2 s_1}
{{\bar c}^{1+\theta
-2(1-\theta)\lc{\max\{0,\delta\}}}}
    \right]\|x-y\|_H^2
    +\frac12\|X-Y\|^2_{C^0([0,T_0]; L^2(\Omega; H))}\,.
\end{align*}
By employing now a standard patching argument 
on sub-intervals $[T_0,2T_0]$, $[2T_0,3T_0]$, \ldots, until the final time $T$, the continuous dependence holds
also on the whole interval $[0,T]$.
This proves Theorem~\ref{th:3}
when $\alpha=s=0$ and under \ref{H6_0}.

\underline{\sc Step 2.} 
\lc{
We consider here the general case $\alpha\geq0$
and $s\in[-\beta,\alpha]$ with $(\alpha,s)\neq(0,0)$, and
$x,y\in H\cap D(A^s)$.
We perform exactly the same construction presented in 
{\sc Step 3} of the proof of Theorem~\ref{th:1},
namely we set
$\mathcal H:=D(A^\alpha)$, $\tilde e_k:=\lambda_k^{-\alpha}e_k$ for all $k\in\N_+$, 
$\mathcal A:=A_{|D(A^\alpha)}$, $\mathcal W:=A^{-\alpha}W$, 
${\tilde\beta}:=\alpha+\beta$, ${\tilde\delta}:=\delta-\alpha$. In this setting, we are going to apply the abstract result 
Proposition~\ref{th:uniq_path-app}:
to this end, we choose
${\tilde\alpha}:=\alpha-s$.
Let us now check that assumptions \ref{A0}--\ref{A5} are satisfied. 
First, recall that 
$D(\mathcal A^{\sigma}) = D(A^{\sigma+\alpha})$ 
for all $\sigma\in\mathbb R$,
and that 
$\mathcal A^{-\sigma}\in\cL^1(\mathcal H,\mathcal H)$ if and only if $A^{-\sigma}\in\cL^1(H,H)$.
Assumptions \ref{A0}--\ref{A1} hold by construction. 
Secondly, by \ref{H2} one has that 
$\tilde\beta=\alpha+\beta\in[0,1)$, 
$\tilde\delta=\delta-\alpha\in(-\frac12, \frac12]$, $\tilde\beta+\tilde\delta=\beta+\delta\in[0,\frac12]$ with $\tilde\beta+\tilde\delta<\frac12$ if $\tilde\delta<0$,
while by the fact that $s\in[-\beta,\alpha]$
it holds that $\tilde\alpha\in[0,1)$,
so that \ref{A2} is satisfied.
Assumption \ref{A3} follows from \ref{H3}, since $\mathcal A^{-(1+2\tilde\delta)+\eta}=\mathcal A^{-(1+2\delta)+2\alpha+\eta}$. 
Assumption \ref{A4} follows from $B\in C^\theta_b(D(A^\alpha);D(A^{-\beta}))$
and the fact that $D(\mathcal A^{-\tilde \beta})=D(A^{-\beta})$, and \ref{A5} coincides with \ref{H5}.
Furthermore, $(X,\mathcal W)$ and $(Y,\mathcal W)$ are weak mild solutions to \eqref{SDE-app} in the sense of Definition~\ref{def-sol-app} with initial data $x,y\in D(\mathcal A^{-\tilde\alpha})=D(A^s)$, respectively. 
Indeed, since $\mathcal A^{-\tilde\delta}\,\d\mathcal W=A^{-\delta}\,\d W$, one has $\mathcal W_{\mathcal A}=W_A$, so that \eqref{mild-app} coincides with \eqref{mild_al}. 
Moreover, $X, Y\in C^0((0,+\infty);\mathcal H)$ $\P$-a.s.~by Definition~\ref{def-sol}, and $X,Y\in C^0(\R_+;D(A^s))$ $\P$-a.s.~by Remark~\ref{rmk:conv2}.}

\lc{
Eventually, let us check conditions 
\eqref{trace_app}--\eqref{ineq_alfa_app}.
Clearly, $\tilde\alpha\leq\tilde\beta$ by construction.
Moreover, we note that 
since $\max\{-\tilde\alpha,\tilde\delta\}=\max\{s,\delta\}-\alpha$
\begin{align*}
-(1+\theta)+2\tilde\beta+2(1-\theta)\max\{-\tilde\alpha,\tilde\delta\}
&=-(1+\theta)+2\beta+2(1-\theta)\max\{s,\delta\}+2\theta\alpha\,,\\
\tilde\alpha+(1-\theta)\max\{-\tilde\alpha,\tilde\delta\}
&=\theta\alpha+(1-\theta)\max\{s,\delta\}-s\,,
\end{align*}
so that \eqref{trace_app}--\eqref{ineq_alfa_app}
are implied by assumptions \eqref{trace_s}--\eqref{ineq_s}.
Hence, the statement of Theorem~\ref{th:3} follows from Proposition~\ref{th:uniq_path-app}.}
\end{proof}

\begin{proof}[Proof of Theorem~\ref{th:2}]
\underline{\sc Step 1}. \lc{We consider here the case $B\in C^\theta_b(D(A^\alpha);D(A^{-\beta}))$  first. If $\delta\geq\alpha$,
    Theorem~\ref{th:3}
    with $s=\alpha$ applies
    thanks to \eqref{trace1}--\eqref{ineq2_al}.
    If $\delta<\alpha$,
    Theorem~\ref{th:3}
    with $s=\delta$ applies
    thanks to \eqref{trace1}, since in this case
    \eqref{ineq_s} is always true (as $\alpha-\delta<\frac12$).
    In both cases, choosing $x=y\in D(A^\alpha)$ in \eqref{cont_dep_s} yields that $X(t)=Y(t)$ $\P$-almost surely for every $t\in[0,T]$. Since $X$ and $Y$ have continuous trajectories with values in $H$ and $T>0$ is arbitrary, this implies that
    $\P\left(X(t)=Y(t),\;\;\forall\,t\geq0\right)=1$.}

\underline{\sc Step 2.} 
We show here how to remove the boundedness 
assumption on $B$.
For every $N\in\N$ we define the stopping time
\[
\tau_N:=T\wedge \inf\{t\in [0,T]\, :\, \norm{X(t)}_{D(A^\alpha)}>N\}\wedge \inf\{t\in [0,T]\, :\, \norm{Y(t)}_{D(A^\alpha)}>N\}.
\]
It is not difficult to show that $\tau_N\nearrow T$ $\mathbb{P}$-almost surely as $N\rightarrow\infty$. Let $\Pi_N$ be the orthogonal projection on the closed ball of $D(A^\alpha)$ with center in $0$ and radius $N$ and let $\overline B_N:D(A^\alpha)\rightarrow D(A^{-\beta})$ be defined by 
\[
\overline B_N(x):=B(\Pi_Nx)\,,\quad x\in D(A^\alpha)\,.
\]
Since $B \in C^\theta_{{\rm loc}}(D(A^\alpha);D(A^{-\beta}))$
by \ref{H4}, it holds that $\overline B_N \in C^\theta_{b}(D(A^\alpha);D(A^{-\beta}))$ for every $N\in\N$. 
Let now $N\in\mathbb N$ be fixed.
By following the localisation-extension arguments
in \cite[Section 4.2]{BOS} 
and exploiting the classical Yamada-Watanabe strong existence results (see \cite{On04}),
there exist two weak mild solutions
$\widetilde{X}_N$ and $\widetilde{Y}_N$
to \eqref{SDE} with $B$ replaced by $\overline B_N$, 
defined on the same probability space $(\Omega, \mathcal F, \P)$,
such that
\begin{equation*}
X=\widetilde{X}_N\,,\quad  
Y=\widetilde{Y}_N\,,\qquad\text{in } [\![0,\tau_N]\!]\,.
\end{equation*}
Since $\overline B_N \in C^\theta_{b}(D(A^\alpha);D(A^{-\beta}))$ for every $N\in\N$, by the pathwise uniqueness result of 
Theorem~\ref{th:2} just proved in {\sc Step 1}
\lc{for $B\in C^\theta_b(D(A^\alpha);D(A^{-\beta}))$}, we infer that 
\begin{align*}
\mathbb P(\widetilde X_N(t)=\widetilde Y_N(t),
\quad\forall\,t\in[0,T])=1\,,
\end{align*}
so that 
\[
  X=Y \quad\text{in } [\![0,\tau_N]\!] \quad\forall\,N\in\mathbb  N\,.
\]
By setting $\Omega_N:=\{\omega\in \Omega: X(t)=Y(t),\;\,\forall\,t\in[0,\tau_N(\omega)]\}$, we have then that 
$\P(\Omega_N)=1$ for every $N\in\mathbb N$.
Since $\Omega_N\supseteq \Omega_{N+1}$ for every $N\in\N$,  we get 
\[
\Omega_N\downarrow \{\omega\in\Omega\, :\, X(\omega,t)=Y(\omega,t),\quad \forall\, t\in [0,T]\}
\]
and we conclude by the continuity of $\mathbb{P}$.
This proves the pathwise uniqueness result of Theorem~\ref{th:2}
with initial data in $D(A^\alpha)$, without any boundedness assumption.

\underline{\sc Step 3.} 
\lc{We consider here the case $B\in C^\theta_b(D(A^\alpha);D(A^{-\beta}))$ and initial data in $H$. If $\delta\geq0$,
    Theorem~\ref{th:3}
    with $s=0$ applies
    thanks to \eqref{trace1} and \eqref{ineq3_al}.
    If $\delta<0$,
    Theorem~\ref{th:3}
    with $s=\delta$ applies
    thanks to \eqref{trace1}, since in this case
    \eqref{ineq_s} is always true (as $\alpha-\delta<\frac12$).
    In both cases, pathwise uniqueness follows by choosing $x=y\in H$ in \eqref{cont_dep_s} and arguing as in {\sc Step 1}.}
\end{proof}

\section{Examples}
\label{Exam}
We present here some applications of the main theorems
to some classes of explicit examples.
We also briefly discuss the extensions with respect to the 
available state of the art.

Throughout this section, let $\OO$ be a smooth bounded domain in $\mathbb R^d$, $d\geq1$. We denote by $A_D$ and $A_N$
the negative Laplace operators on $L^2(\OO)$, 
with homogeneous Dirichlet and Neumann 
boundary conditions, respectively. 

Preliminarily, we note that if $A$ is either $A_D$ or $\operatorname{I}+A_N$,
it holds that $A^{-\sigma}\in\cL^1(L^2(\OO), L^2(\OO))$
if and only if $\sigma>\frac d2$. Indeed, given a complete orthonormal system $(e_k)_k$ of $L^2(\OO)$
made of eigenvectors of $A$ with respective eigenvalues 
$(\lambda_k)_k$, it holds that $\lambda_k\sim k^{\frac2d}$
as $k\to\infty$, so that
\begin{align*}
    \sum_{k=1}^\infty(A^{-\sigma}e_k, e_k)_{L^2(\OO)}
    =\sum_{k=1}^\infty\lambda_k^{-\sigma}
    \sim\sum_{k=1}^\infty k^{-\frac{2\sigma}d}<+\infty
\end{align*}
if and only if $\frac{2\sigma}d>1$.

Furthermore, for $s\geq0$ we set
\[
V_{2s}:=D(A_D^s)\,,
\]
while for $r<0$ we set
\[
  V_{2r}:=(V_{-2r})^*\,.
\]
Let us recall that the spaces $V_{2s}$, $s\in\mathbb R$, 
can be characterised in terms of suitable Sobolev spaces.
For example, one has $V_2=D(A_D)=H^2(\OO)\cap H^1_0(\OO)$ and
\[
  V_{2s}=
  \begin{cases}
H^{2s}(\OO)\quad&\text{if } s\in\left[0,\frac14\right)\,,\\
H^{\frac12}_{00}(\OO)\quad&\text{if } s=\frac14\,,\\
H^{2s}_0(\OO) \quad&\text{if } s\in\left(\frac14,\frac12\right]\, ,
  \end{cases}
\]
where $H^{\frac12}_{00}(\OO)$ is the Lions--Magenes space. This follows from classical interpolation 
theory in Sobolev spaces, see for instance~\cite[Chap.~1]{LM}.
We recall a technical result on multiplication of functions and distributions.
\begin{lemma}\label{Lemma-immersioni}
Let $\eta\in\left[\frac14,\frac12\right]$
and $\sigma\geq \frac12-\eta$ be such that $2\eta+\sigma>\frac{d+2}{4}$. Then, for every $f\in V_{2\eta}$ and $g\in V_{2\eta-1}$, 
the product $fg$ is well-defined as an element of 
$V_{-2\sigma}$, and
there exists a positive constant $C$ such that 
\begin{align}\label{Besov-mol_n}
\|fg\|_{V_{-2\sigma}}\leq C\|f\|_{V_{2\eta}}\|g\|_{V_{2\eta-1}}
\qquad   \forall\,f\in V_{2\eta}\,,\quad\forall\, g\in V_{2\eta-1}\,.  
\end{align}
\end{lemma}
\begin{proof}
Let $w\in C^\infty_c(\OO)$ be an arbitrary test function.
Since $C^\infty_c(\OO)$ is dense both in $V_{2\eta}$
and $V_{2\eta-1}$, it is not restrictive to assume first that $f,g\in C^\infty_c(\OO)$. We have that
\begin{align*}
|\langle fg,w\rangle_{V_{-2\sigma}, V_{2\sigma}}|
=\left|\int_\OO fgw\right|
= |\langle g,fw\rangle_{V_{2\eta-1}, V_{1-2\eta}}|
\leq \|g\|_{V_{2\eta-1}}\|fw\|_{V_{1-2\eta}}\,.
\end{align*}
We prove now that there exists $C>0$, 
independent of $f,w$, such that
\begin{align}\label{fw}
\|fw\|_{V_{1-2\eta}}
\leq C\|f\|_{V_{2\eta}}\|w\|_{V_{2\sigma}}\,.
\end{align}
If $\eta\in(\frac14, \frac12]$, then one has $V_{1-2\eta}=H^{1-2\eta}(\OO)$ and $V_{2\eta}=H^{2\eta}_0(\OO)$.
Hence, recalling also that the $H^{2\eta}_0(\OO)$-norm
is equivalent to the restriction of the $H^{2\eta}(\OO)$-norm
to $H^{2\eta}_0(\OO)$, 
one can apply the pointwise multiplication result in 
\cite[Thm.~7.4]{Beh-Hol2021} 
with $s_1=2\eta$, $s_2=2\sigma$, $s=1-2\eta$ and $p_1=p_2=p=2$:
this proves exactly \eqref{fw} for some $C$ independent of $f$ and $w$.
If $\eta=\frac14$ (hence $\sigma>\frac d4\geq\frac14$),
we have that $V_{2\eta}=V_{1-2\eta}=H^{\frac12}_{00}(\OO)$
and $V_{2\sigma}=H^{2\sigma}_0(\OO)$.
We also recall that the $0$-extension operator 
outside $\OO$ is linear and continuous as a map
$L_0:H^{\frac12}_{00}(\OO)\to H^{\frac12}(\mathbb R^d)$
and $L_0:H^{2\sigma}_{0}(\OO)\to H^{2\sigma}(\mathbb R^d)$.
Consequently, we have that 
\begin{align*}
    \|fw\|_{V_{1-2\eta}}=\|fw\|_{H^{\frac12}_{00}(\mathbb R^d)}
    =\|L_0(fw)\|_{H^{\frac12}(\mathbb R^d)}=
    \|(L_0f)(L_0w)\|_{H^{\frac12}(\mathbb R^d)}\,.
\end{align*}
By applying the pointwise multiplication result in 
\cite[Thm.~7.3]{Beh-Hol2021} 
with $s_1=2\eta=\frac12$, $s_2=2\sigma$, $s=1-2\eta=\frac12$ and $p_1=p_2=p=2$, we get
\[
\|(L_0f)(L_0w)\|_{H^{\frac12}(\R^d)}\leq 
C\|L_0f\|_{H^{\frac12}(\mathbb R^d)}
\|L_0w\|_{H^{2\sigma}(\R^d)}
\leq C\|f\|_{H^{\frac12}_{00}(\OO)}\|w\|_{H^{2\sigma}_0(\OO)}\,,
\]
where $C$ only depends on $L_0$. By putting the two inequalities together,  we infer that \eqref{fw} holds also in the case $\eta=\frac14$.
The inequality \eqref{fw} implies that 
\begin{align*}
|\langle fg,w\rangle_{V_{-2\sigma}, V_{2\sigma}}|
\leq C\|f\|_{V_{2\eta}}\|g\|_{V_{2\eta-1}}
\|w\|_{V_{2\sigma}}
\end{align*}
for every $f,g,w\in C^\infty_c(\OO)$.
By density of $C^\infty_c(\OO)$ in $V_{2\sigma}$,
this shows that 
\[
  \|fg\|_{V_{-2\sigma}}\leq C\|f\|_{V_{2\eta}}\|g\|_{V_{2\eta-1}}
  \qquad\forall\,f,g\in C^\infty_c(\OO)\,.
\]
Since $C^\infty_c(\OO)$ is dense both in $V_{2\eta}$
and $V_{2\eta-1}$, this implies that the product $fg$
is well-defined in $V_{-2\sigma}$
for $f\in V_{2\eta}$ and $g\in V_{2\eta-1}$, 
and the thesis follows.
\end{proof}

We discuss now the main classes of examples. 
For sake of clarity, at the end of each example 
we add a table highlighting the most significant choices 
of the coefficients for which uniqueness holds. 
We use the concise notation $\sigma^+$ (or $\sigma^-$)
to indicate a choice of the parameters 
arbitrarily close to the value $\sigma$ from above (resp.~from below).

\subsection{Heat equations with perturbations}
\label{ssec:heat}
The paradigmatic example of semilinear stochastic PDE is given by 
nonlinear perturbations of the heat equation in the form
  \begin{align*}
  \d X - \Delta X\,\d t = (-\Delta)^\beta F((-\Delta)^\alpha X)\,\d t 
  + (-\Delta)^{-\delta}\,\d W 
  \quad&\text{in } (0,T)\times\OO\,,\\
  X=0\quad&\text{in } (0,T)\times\partial\OO,\\
  X(0)=x \quad&\text{in } \OO\,,
  \end{align*}
in the case of Dirichlet boundary conditions, where 
$F:\mathbb R\to \mathbb R$ is $\theta$-H\"older continuous
for some $\theta\in(0,1)$,
$x$ is a given initial datum, and 
\lc{$\alpha,\beta\in[0,1)$, $\delta\in(-\frac12+\alpha, \frac12]$, 
and $\beta+\delta\in[0,\frac12]$
with $\delta\geq\alpha$ if $\beta+\delta=\frac12$.}

We can rewrite such SPDE as evolution equation on 
$H=L^2(\OO)$ in the abstract form, 
\begin{align}\label{SDE-new}
\d X +A_DX \,\d t=B(X)\,\d t
+A_D^{-\delta}\,\d W\,, \qquad X(0)=x\,,    
\end{align}
where $W$ is an $H$-cylindrical Wiener process, $x\in H$, and
the nonlinear operator $B: D(A_D^\alpha)\to D(A_D^{-\beta})$ is defined as
\[
  \langle B(\varphi),\psi\rangle_{D(A_D^{-\beta}), D(A_D^\beta)}:=
  \int_\OO F(A_D^\alpha\varphi(\xi))
  A_D^\beta\psi(\xi)\,\d\xi,
  \qquad\varphi\in D(A_D^\alpha)\,,\quad\psi\in D(A_D^{\beta})\,.
\]
Note that $B$ is well-defined in $C^\theta(D(A_D^\alpha), D(A_D^{-\beta}))$ since $F\in C^\theta(\mathbb R, \mathbb R)$, hence also linearly bounded. 

Let us check the main assumptions of Theorems~\ref{th:1}, \ref{th:2}, and \ref{th:3}. Clearly, \ref{H0}--\ref{H2} and \ref{H4} are satisfied.
As far as assumption \ref{H3} is concerned, we need that $1+2\delta-2\alpha>\frac d2$, i.e.~$\delta> \frac{d}{4} +\alpha- \frac{1}{2}$. Since $\delta\leq\frac12-\beta$, this yields also the condition $\alpha+\beta<1-\frac d4$.
This means that for every $\alpha,\beta$ such that
$\alpha+\beta<1-\frac d4$,
we can choose any
$\delta\in(\frac{d}{4} +\alpha- \frac{1}{2},\frac12-\beta]$ \lc{with $\delta\geq-\beta$, and with $\delta\geq\alpha$ if $\delta=\frac12-\beta$}.
Furthermore, in the critical case $\beta+\delta=\frac12$
a further smallness condition is required on $F$ 
in order to satisfy assumption \ref{H5}.
Eventually, conditions \eqref{trace1}--\eqref{ineq2_al} are satisfied 
provided that $1+\theta-2\beta-2(1-\theta)\delta-2\theta\alpha>\frac d2$ and 
$(1-\theta)\delta<(1-\theta)\alpha+\frac\theta2$, i.e.
\[
(1-\theta)\delta<\min\left\{
\frac12-\frac d4-\beta, 
\alpha\right\}
+ \frac\theta2-\theta\alpha\,.
\]
In order for \eqref{trace1}--\eqref{ineq2_al} to be compatible with \ref{H3} we need that 
\[
  (1-\theta)\left(\frac d4 + \alpha - \frac12\right)
  <\min\left\{
\frac12-\frac d4-\beta, 
\alpha\right\}
+ \frac\theta2-\theta\alpha
\]
i.e.~$\alpha+\beta<1-\frac d4\left(2-\theta\right)$
and $\theta>1-\frac 2d$\lc{; note that, since $\alpha+\beta\geq0$, the first condition also yields $\theta>2-\frac4d$, which is stronger than $\theta>1-\frac2d$ if $d\geq3$}.
Analogously, \eqref{ineq3_al} is satisfied if also 
$(1-\theta)\delta<\frac\theta2-\theta\alpha$, which requires that $(1-\theta)\left(\frac d4 + \alpha - \frac12\right)
<\frac\theta2-\theta\alpha$, i.e.~$\alpha<\frac12-\frac d4(1-\theta)$.
The results contained in Theorems~\ref{th:1}, \ref{th:2}, and \ref{th:3} can be then collected as follows.

{\sc Dimension $d=1$.} 
\begin{itemize}
\item Weak uniqueness holds for initial data in $D(A_D^\alpha)$
for 
\[
\alpha+\beta\in\left[0,\frac34\right)\,, 
\qquad\theta\in(0,1)\,,
\]
and for every $\delta\in\left(\alpha-\frac14,\frac12-\beta\right]$, \lc{with $\delta\geq\alpha$
if $\delta+\beta=\frac12$.}
If $F$ is bounded, weak uniqueness holds also 
for initial data in $H$ for the same range 
of parameters.
\item Pathwise uniqueness holds for initial data in $D(A_D^\alpha)$ for 
\[
\alpha+\beta\in\left[0,\frac12+\frac\theta4\right)\,, 
\qquad\theta\in(0,1)\,,
\]
and for every $\delta\in\left(\alpha-\frac14,
\frac12-\beta\right]\cap 
\left(\alpha-\frac14,
\frac1{1-\theta}(\min\{\frac14-\beta, \alpha\}+\frac\theta2-\theta\alpha)\right)$ with \lc{$\delta\geq\alpha$ and} the smallness condition on $F$
if $\delta+\beta=\frac12$. 
\item Pathwise uniqueness 
holds for initial data in $H$ if $F$ is bounded, for 
\[
 \alpha\in\left[0,\frac14+\frac\theta4\right)\,,\qquad
\alpha+\beta\in\left[0,\frac12+\frac\theta4\right)\,, 
\qquad\theta\in(0,1)\,,
\]
and for every $\delta\in\left(\alpha-\frac14,
\frac12-\beta\right]\cap 
\left(\alpha-\frac14,
\frac1{1-\theta}(\min\{\frac14-\beta, 0 \}+\frac\theta2-\theta\alpha)\right)$, with \lc{$\delta\geq\alpha$ and} the smallness condition on $F$
if $\delta+\beta=\frac12$. 
\end{itemize}

{\sc Dimension $d=2$.}

\begin{itemize}
\item Weak uniqueness holds for initial data in $D(A_D^\alpha)$, for 
\[
\alpha+\beta\in\left[0,\frac12\right)\,, 
\qquad\theta\in(0,1)\,,
\]
and for every $\delta\in\left(\alpha,\frac12-\beta\right]$.
If $F$ is bounded, weak uniqueness holds also 
for initial data in $H$ for the same range 
of parameters.
\item Pathwise uniqueness holds for initial data in $D(A_D^\alpha)$, for
\[
\alpha+\beta\in\left[0,\frac\theta2\right)\,, 
\qquad\theta\in(0,1)\,,
\]
and for every $\delta\in\left(\alpha,
\frac12-\beta\right]\cap 
\left(\alpha,
\frac1{1-\theta}(\frac\theta2-\beta-\theta\alpha)\right)$, with the smallness condition on $F$
if $\delta+\beta=\frac12$.  If $F$ is bounded, then pathwise uniqueness holds also for initial data in $H$. 

\end{itemize}

{\sc Dimension $d=3$.}

\begin{itemize}
\item Weak uniqueness holds for initial data in $D(A_D^\alpha)$, for 
\[
\alpha+\beta\in\left[0,\frac14\right)\,, 
\qquad\theta\in(0,1)\,,
\]
and for every $\delta\in\left(\alpha+\frac14,\frac12-\beta\right]$. If $F$ is bounded, weak uniqueness holds also 
for initial data in $H$ for the same range 
of parameters.
\item Pathwise uniqueness holds for initial data in $D(A_D^\alpha)$, for
\[
\alpha+\beta\in\left[0,\frac34\theta-\frac12\right)\,, 
\qquad\theta\in\lc{\left(\frac23,1\right)\,,}
\]
and for every $\delta\in\left(\alpha+\frac14,
\frac12-\beta\right]\cap 
\left(\alpha+\frac14,
\frac1{1-\theta}(\frac\theta2-\frac14-\beta-\theta\alpha)\right)$, with the smallness condition on $F$
if $\delta+\beta=\frac12$.  If $F$ is bounded, then pathwise uniqueness holds also for initial data in $H$. 
\end{itemize}

\begin{remark}\label{rmk:but}
\lc{We underline that, in the non-singular case 
$\alpha=\beta=0$ and $B$ bounded,
the techniques developed in \cite{But} allow to achieve strong existence and pathwise uniqueness up to dimension $7$, with less restrictive conditions on $\theta$ than ours. On the other hand, the approach of the present paper allows to deal with perturbations in differential form, i.e.~$\alpha,\beta>0$, which are not covered by \cite{But}.} 
\end{remark}

\subsection{Fractional heat equations with perturbations}
\label{ssec:heat_fr}
A large class of nonlocal PDEs associated to 
fractional diffusion, with possibly nonlocal boundary conditions, can be framed into the abstract evolution equation
\begin{equation}
\label{SDE-new-fr}
\d X +A_D^\gamma X \,\d t=A_D^{\nu}F(A_D^\mu X)\,\d t
+A_D^{-\rho}\,\d W\,, \qquad X(0)=x\,,    
\end{equation}
where $\gamma>0$, $\mu, \nu\geq0$, and $\rho\in\R$ are given constants. By using the notation and arguing as in Subsection~\ref{ssec:heat}, we can find the range of values for such parameters so that uniqueness holds. To this end, we recast the equation in the setting of Theorems~\ref{th:1}, \ref{th:2}, and \ref{th:3} by considering the operator $A:=A_D^\gamma$ and with the choices of the coefficients
$\alpha=\frac\mu\gamma$, $\beta=\frac\nu\gamma$, and 
$\delta=\frac\rho\gamma$. \lc{The conditions $\alpha,\beta\in[0,1)$, $\delta\in\left(-\frac12+\alpha, \frac12\right]$, and $\beta+\delta\in\left[0,\frac12\right]$ 
with $\delta\geq\alpha$ if $\beta+\delta=\frac12$
translate as $\mu,\nu\in[0,\gamma)$, $\rho\in\left(-\frac\gamma2+\mu,\frac\gamma2\right]$,
$\nu+\rho\in\left[0,\frac\gamma2\right]$
with $\rho\geq\mu$ if $\nu+\rho=\frac\gamma2$.
}

Assumption \ref{H3} is satisfied provided that
$\gamma(1+2\delta-2\alpha)>\frac d2$, 
i.e., $\rho> \frac{d}{4} +\mu- \frac{\gamma}{2}$. Since $\rho\leq\frac\gamma2-\nu$, this yields also the condition $\mu+\nu<\gamma-\frac d4$.
This means that for every $\mu,\nu$ such that 
$\mu+\nu<\gamma-\frac d4$, 
we can choose any 
$\rho\in\left(\frac{d}{4} +\mu- \frac{\gamma}{2},\frac\gamma2-\nu\right]$ \lc{with $\rho\geq-\nu$, and with $\rho\geq\mu$ if $\rho=\frac\gamma2-\nu$ (here and below, all the ranges for $\rho$ are tacitly intersected with $[-\nu,+\infty)$, according to \ref{H2})}.
Furthermore, 
conditions \eqref{trace1}--\eqref{ineq2_al} are satisfied 
provided that $\gamma(1+\theta-2\beta-2(1-\theta)\delta-2\theta\alpha)>\frac d2$ and 
$(1-\theta)\delta<(1-\theta)\alpha+\frac\theta2$, i.e.,
\[
(1-\theta)\rho<\min\left\{
\frac\gamma2-\frac d4-\nu, 
\mu\right\}
+ \frac{\gamma\theta}2-\theta\mu\,.
\]
In order for \eqref{trace1}--\eqref{ineq2_al} to be compatible with \ref{H3} we need that 
\[
(1-\theta)\left(\frac d4 + \mu - \frac\gamma2\right)<\min\left\{
\frac\gamma2-\frac d4-\nu, 
\mu\right\}+ \frac{\gamma\theta}2-\theta\mu,
\]
i.e., $\mu+\nu<\gamma-\frac d4\left(2-\theta\right)$
and $\theta>1-\frac{2\gamma}d$.
Since $\mu+\nu>0$, the condition $\gamma>\frac d4(2-\theta)$
yields $\theta>2-\frac{4\gamma}d$, which is stronger than 
$\theta>1-\frac{2\gamma}d$. Moreover, since $\theta<1$
one is forced to choose $2-\frac{4\gamma}d<1$, i.e.~$\gamma>\frac d4$.
Analogously, \eqref{ineq3_al} is satisfied if also 
$(1-\theta)\delta<\frac\theta2-\theta\alpha$, which requires that $(1-\theta)\left(\frac d4 + \mu - \frac\gamma2\right)
<\frac{\gamma\theta}2-\theta\mu$, i.e.~$\mu<\frac\gamma2-\frac d4(1-\theta)$: hence, as before, this is compatible with
$\mu+\nu<\gamma-\frac d4\left(2-\theta\right)$
if $\gamma>\frac d2$.
The results contained in Theorems~\ref{th:1}, \ref{th:2}, and \ref{th:3} can be then collected as follows.

\begin{itemize}
\item Weak uniqueness holds for initial data in 
$D(A_D^\mu)$ for 
\[ \gamma > \frac{d}{4}\,,
\qquad\theta\in(0,1)\,,
\qquad 
\mu+\nu\in\left[0,\gamma - \frac d4\right)\,, 
\]
and for every $\rho\in\left(\frac d4 + \mu-\frac \gamma2,\frac\gamma2-\nu\right]$, 
\lc{with $\rho\geq\mu$ if $\nu+\rho=\frac\gamma2$.}
If $F$ is bounded, weak uniqueness holds also 
for initial data in $H$ for the same range 
of parameters.
\item Pathwise uniqueness holds for initial data in $D(A_D^\mu)$, for 
\[ \gamma > \frac{d}{4}\,, \qquad
\qquad\theta\in\left(\max\left\{0, 2-\frac{4\gamma}d\right\},1\right)\,,
\qquad
\mu+\nu\in\left[0,\gamma - \frac{d}{4}(2-\theta)\right)\,, 
\]
and for every $\rho\in\left(\frac{d}{4} +\mu- \frac{\gamma}{2},\frac\gamma2-\nu\right]\cap 
\left(\frac{d}{4} +\mu- \frac{\gamma}{2},
\frac1{1-\theta}(\min\left\{
\frac\gamma2-\frac d4-\nu, 
\mu\right\}
+ \frac{\gamma\theta}2-\theta\mu)\right)$, 
\lc{with $\rho\geq\mu$ and}
the smallness condition on $F$
if $\rho+\nu=\frac\gamma2$. 
\item Pathwise uniqueness 
holds for initial data in $H$ if $F$ is bounded, for
\begin{alignat*}{2} 
&\gamma > \frac d2\,, \qquad
&&\theta\in\left(\max\left\{0, 2-\frac{4\gamma}d\right\},1\right)\,,\\
&\mu\in\left[0,
\min\left\{
\frac\gamma2- \frac{d}{4}(1-\theta),
\gamma - \frac{d}{4}(2-\theta)
\right\}
\right)\,,\qquad
&&\mu+\nu\in\left[0,\gamma - \frac{d}{4}(2-\theta)\right)\,, 
\end{alignat*}
and for every $\rho\in\left(\frac{d}{4} +\mu- \frac{\gamma}{2},\frac\gamma2-\nu\right]\cap 
\left(\frac{d}{4} +\mu- \frac{\gamma}{2},
\frac1{1-\theta}(\min\left\{
\frac\gamma2-\frac d4-\nu, 
0\right\}
+ \frac{\gamma\theta}2-\theta\mu)\right)$, 
\lc{with $\rho\geq\mu$ and}
the smallness condition on $F$
if $\rho+\nu=\frac\gamma2$.
\end{itemize}

Notice that the range in the pathwise uniqueness tends to the ones of the weak uniqueness case as $\theta$ goes to $1$.
Furthermore, we stress that the freedom of choice on the values of $\gamma$ allows us to recover uniqueness in higher spatial dimension. 
From the applications point of view, given an explicit equation for which the dimension $d$, the differential order of the nonlinearity $\mu+ \nu$ and the H\"older regularity $\theta$ are fixed, one can obtain the minimal fractional order $\gamma$ for which uniqueness holds. In fluid-dynamical models this is typically related to the tuning of the hyperviscosity coefficient and will be discussed in detail in the next subsections. 

We provide three tables highlighting some notable choices
of the coefficients for which one has 
weak uniqueness, pathwise uniqueness with $\theta\cong1$,
and pathwise uniqueness with $\theta\cong0$.
For each table, we provide three different scenarios,
corresponding to: 
the minimal admissible value of $\gamma$,
the classical choice $\gamma=1$ of the heat equation
of Subsection~\ref{ssec:heat},
and the case $\rho=0$ of space-time white noise.

{\small
\begin{table}[!h]
\caption{Range of coefficients for weak uniqueness}
\label{tab0w}
\begin{tabular}{ |c||c|c|c|c|c|}
 \hline
 \multicolumn{6}{|c|}{Weak uniqueness} \\
 \hline\hline
 Dimension& 
 Hyperviscosity $\gamma$&
 H\"older coeff.~$\theta$&
 Coeff.~$\mu$&
 Coeff.~$\nu$&
 Noise colour $\rho$\\[3pt]
 \hline
 \multirow{3}{*}{$d=1$}   & $\frac{1}{4}^+$ &$(0,1)$ & 
 $0$ & $0^+$ & $\frac18^-$\\[3pt]
 & $1$ &$(0,1)$ & 
 $0$ & $\frac34^-$ & $(-\frac{1}{4})^+$\\[3pt]
 & $\frac{1}{2}^+$ &$(0,1)$ & 
 $0$ & $\frac{1}{4}^+$ & $0$\\[3pt]
 \hline
 \multirow{3}{*}{$d=2$} & $\frac{1}{2}^+$ &$(0,1)$ & 
 $0$ & $0^+$ & $\frac14^-$\\[3pt]
 &   $1$  & $(0,1)$   &$0$
 & $\frac12^-$ & $0^+$\\[3pt]
 &   
 $1^+$  & $(0,1)$   &$0$
 & $\frac{1}{2}^+$ & $0$\\[3pt]
 \hline
 \multirow{3}{*}{$d=3$} & $\frac{3}{4}^+$ &$(0,1)$ & 
 $0$ & $0^+$ & $\frac38^-$\\[3pt]
 &  $1$  & $(0,1)$   &$0$
 & $\frac14^-$ & $\frac14^+$\\[3pt]
 &   
 $\frac32^+$  & $(0,1)$   &$0$
 & $\frac34^+$ & $0$\\[3pt]
 \hline
\end{tabular}
\end{table}

\begin{table}[!h]
\caption{Range of coefficients for pathwise uniqueness: case $\theta\cong1$}
\label{tab5}
\begin{tabular}{ |c||c|c|c|c|c|}
 \hline
 \multicolumn{6}{|c|}{Pathwise uniqueness} \\
 \hline\hline
 Dimension& 
 Hyperviscosity $\gamma$&
 H\"older coeff.~$\theta$&
 Coeff.~$\mu$&
 Coeff.~$\nu$&
 Noise colour $\rho$\\[3pt]
 \hline
 \multirow{3}{*}{$d=1$}   & $\frac{1}{4}^+$ &$1^-$ & 
 $0$ & $0^+$ & $\frac18^-$\\[3pt]
 & $1$ &$1^-$ & 
 $0$ & $\frac34^-$ & $(-\frac{1}{4})^+$\\[3pt]
 & $\frac{1}{2}^+$ &$1^-$ & 
 $0$ & $\frac{1}{4}^+$ & $0$\\[3pt]
 \hline
 \multirow{3}{*}{$d=2$} & $\frac{1}{2}^+$ &$1^-$ & 
 $0$ & $0^+$ & $\frac14^-$\\[3pt]
 &   $1$  & $1^-$   &$0$
 & $\frac12^-$ & $0^+$\\[3pt]
 &   
 $1^+$  & $1^-$   &$0$
 & $\frac{1}{2}^+$ & $0$\\[3pt]
 \hline
 \multirow{3}{*}{$d=3$} & $\frac{3}{4}^+$ &$1^-$ & 
 $0$ & $0^+$ & $\frac38^-$\\[3pt]
 &  $1$  & $1^-$   &$0$
 & $\frac14^-$ & $\frac14^+$\\[3pt]
 &   
 $\frac32^+$  & $1^-$   &$0$
 & $\frac34^+$ & $0$\\[3pt]
 \hline
\end{tabular}
\end{table}

\begin{table}[!h]
\caption{Range of coefficients for pathwise uniqueness: case $\theta\cong0$}
\label{tab6}
\begin{tabular}{ |c||c|c|c|c|c|}
 \hline
 \multicolumn{6}{|c|}{Pathwise uniqueness} \\
 \hline\hline
 Dimension& 
 Hyperviscosity $\gamma$&
 H\"older coeff.~$\theta$&
 Coeff.~$\mu$&
 Coeff.~$\nu$&
 Noise colour $\rho$\\[3pt]
 \hline
 \multirow{3}{*}{$d=1$}   & $\frac{1}{2}$ &$0^+$ & 
 $0$ & $0^+$ & $0^+$\\[3pt]
 & $1$ &$0^+$ & 
 $0$ & $\frac12^+$ & $(-\frac{1}{4})^+$\\[3pt]
 & $\frac{1}{2}^+$ &$0^+$ & 
 $0$ & $0^+$ & $0$\\[3pt]
 \hline
 \multirow{2}{*}{$d=2$} & $1$ &$0^+$ & 
 $0$ & $0^+$ & $0^+$\\[3pt]
 &   
 $1^+$  & $0^+$   &$0$
 & $0^+$ & $0$\\[3pt]
 \hline
 \multirow{2}{*}{$d=3$} & $\frac{3}{2}$ &$0^+$ & 
 $0$ & $0^+$ & $0^+$\\[3pt]
 &   
 $\frac32^+$  & $0^+$   &$0$
 & $0^+$ & $0$\\[3pt]
 \hline
\end{tabular}
\end{table}
}

\vspace{2pt}

\subsection{Burgers equations with perturbations}
\label{Sc:Burgers}
We consider here Burgers equations in the form 
\begin{align*}
\d X -\Delta X \,\d t + 
(X\cdot \nabla)X\,\d t= (-\Delta)^{-\rho}\,\d W \quad
&\text{in } (0,T)\times\OO\,,\\
X=0\quad
&\text{in } (0,T)\times\partial\OO\,,\\
X(0)=x \quad&\text{in } \OO\,,
\end{align*}
where $\rho>-\frac12$, $X$ is an $\mathbb R^d$-valued process, 
$-\Delta$ is the classical vectorial extension of the negative Laplace operator, $W$ is an $H$-cylindrical Wiener process with $H=L^2(\OO;\R^d)$, 
and the nonlinearity $B_1$ is defined on smooth functions as
\begin{equation}\label{B1}
  [B_1(\varphi)]_i:=[(\varphi\cdot\nabla)\varphi]_{i}:=
  \sum_{j=1}^d\varphi_j\partial_{j}\varphi_i\,,
  \quad i=1,\ldots,d\,, \quad \varphi\in C^\infty_c(\OO;\mathbb R^d)\,.
\end{equation}
Note that it holds
\[
  B_1(\varphi)=(J\varphi)\varphi \quad\forall\,\varphi\in C^\infty_c(\OO;\mathbb R^d)\,,
\]
where $J$ denotes the Jacobian, so that 
\[
B_1(\varphi)=\frac12\nabla(|\varphi|^2)\quad\forall\,\varphi\in C^\infty_c(\OO;\mathbb R^d)
\qquad\Leftrightarrow\qquad d=1\,.
\]

\lc{We stress that the same machinery that we present here will yield the
Navier-Stokes case with only slight operator-theoretic modifications, see Subsection~\ref{Sc:Navier-Stokes}.}
With a slight abuse of notation,
we use the same symbol $A_D$ for the 
vector-valued realisation 
of the negative Laplacian with homogeneous Dirichlet 
conditions, as well as the symbols $V_{2s}$, $s\in \R$, 
for the domain of its $s$-th power.

In dimension $d=1$, since
$V_{\varepsilon}\subseteq L^{\frac2{1-2\varepsilon}}(\OO)$
and $L^{\frac{1}{1-2\varepsilon}}(\OO)\subseteq V_{-\frac12+2\varepsilon}$
for all $\varepsilon\in\left(0,\frac12\right)$
by the Sobolev embeddings, one has that
$B_1\in C^{0,1}_{{\rm loc}}(V_\varepsilon; 
V_{-\frac32+2\varepsilon})$
for all $\varepsilon\in\left(0,\frac12\right)$.
This means that in dimension $d=1$, one can take 
$\mu=\frac\varepsilon2>0$ and $\nu\geq\frac34-2\mu$.
In dimension $d\geq2$, 
as a consequence of Lemma~\ref{Lemma-immersioni}, 
it readily follows that for every 
$\mu\in \left[\frac14,\frac12\right]$ and $\nu\geq \frac12-\mu$ such that $2\mu+\nu>\frac{d+2}4$, the operator 
$B_1$ has a well-defined extension
$B_1:V_{2\mu}\to V_{-2\nu}$ and it holds that 
\[
  \norm{B_1(\varphi)}_{V_{-2\nu}}
  \leq C\norm{\varphi}_{V_{2\mu}}^2
  \quad\forall\,\varphi\in V_{2\mu}\,.
\]
Furthermore, for every $\varphi,\psi\in V_{2\mu}$,
by using again Lemma~\ref{Lemma-immersioni} we have 
\begin{align*}
\norm{B_1(\varphi)-B_1(\psi)}_{V_{-2\nu}}&\leq
\sum_{i,j=1}^d\norm{\varphi_j\partial_{j}\varphi_i
-\psi_j\partial_{j}\psi_i}_{V_{-2\nu}}\\
&\leq\sum_{i,j=1}^d\left[
\norm{(\varphi_j-\psi_j)\partial_j\psi_i}_{V_{-2\nu}}
+\norm{\varphi_j\partial_j(\varphi_i-\psi_i)}_{V_{-2\nu}}
\right]\\
& \leq C\sum_{i,j=1}^d\left[
\norm{\varphi_j-\psi_j}_{V_{2\mu}}
\norm{\partial_j\psi_i}_{V_{2\mu-1}}
+\norm{\varphi_j}_{V_{2\mu}}\norm{\partial_j(\varphi_i-\psi_i)}_{V_{2\mu-1}}
\right]\\
&\leq Cd^2
\left[\norm{\varphi}_{V_{2\mu}}+\norm{\psi}_{V_{2\mu}}\right]
\norm{\varphi-\psi}_{V_{2\mu}}\  
\end{align*}
so that $B_1:V_{2\mu}\to V_{-2\nu}$ is locally Lipschitz continuous, i.e.~$B_1\in C^{0,1}_{{\rm loc}}(V_{2\mu}; V_{-2\nu})$.

In order to allow for more generality, we consider also 
possible hyperviscous diffusions and further 
nonlinear perturbations in the abstract form 
\begin{align}
\label{Burgers}
\d X + A_D^\gamma X \,\d t + 
B_1(X)\,\d t= B_2(X)\,\d t+A_D^{-\rho}\,\d W\,, 
\qquad X(0)=x\,,   
\end{align}
where $\gamma>0$ is a parameter that describes 
the viscous diffusion effect and 
$B_2\in C^\theta_{{\rm loc}}(V_{2\mu};V_{-2\nu})$ for some $\theta\in(0,1)$ has linear growth. 
Note that if $B_2\equiv 0$ and $\gamma=1$, then \eqref{Burgers} 
is the classical Burgers equation. Otherwise, $B_2$ can 
be considered as a suitable perturbation.

It is natural to set 
$A:=A_D^\gamma$, $\alpha:=\frac\mu\gamma$, 
$\beta:=\frac\nu\gamma$, $\delta:=\frac\rho\gamma$.
\lc{We recall that the conditions 
     $\alpha,\beta\in[0,1)$, $\delta\in\left(-\frac12+\alpha, \frac12\right]$, 
    and $\beta+\delta\in\left[0,\frac12\right]$ 
    with $\delta\geq\alpha$ if $\beta+\delta=\frac12$
    translate as
    $\mu,\nu\in[0,\gamma)$, 
    $\rho\in    \left(-\frac\gamma2+\mu,\frac\gamma2\right]$,
    and $\nu+\rho\in\left[0,\frac\gamma2\right]$
    with $\rho\geq\mu$ if $\nu+\rho=\frac\gamma2$, respectively.
}
By arguing as in Subsection~\ref{ssec:heat_fr} and 
by taking into account the further conditions on $\mu,\nu$ obtained above, 
the results contained in Theorems~\ref{th:1}, \ref{th:3}, and \ref{th:2} can be collected as follows.

{\sc Dimension $d=1$.} 
\begin{itemize}
\item Weak uniqueness holds for initial data in $D(A_D^\mu)$ for 
\begin{alignat*}{2}
&\gamma>\textcolor{blue}{\frac34}\,, 
\qquad\theta\in(0,1)\,,
\qquad 
&&\mu\in\left(\max\{0,1-\gamma\}, \color{blue}\frac14\right)\,,\\
&\nu\geq\max\left\{0,\frac34-2\mu\right\}\,,\qquad
&&\mu+\nu\in\left[0,\gamma-\frac 14\right)\,, 
\end{alignat*}
and for every $\rho\in\left(\frac 14 + \mu-\frac \gamma2,\frac\gamma2-\nu\right]$,
\lc{with $\rho\geq\mu$ if $\nu+\rho=\frac\gamma2$}.
The condition $\gamma>\textcolor{blue}{\frac34}$ is the minimal one ensuring 
a possible choice of $(\mu,\nu)$ satisfying the other conditions.
In particular, 
weak uniqueness holds for the classical 
one-dimensional Burgers equation ($\gamma=1$)
with initial data in 
$V_{2\varepsilon}=H^{2\varepsilon}(\OO)$
for every $\varepsilon\in(0,\frac14)$, by choosing 
$\mu=\varepsilon$, $\nu=\frac34-2\varepsilon$, and 
$\rho\in\left(-\frac14+\varepsilon, -\frac14+2\varepsilon\right)$.
\lc{Moreover, weak uniqueness holds for $\gamma=1$ also with $\rho=0$ (space-time white noise) and initial data in
$V_{2\varepsilon}=H^{2\varepsilon}(\OO)$
for every $\varepsilon\in(\frac18,\frac14)$, by choosing $\mu=\varepsilon$ and $\nu\in[\frac34-2\varepsilon,\frac12)$.}

\item Pathwise uniqueness holds for initial data in $D(A_D^\mu)$ for 
\begin{alignat*}{2}
&\gamma>\textcolor{blue}{\frac34}\,, 
\qquad\lc{\theta\in(\max\{0,4(1-\gamma)\},1)}\,,\qquad
&&\mu\in\left(\max\{0,1-\gamma\}, \color{blue}\frac14\right)\,,\\
&\nu\geq\max\left\{0,\frac34-2\mu\right\}\,,\qquad
&&\mu+\nu\in\left[0,\gamma+\frac\theta4-\frac12\right)\,, 
\end{alignat*}
and $\rho\in\left(\frac{1}{4} +\mu- \frac{\gamma}{2},\frac\gamma2-\nu\right)\cap 
\left(\frac{1}{4} +\mu- \frac{\gamma}{2},
\frac1{1-\theta}(\min\left\{
\frac\gamma2-\frac 14-\nu, 
\mu\right\}
+ \frac{\gamma\theta}2-\theta\mu)\right)$. 
In particular, 
pathwise uniqueness holds for the classical 
one-dimensional Burgers equation ($\gamma=1$)
with initial data in 
$V_{2\varepsilon}=H^{2\varepsilon}(\OO)$
for every $\varepsilon\in(0,\frac14)$ by choosing 
$\mu=\varepsilon$, $\nu=\frac34-2\varepsilon$, 
$\theta\in(1-4\varepsilon,1)$,
and $\rho\in\left(-\frac14+\varepsilon, -\frac14+2\varepsilon\right)
\cap\left(-\frac14+\varepsilon, 
\frac1{1-\theta}\left(-\frac12+2\varepsilon+\frac\theta2-\theta\varepsilon\right)\right)$.
Analogously, pathwise uniqueness holds
with data in
$V_{\frac12}=H^{\frac12}_{00}(\OO)$
by choosing
$\mu=\frac14$,
$\theta\in(0,1)$,
$\rho\in\left(0,\frac14
 \min\{1,\frac{\theta}{1-\theta}\}\right)$,
\lc{and $\nu>\frac14$ sufficiently close to $\frac14$.}
\end{itemize}

{\sc Dimension $d=2$.}

\begin{itemize}
\item Weak uniqueness holds for initial data in $D(A_D^\mu)$ for 
\[
\gamma>1\,, 
\qquad\theta\in(0,1)\,,
\qquad
\mu\in\left[\frac14,\frac12\right]\,,\qquad
\mu+\nu\in\left[\frac12,\gamma-\frac12\right)\,,
\qquad 2\mu+\nu>1\,,
\]
and for every $\rho\in\left(\frac 12 + \mu-\frac \gamma2,\frac\gamma2-\nu\right]$,
\lc{with $\rho\geq\mu$ if $\nu+\rho=\frac\gamma2$}.
In particular, 
weak uniqueness holds for the 
two-dimensional hyperviscous Burgers equation 
in the following two interesting regimes:
either we take initial data in 
$D(A_D^{\frac14})=V_{\frac12}=H^{\frac12}_{00}(\OO)$ by choosing 
$\gamma>\frac54$, $\mu=\frac14$, 
$\nu\in\left(\frac12,\gamma-\frac34\right)$, and 
$\rho\in\left(\frac34-\frac\gamma2,\frac\gamma2-\nu\right]$, or we take 
initial data in $D(A_D^{\frac12})=V_{1}=H^{1}_{0}(\OO)$ by choosing 
$\gamma>1$, $\mu=\frac12$, 
$\nu\in(0,\gamma-1)$, and 
$\rho\in\left(1-\frac\gamma2,\frac\gamma2-\nu\right]$,
always \lc{with $\rho\geq\mu$ if $\nu+\rho=\frac\gamma2$}.

\item Pathwise uniqueness holds for initial data in $D(A_D^\mu)$ for 
\[
\gamma>1\,,
\qquad\theta\in(\max\{0,3-2\gamma\},1)\,,
\qquad
\mu\in\left[\frac14,\frac12\right]\,,\qquad
\mu+\nu\in\left[\frac12,\gamma+\frac\theta2-1\right)\,, \qquad
2\mu+\nu>1\,,
\]
and for every $\rho\in\left(\frac{1}{2} +\mu- \frac{\gamma}{2},
\frac\gamma2-\nu\right)\cap 
\left(\frac{1}{2} +\mu- \frac{\gamma}{2},
\frac1{1-\theta}(\min\left\{
\frac\gamma2-\frac 12-\nu, 
\mu\right\}
+ \frac{\gamma\theta}2-\theta\mu)\right)$.
In particular, 
pathwise uniqueness holds for the 
two-dimensional hyperviscous Burgers equation 
in the following two interesting regimes:
either we take initial data in 
$D(A_D^{\frac14})=V_{\frac12}=H^{\frac12}_{00}(\OO)$ by choosing 
$\gamma>\frac54$, 
$\theta\in\left(\max\{0,\frac72-2\gamma\}, 1\right)$, $\mu=\frac14$, 
$\nu\in\left(\frac12,\gamma+\frac\theta2-\frac54\right)$, and 
$\rho\in\left(\frac34-\frac\gamma2,\frac\gamma2-\nu\right)
\cap \left(\frac{3}{4} - \frac{\gamma}{2},
\frac1{1-\theta}(\min\left\{
\frac\gamma2-\frac 12-\nu, 
\frac14\right\}
+ \frac{\gamma\theta}2-\frac\theta4)\right)$, or we take 
initial data in $D(A_D^{\frac12})=V_{1}=H^{1}_{0}(\OO)$ by choosing 
the parameters
$\gamma>1$, 
$\theta\in(\max\{0,3-2\gamma\},1)$,
$\mu=\frac12$, 
$\nu\in\left(0,\gamma+\frac\theta2-\frac32\right)$, and 
$\rho\in\left(1-\frac\gamma2,\frac\gamma2-\nu\right)
\cap\left(1- \frac{\gamma}{2},
\frac1{1-\theta}(\min\left\{
\frac\gamma2-\frac 34-\nu, 
\frac12\right\}
+ \frac{\gamma\theta}2-\frac\theta2)\right)$.
\end{itemize}

{\sc Dimension $d=3$.}

\begin{itemize}
\item Weak uniqueness holds for initial data in $D(A_D^\mu)$ for 
\[
\gamma>\frac54\,,
\qquad\theta\in(0,1)\,,
\qquad
\mu\in\left[\frac14,\frac12\right]\,,\qquad
\mu+\nu\in\left[\frac12,\gamma-\frac34\right)\,, 
\qquad 2\mu+\nu>\frac54\,,
\]
and for every $\rho\in\left(\frac 34 + \mu-\frac \gamma2,\frac\gamma2-\nu\right]$,
\lc{with $\rho\geq\mu$ if $\nu+\rho=\frac\gamma2$}.
In particular, 
weak uniqueness holds for the 
three-dimensional hyperviscous Burgers equation 
in the following two interesting regimes:
either we take initial data in 
$D(A_D^{\frac14})=V_{\frac12}=H^{\frac12}_{00}(\OO)$ by choosing 
$\gamma>\frac74$, $\mu=\frac14$, 
$\nu\in\left(\frac34,\gamma-1\right)$, and 
$\rho\in\left(1-\frac\gamma2,\frac\gamma2-\nu\right]$, or we take 
initial data in $D(A_D^{\frac12})=V_{1}=H^{1}_{0}(\OO)$ by choosing 
$\gamma>\frac32$, $\mu=\frac12$, 
$\nu\in\left(\frac14,\gamma-\frac54\right)$, and 
$\rho\in\left(\frac54-\frac\gamma2,\frac\gamma2-\nu\right]$,
always \lc{with $\rho\geq\mu$ if $\nu+\rho=\frac\gamma2$}.

\item Pathwise uniqueness holds for initial data in $D(A_D^\mu)$ for 
\begin{alignat*}{2}
&\gamma>\frac54\,,\qquad
&&\theta\in\left(\max\left\{0, 
\frac83-\frac43\gamma\right\},1\right)\,,\\
&\mu\in\left[\frac14,\frac12\right]\,,\qquad &&\mu+\nu\in\left[\frac12,\gamma+\frac34\theta-\frac32\right)\,, 
\qquad 2\mu+\nu>\frac54\,,
\end{alignat*}
and for every $\rho\in\left(\frac{3}{4} +\mu- \frac{\gamma}{2},\frac\gamma2-\nu\right)\cap 
\left(\frac{3}{4} +\mu- \frac{\gamma}{2},
\frac1{1-\theta}(\min\left\{
\frac\gamma2-\frac d4-\nu, 
\mu\right\}
+ \frac{\gamma\theta}2-\theta\mu)\right)$.
In particular, 
pathwise uniqueness holds for the 
three-dimensional hyperviscous Burgers equation 
in the following two interesting regimes:
either we take initial data in 
$D(A_D^{\frac14})=V_{\frac12}=H^{\frac12}_{00}(\OO)$ by choosing 
$\gamma>\frac74$, 
$\theta\in\left(\max\{0,\frac{10}3-\frac43\gamma\}, 1\right)$,
$\mu=\frac14$, 
$\nu\in\left(\frac34,\gamma+\frac34\theta-\frac74\right)$, and 
$\rho\in\left(1-\frac\gamma2,\frac\gamma2-\nu\right)
\cap \left(1 - \frac{\gamma}{2},
\frac1{1-\theta}(\min\left\{
\frac\gamma2-\frac 34-\nu, 
\frac14\right\}
+ \frac{\gamma\theta}2-\frac\theta4)\right)$, or we take 
initial data in $D(A_D^{\frac12})=V_{1}=H^{1}_{0}(\OO)$ by choosing the parameters
$\gamma>\frac32$, 
$\theta\in\left(\max\{0,3-\frac43\gamma\}, 1\right)$,
$\mu=\frac12$, 
$\nu\in\left(\frac14,\gamma+\frac34\theta-2\right)$, and 
$\rho\in\left(\lc{\frac54}-\frac\gamma2,\frac\gamma2-\nu\right)
\cap\left(\lc{\frac54}- \frac{\gamma}{2},
\frac1{1-\theta}(\min\left\{
\frac\gamma2-\frac 34-\nu, 
\frac12\right\}
+ \frac{\gamma\theta}2-\frac\theta2)\right)$.
\end{itemize}

We summarise in the following tables the most relevant choices
of the parameters that ensure uniqueness for hyperviscous 
Burgers equation \eqref{Burgers}. 
More specifically, we highlight here the 
case where the regularity of the perturbation $B_2$ is 
analogous to the one of $B_1$, namely the case where
the coefficient $\theta$ is close to $1$. In this framework, 
for every space dimension we optimise the minimal value of $\gamma$
ensuring uniqueness in the limiting cases $\mu=\frac14$ and $\mu=\frac12$ (and even $\mu=0$ in dimension $1$), 
the former providing a uniqueness result for 
a larger set of initial data and the latter being instead the most classical in the literature.

{\small
\begin{table}[!h]
\caption{Range of coefficients for weak uniqueness}
\label{tab1}
\begin{tabular}{ |c||c|c|c|c|c|}
 \hline
 \multicolumn{6}{|c|}{Weak uniqueness} \\
 \hline\hline
 Dimension& 
 Hyperviscosity $\gamma$&
 H\"older coeff.~$\theta$&
 Coeff.~$\mu$&
 Coeff.~$\nu$&
 Noise colour $\rho$\\[3pt]
 \hline
 \multirow{3}{*}{$d=1$}   & $1$ &$(0,1)$ & 
 $0^+$ & $\frac34^{-}$ & $(-\frac14)^+$\\[3pt]
 & $1$ &$(0,1)$ & 
 $\lc{\frac18^+}$ & $\lc{\frac12^-}$ & $0$\\[3pt]
 & $1$ &$(0,1)$ & 
 $\frac14$ & $\frac14^+$ & $0^+$\\[3pt]
 \hline
 \multirow{2}{*}{$d=2$}&   
 $\frac54^+$  & $(0,1)$   &$\frac14$
 & $\frac12^+$ & $\frac18^-$\\[3pt]
 &   
 $1^+$  & $(0,1)$   &$\frac12$
 & $0^+$ & $\frac12^-$\\[3pt]
 \hline
 \multirow{2}{*}{$d=3$}&   
 $\frac74^+$  & $(0,1)$   &$\frac14$
 & $\frac34^+$ & $\frac18^-$\\[3pt]
 &   
 $\frac32^+$  & $(0,1)$   &$\frac12$
 & $\frac14^+$ & $\frac12^-$\\[3pt]
 \hline
\end{tabular}
\end{table}

\begin{table}[!h]
\caption{Range of coefficients for pathwise uniqueness: case $\theta\cong1$}
\label{tab2}
\begin{tabular}{ |c||c|c|c|c|c|}
 \hline
 \multicolumn{6}{|c|}{Pathwise uniqueness} \\
 \hline\hline
 Dimension& 
 Hyperviscosity $\gamma$&
 H\"older coeff.~$\theta$&
 Coeff.~$\mu$&
 Coeff.~$\nu$&
 Noise colour $\rho$\\[3pt]
 \hline
 \multirow{2}{*}{$d=1$}   & $1$ & $1^-$ & 
 $0^+$ & $\frac34^-$ & $(-\frac14)^+$\\[3pt]
 & $1$ &$1^-$ & 
 $\frac14$ & $\frac14^+$ & $0^+$\\[3pt]
 \hline
 \multirow{2}{*}{$d=2$}&   
 $\frac54^+$  & $1^-$   &$\frac14$
 & $\frac12^+$ & $\frac18^-$\\[3pt]
 &   
 $1^+$  & $1^-$   &$\frac12$
 & $0^+$ & $\frac12^-$\\[3pt]
 \hline
 \multirow{2}{*}{$d=3$}&   
 $\frac74^+$  & $1^-$   &$\frac14$
 & $\frac34^+$ & $\frac18^-$\\[3pt]
 &   
 $\frac32^+$  & $1^-$   &$\frac12$
 & $\frac14^+$ & $\frac12^-$\\[3pt]
 \hline
\end{tabular}
\end{table}
}

As a particular case, if $B_2\equiv0$, i.e.~\eqref{Burgers}
is the classical hyperviscous Burgers equation, 
the second table provides the range of parameters for 
pathwise uniqueness. For example, in dimension $3$ we obtain 
pathwise uniqueness for initial data in $H^{\frac12}_{00}(\OO)$
provided that the hyperviscosity satisfies $\gamma>\frac74$.
As far as existence of solutions is concerned, 
this is widely studied by using ad-hoc techniques.
In particular, with $\gamma=1$
the equation is well-posed in dimension $1$
with initial data in $L^2(\OO)$ (see \cite{gy98,LiuRo}).
In higher dimensions, the case $\gamma=1$ is more delicate:
the literature is extremely vast and among the numerous contributions
we refer to \cite{Brz, zhang2020} and the references therein.

An alternative perspective would be to consider the case
where the perturbation $B_2$ is less regular than $B_1$, namely 
that $\theta$ is closer to $0$. In such situation,
the optimisation of $\gamma$ would naturally increase
the minimal value of the parameter
for which uniqueness holds.

{\small
\begin{table}[!h]
\caption{Range of coefficients for pathwise uniqueness: case $\theta\cong0$}
\label{tab3}
\begin{tabular}{ |c||c|c|c|c|c|}
 \hline
 \multicolumn{6}{|c|}{Pathwise uniqueness} \\
 \hline\hline
 Dimension& 
 Hyperviscosity $\gamma$&
 H\"older coeff.~$\theta$&
 Coeff.~$\mu$&
 Coeff.~$\nu$&
 Noise colour $\rho$\\[3pt]
 \hline
 \multirow{1}{*}{$d=1$}&   
 $1$  & $0^+$   &$\frac14$
 & $\frac14^+$ & $0^+$\\[3pt]
 \hline
 \multirow{1}{*}{$d=2$}&   
 $\frac32^+$  & $0^+$   &$\frac12$
 & $0^+$ & \textcolor{blue}{$\frac14^+$}\\[3pt]
 \hline
 \multirow{1}{*}{$d=3$}&   
  $\frac94^+$  & $0^+$   &$\frac12$
 & $\frac14^+$ & \textcolor{blue}{$\frac18^+$}\\[3pt]
 \hline
\end{tabular}
\end{table}
}

\subsection{Navier-Stokes equations with perturbations}
\label{Sc:Navier-Stokes}
We consider here Navier-Stokes equations in the form 
\begin{align*}
\d X -\Delta X \,\d t + 
(X\cdot \nabla)X\,\d t + \nabla p\,\d t
= (-\Delta)^{-\rho}\,\d W \quad
&\text{in } (0,T)\times\OO\,,\\
\nabla\cdot X = 0 \quad&\text{in } (0,T)\times\OO\,,\\
X=0\quad
&\text{in } (0,T)\times\partial\OO\,,\\
X(0)=x \quad&\text{in } \OO\,,
\end{align*}
where $\rho>-\frac12$, $X$ is an $\mathbb R^d$-valued process, 
$-\Delta$ is the classical vectorial extension of the negative Laplace operator, $W$ is an $L^2(\OO;\R^d)$-cylindrical Wiener process, 
and $p$ represents the pressure.

We define the solenoidal vector-valued spaces
	\begin{align*}
		H_\sigma &:= \overline{\{ v \in  C^\infty_0(\OO;\mathbb R^d): \nabla \cdot v = 0 \text{ in } \OO\}}^{L^2(\OO;\mathbb R^d)}\,,
        \\
		V_\sigma &:= \overline{\{ v \in C^\infty_0(\OO;\mathbb R^d) : \nabla \cdot v = 0 \text{ in } \OO\}}^{H^1(\OO;\R^d)}\,.
	\end{align*}
    The Stokes operator $A:V_\sigma\to V_\sigma^*$ is defined as
	the canonical Riesz isomorphism of $V_\sigma$, and can be seen
    as the composition of the classical negative Dirichlet (vectorial)
    Laplacian $A_D$ with the Leray projection.

The nonlinearity $B_1$ is defined on smooth functions 
exactly as in \eqref{B1}. Moreover, 
it is immediate to check that the same computations performed in 
Subsection~\ref{Sc:Burgers} yields that 
$B_1:D(A^\mu)\to D(A^{-\nu})$ is well-defined and
\begin{align*}
\norm{B_1(\varphi)-B_1(\psi)}_{D(A^{-\nu})}\leq
C
\left[\norm{\varphi}_{D(A^\mu)}+\norm{\psi}_{D(A^\mu)}\right]
\norm{\varphi-\psi}_{D(A^{\mu})} \quad\forall\,\varphi,\psi\in
D(A^\mu)\,,
\end{align*}
provided that $\mu\in[\frac14,\frac12]$, $\nu\geq\frac12-\mu$,
and $2\mu+\nu>\frac{d+2}4$.

We can then frame equation \eqref{NS} in abstract form 
in the space $H_\sigma$, by 
possibly including also hyperviscous diffusions and further 
nonlinear perturbations:
\begin{align}
\label{NS}
\d X + A^\gamma X \,\d t + 
B_1(X)\,\d t= B_2(X)\,\d t+A^{-\rho}\,\d W\,, 
\qquad X(0)=x\,.
\end{align}
Here, $\gamma\geq1$ is a parameter that describes 
the viscous diffusion effect and 
$B_2\in C^\theta_{{\rm loc}}(D(A^{\mu});D(A^{-\nu}))$ for some $\theta\in(0,1)$ has linear growth. 
Note that if $B_2\equiv 0$ and $\gamma=1$, then \eqref{NS} 
is the classical Navier-Stokes equation. Otherwise, $B_2$ can 
be considered as a suitable perturbation.

It is clear, at least for what concerns our analysis,
that the setting is entirely analogous to the one described in 
Subsection~\ref{Sc:Burgers} for the Burgers equation.
Hence, the range of the coefficients $\gamma, \theta, \mu,\nu$ and $\rho$ for which weak or pathwise uniqueness hold
is the same as the one presented in Subsection~\ref{Sc:Burgers}.
In particular, the tables \ref{tab1}--\ref{tab3}
are still valid also in this case.
The only difference concerns the characterisation of the 
fractional powers $D(A^s)$, $s\in[0,1]$, that here are 
suitable solenoidal spaces instead.

Let us note that uniqueness by noise for hyperviscous Navier-Stokes
equations has been studied also in \cite{Agresti} for small
hyperviscosity and less regular initial data.
In this case, however, the notion of uniqueness is weaker,
in the sense that it is intended only with high probability.
The critical case of initial data in $H^{\frac12}$ with $\gamma=1$
has been covered instead in \cite{AyXu}, still under smallness
conditions on the initial data and for uniqueness with 
high probability.
See also \cite[Thm.~2.11]{AG} for the case of even negative regularity for the initial data.

\subsection{Cahn-Hilliard equations with perturbations}
Following \cite{BOS}, in this subsection we consider SPDEs of the form
\begin{align}
\label{cahn-hilliard}
\left\{
\begin{array}{ll}
\d X-\Delta(-\Delta X+F_1(X))\,\d t=
F_2(X,\nabla X, D^2 X)\,\d t+(\operatorname{I}-\Delta)^{-\rho}\,\d W & \textrm{in }(0,T)\times \mathcal O\,,\\
\partial_{\bm n}X=\partial_{\bm n} \Delta X=0 & \textrm{in }(0,T)\times \partial \mathcal O\,,\\
X(0)=x_0 & \textrm{in }\mathcal O\,.
\end{array}
\right.
\end{align}
Here, $F_1:\R\to \R$ is a $C^2$-function with locally Lipschitz-continuous derivatives, $F_2$ is a $\theta$-H\"older continuous function for some $\theta\in(0,1)$, $W$ is a cylindrical Wiener process on $H=L^2(\mathcal O)$, $\partial_{\bm n}$ denotes the outer normal derivative on $\partial \mathcal O$ and $x_0\in L^2(\mathcal O)$.  Since
\begin{align*}
\Delta F_1(\varphi)={\rm div}[F'_1(\varphi)\nabla\varphi]
= F_1''(\varphi)|\nabla\varphi|^2+F_1'(\varphi)\Delta\varphi\,,
\qquad \varphi\in H^2(\mathcal O)\,,
\end{align*}
we introduce the operator $B_1:H^2(\mathcal O)\to L^2(\mathcal O)$ defined as
\begin{align*}
B_1(\varphi)(\xi)=-F_1''(\varphi(\xi))|\nabla \varphi(\xi)|^2-F'_1(\varphi(\xi))\Delta\varphi(\xi)\,,\qquad
\xi\in \mathcal O\,,\quad
\varphi\in H^2(\mathcal O)\,,  
\end{align*}
which is well-defined, measurable, and locally Lipschitz-continuous.
Indeed, since 
$F_1'$ and $F_1''$ are locally Lipschitz-continuous, 
for every $\varphi_1, \varphi_2\in H^2(\OO)$, we have
\begin{align*}
  \norm{B_1(\varphi_1)-B_1(\varphi_2)}_H
  &\leq \norm{(F_1''(\varphi_1)-F_1''(\varphi_2))|\nabla\varphi_2|^2}_H
  +\norm{F_1''(\varphi_1)(|\nabla\varphi_1|^2-|\nabla\varphi_2|^2)}_H\\
  &\quad +\norm{(F_1'(\varphi_1)-F_1'(\varphi_2))\Delta\varphi_2}_H
  +\norm{F_1'(\varphi_1)(\Delta\varphi_1-\Delta\varphi_2)}_H\\
&\leq C\left(\norm{\varphi_1}_{L^\infty(\OO)}
+\norm{\varphi_2}_{L^\infty(\OO)}\right)
\left(\norm{\nabla\varphi_2}_{L^4(\OO)}^4
+\norm{\Delta\varphi_2}_H\right)
\norm{\varphi_1-\varphi_2}_{L^\infty(\OO)}\\
&\quad+C\left(1+\norm{\varphi_1}_{L^\infty(\OO)}\right)
\left(\norm{\nabla\varphi_1}_{L^3(\OO)}
+\norm{\nabla\varphi_2}_{L^3(\OO)}\right)
\norm{\nabla(\varphi_1-\varphi_2)}_{L^6(\OO)}\\
&\quad+C\left(1+\norm{\varphi_1}_{L^\infty(\OO)}\right)
\norm{\Delta\varphi_1-\Delta\varphi_2}_H\,,
\end{align*}
and by using that $H^2(\OO)\subseteq L^\infty(\OO)$ and
$H^1(\OO)\subseteq L^6(\OO)$, we get 
\[
\norm{B_1(\varphi_1)-B_1(\varphi_2)}_H
\leq C\left(1+\norm{\varphi_1}_{H^2(\OO)}^5
+\norm{\varphi_2}_{H^2(\OO)}^5\right)\norm{\varphi_1-\varphi_2}_{H^2(\OO)}\,.
\]
Furthermore, we introduce the operator $B_2:H^2(\mathcal O)\to L^2(\mathcal O)$, defined as
\begin{align*}
B_2(\varphi)(\xi)
= F_2(\varphi(\xi), \nabla \varphi(\xi), D^2\varphi(\xi)), \qquad 
\xi\in\mathcal O\,,\quad
\varphi\in H^2(\mathcal O)\,.
\end{align*}
Hence, if we set $A:=\operatorname{I}+A_N$, 
then the SPDE \eqref{cahn-hilliard} reads as
\begin{align*}
\d X+A^2X\,\d t
=B(X)\,\d t +(A^2)^{-\frac{\rho}{2}}\,\d W\,, \qquad X(0)=x_0\in H,
\end{align*}
where $B:D(A)\to H$ is the operator defined as 
\[
B(\varphi)=
B_2(\varphi)-B_1(\varphi)
+\varphi-2\Delta\varphi\,, \quad
\varphi\in D(A)= D(A_N)\,.
\]
In the notation of Theorems~\ref{th:1}, \ref{th:2}, and \ref{th:3},
we have $\alpha=\frac12$, $\beta=0$, and $\delta=\frac\rho2$.
It is immediate to check that assumptions \ref{H0}--\ref{H4}
are satisfied for $\rho\in\left(\frac d4,1\right]$, while 
\ref{H5} for $\rho\in \left(\frac d4, 1\right)$.
Moreover, \eqref{trace1} yields $(1-\theta)\rho<1-\frac d4$, 
while \eqref{ineq2_al} yields $(1-\theta)\rho<1$ which is weaker than \eqref{trace1}.
The results contained in Theorems~\ref{th:1}, \ref{th:2}, and \ref{th:3} can be then collected as follows.
\begin{itemize}
\item Weak uniqueness holds for initial data in 
$D(A_N)$ for 
\[ \rho\in\left(\frac d4, 1\right]\,,
\qquad\theta\in(0,1)\,.
\]
\item Pathwise uniqueness holds for initial data in $D(A_N)$
for 
\[ \rho\in\left(\frac d4, 1\right)\,,
\qquad\theta\in(0,1)\,,
\qquad (1-\theta)\rho<1-\frac d4\,. 
\]
\end{itemize}

Let us stress that these results hold for every growth of the 
nonlinearity $F_1$, but can be refined under further
growth conditions on $F_1$ \lc{and if $F_2=F_2(\varphi,\nabla\varphi)$ does not depend on $D^2\varphi$}. 
If one considers the classical case where
$F_1$ is the derivative of 
the paradigmatic quartic polynomial potential, then the operator
$B_1$ can be seen equivalently, \lc{in dimension $d\in\{1,2\}$,} as a well-defined locally Lipschitz operator 
\lc{$B_1:H^1(\OO)\to H^{-1-\varepsilon}(\OO)$,
for any $\varepsilon\in(0,1)$ (also $\varepsilon=0$ if $d=1$) thanks to the inclusion 
$H^1(\OO)\subseteq L^s(\OO)$ for every $s\in(1,\infty)$. Indeed, if $d\in\{1,2\}$,
given $\varepsilon\in(0,1)$ one can choose $p\in(1,2)$ such that 
$L^p(\OO)\subset H^{-\varepsilon}(\OO)$ and, for $q>2$ with 
$\frac1p=\frac1q+\frac12$, we get
\begin{align*}
\norm{B_1(\varphi_1)-B_1(\varphi_2)}_{H^{-1-\varepsilon}(\OO)}
&\lesssim\norm{\nabla F_1(\varphi_1)-\nabla F_1(\varphi_2)}_{L^p(\OO)}\\
  &\leq \norm{F_1'(\varphi_1)}_{L^q(\OO)}\norm{\nabla(\varphi_1-\varphi_2)}_{H}
+\norm{F_1'(\varphi_1)-F_1'(\varphi_2)}_{L^q(\OO)}\norm{\nabla\varphi_2}_{H}\\
&\leq C\left(1+\norm{\varphi_1}^2_{H^1(\OO)}+\norm{\varphi_2}^2_{H^1(\OO)}\right)
\norm{\varphi_1-\varphi_2}_{H^1(\OO)}
\end{align*}
by the embedding $H^1(\OO)\subseteq L^{2q}(\OO)$. \lc{Moreover, since $F_2$ is $\theta$-H\"older continuous and does not depend on $D^2\varphi$, the operator $B_2$ is $\theta$-H\"older continuous from $H^1(\OO)$ to $L^2(\OO)\subseteq H^{-1-\varepsilon}(\OO)$.}
Hence, if $F_1$ satisfies such further growth condition and $F_2$ does not depend on $D^2\varphi$, 
in the setting presented above in $d\in\{1,2\}$
one can take $\alpha=\frac14$ and 
$\beta=\frac{1+\varepsilon}4$. This yields the condition
$\rho\in\left(\frac d4-\frac12, \frac12-\frac\varepsilon2\right)$ for \ref{H0}--\ref{H4}, 
$\rho(1-\theta)<\frac\theta2-\frac d4+\frac{1-\varepsilon}2$ for \eqref{trace1},
and $\rho(1-\theta)<\frac12+\frac\theta2$ for \eqref{ineq2_al}
which is weaker than \eqref{trace1}.
The results contained in Theorems~\ref{th:1}, \ref{th:2}, and \ref{th:3} can be then collected as follows.
\begin{itemize}
\item Weak uniqueness holds for initial data in 
$D(A_N^{\frac12})=H^1(\OO)$ for $d=1,2$ and
\[ \rho\in\left(\frac d4-\frac12, \frac12\right)\,,
\qquad\theta\in(0,1)\,.
\]
\item Pathwise uniqueness holds for initial data in 
$D(A_N^{\frac12})=H^1(\OO)$
for $d=1,2$ and
\[ \rho\in\left(\frac d4-\frac12, \frac12\right)\,,
\qquad\theta\in(0,1)\,,
\qquad (1-\theta)\rho<\frac12+\frac\theta2-\frac d4\,.
\]
\end{itemize}
In dimension $d=3$ the computations above do not transfer directly, since one only has that $H^1(\OO)\subseteq L^6(\OO)$, hence further restrictions on $p,q$ and $s$ are needed.
For this case, we still refer to the case of $H^2(\OO)$
initial data discussed above.}

\subsection{Reaction-diffusion equations with perturbations}
Let us consider SPDEs in the form
  \begin{align*}
  \d X - \Delta X\,\d t  = 
  F_1(X)\,\d t +F_2(X)\,\d t
  + (-\Delta)^{-\delta}\,\d W \quad&\text{in } (0,T)\times\OO\,,\\
  X=0\quad&\text{in } (0,T)\times\partial\OO\,,\\
  X(0)=x_0 \quad&\text{in } \OO\,,
  \end{align*}
where $F_1:\R\to\R$ is locally Lipschitz-continuous,
$F_2\in C^{\theta}(\R)$ for some $\theta\in (0,1)$, 
and there exist $p>2$ and $C>0$ such that 
\[
  |F_1(u)-F_1(v)|\leq C(1+|u|^{p-2}+|v|^{p-2})|u-v|
  \quad\forall\,u,v\in\R\,.
\]
Given $r\in[\max\{2,p-1\}, 2(p-1)]$, 
setting $s:=\frac{r}{p-1}\in[1,2]$
one can define
\[
  B_i:L^r(\OO)\to L^s(\OO)\,, \quad i=1,2\,,
\]
as, for $i=1,2$,
\begin{align*}
  B_i(\varphi)(\xi)&:=
  F_i(\varphi(\xi)),
  \quad\xi\in\OO\,,
  \quad\varphi\in L^r(\OO)\,.
\end{align*}
Note that $B_1$ is locally-Lipschitz continuous: indeed, for all 
$\varphi_1,\varphi_2\in L^r(\OO)$,
since $r>s=\frac r{p-1}$, 
by the H\"older inequality we have
\begin{align*}
    \norm{B_1(\varphi_1)-B_1(\varphi_2)}_{L^s(\OO)}
    &\leq C\norm{(1+|\varphi_1|^{p-2}+|\varphi_2|^{p-2})|\varphi_1-\varphi_2|}_{L^s(s)}\\
    &\leq C\left(1+
    \norm{\varphi_1}^{p-2}_{L^{\frac{sr(p-2)}{r-s}}(\OO)}+
    \norm{\varphi_2}^{p-2}_{L^{\frac{sr(p-2)}{r-s}}(\OO)}\right)
    \norm{\varphi_1-\varphi_2}_{L^{r}(\OO)}\\
    &=C\left(1+
    \norm{\varphi_1}^{p-2}_{L^r(\OO)}+
    \norm{\varphi_2}^{p-2}_{L^r(\OO)}\right)
    \norm{\varphi_1-\varphi_2}_{L^{r}(\OO)},
\end{align*}
Moreover, note that $B_2 \in C^\theta(L^r(\OO), L^s(\OO))$: indeed, 
for all 
$\varphi_1,\varphi_2\in L^r(\OO)$, we have 
\begin{align*}
    \norm{B_2(\varphi_1)-B_2(\varphi_2)}_{L^s(\OO)}
    &\leq \norm{F_2}_{C^\theta(\R)}
    \norm{|\varphi_1-\varphi_2|^\theta}_{L^s(\OO)}\\
    &=\norm{F_2}_{C^\theta(\R)}
    \norm{\varphi_1-\varphi_2}^\theta_{L^{\frac{r\theta}{p-1}}(\OO)}
\end{align*}
and we conclude since $\frac{r\theta}{p-1}\leq r$.
It follows that, if we define
\[
  B:=B_1+B_2:L^r(\OO)\to L^s(\OO)\,,
\]
then $B$ is locally $\theta$-H\"older-continuous. 
We can then formulate the SPDE in abstract form as
\[
  \d X + A_DX\,\d t = B(X)\,\d t + A_D^{-\delta}\,\d W\,,
  \qquad X(0)=x_0\,.
\]
Now, we choose $\alpha\in[0,1)$ and $\beta\in\left[0,\frac12\right]$
such that $\beta+\delta\in\left[0,\frac12\right]$ \lc{(with $\delta\geq\alpha$ if $\beta+\delta=\frac12$)},
$D(A_D^\alpha)\subseteq L^r(\OO)$,
and $L^s(\OO)\subseteq D(A_D^{-\beta})$, so that 
$B$ can be considered as a locally $\theta$-H\"older-continuous 
operator
\[
  B:D(A_D^\alpha)\to D(A_D^{-\beta})\,.
\]
By the classical Sobolev embeddings, 
these inclusions
are satisfied if
$\alpha=\frac{d}{2}(\frac12-\frac1r)=\frac{d(r-2)}{4r}$ and 
$\beta=\frac{d}2(\frac1s-\frac12)=\frac{d(2(p-1)-r)}{4r}$. 
With this choice, 
it is immediate to see that $\alpha\geq0$ and $\beta\geq0$. \lc{If $s=1$ (i.e., $r=p-1$), the embedding $L^1(\OO)\subseteq D(A_D^{-\beta})$ fails for $\beta=\frac d4$: in this case one simply takes $\beta$ slightly larger than $\frac d4$, which is always possible in the arguments below, since all the conditions involving $\alpha+\beta$ are strict inequalities.}

The main assumptions of Theorems~\ref{th:1}, \ref{th:2}, and \ref{th:3}
can be verified in the same way as in Subsection~\ref{ssec:heat}. In particular, 
\ref{H0}--\ref{H2} and \ref{H4} are satisfied, while \ref{H3} requires that $\alpha+\beta<1-\frac d4$
and some $\delta\in\left(\frac{d}{4} +\alpha- \frac{1}{2},\frac12-\beta\right]$, 
namely $r>\frac{2d(p-2)}{4-d}$. This is compatible with the range of $r$ if
$\frac{d(p-2)}{4-d}<p-1$, i.e.~$p(2d-4)<3d-4$: this is always satisfied if $d=1,2$, while if $d=3$ it yields the growth condition $p<\frac52$. Eventually, conditions \eqref{trace1}--\eqref{ineq2_al} are satisfied
with $\alpha+\beta<1-\frac d4\left(2-\theta\right)$
and $\theta>1-\frac 2d$, namely 
$\frac{d(p-2)}{2r}<1-\frac d4(2-\theta)$, 
i.e~$r>\frac{2d(p-2)}{4-d(2-\theta)}$
and $\theta>2-\frac4d$.
This is compatible with the range of $r$ if
$\frac{d(p-2)}{4-d(2-\theta)}<p-1$, 
i.e.~$p((3-\theta)d-4)<(4-\theta)d-4$: again, 
this is always satisfied if $d=1$, while if $d=2$ or $d=3$
it yields the growth conditions 
$p<\frac{2-\theta}{1-\theta}$
and $p<\frac{8-3\theta}{5-3\theta}$, respectively.
The results contained in Theorems~\ref{th:1}, \ref{th:2}, and \ref{th:3} can be then collected as follows.

{\sc Dimension} $d=1$.
Weak and pathwise uniqueness hold for every $\theta\in(0,1)$,
both for $p\in(2,3]$ with initial data in $L^2(\OO)$,
and for $p>3$ with initial data in $D(A_D^{\frac{p-3}{4(p-1)}})$.

{\sc Dimension} $d=2$.
Weak uniqueness holds for every $\theta\in(0,1)$,
\lc{both for $p\in(2,3)$ with initial data in $L^2(\OO)$,
and for $p\geq3$ with initial data in $D(A_D^{\frac{p-3}{2(p-2)}+\varepsilon})$ for all $\varepsilon>0$.
Pathwise uniqueness holds for every $\theta\in(0,1)$,
both for $p\in(2,2+\theta)$ 
with initial data in $L^2(\OO)$,
and for $p\in\left[2+\theta, 2+\frac\theta{1-\theta}\right)$
with initial data in 
$D(A_D^{\frac{p-2-\theta}{2(p-2)}+\varepsilon})$
for all $\varepsilon>0$.}

{\sc Dimension} $d=3$.
Weak uniqueness holds for every $\theta\in(0,1)$,
\lc{both for $p\in(2,\frac73)$ with initial data in $L^2(\OO)$,
and for $p\in[\frac73, \frac52)$ with initial data in 
$D(A_D^{\frac{3p-7}{4(p-2)}+\varepsilon})$
for all $\varepsilon>0$.
Pathwise uniqueness holds for every $\theta\in\left(\frac23,1\right)$,
both for $p\in\left(2,2+(\theta-\frac23)\right)$ 
with initial data in $L^2(\OO)$,
and for $p\in\left[2+(\theta-\frac23), 2+\frac{3\theta-2}{5-3\theta}\right)$
with initial data in 
$D(A_D^{\frac{3p-4-3\theta}{4(p-2)}+\varepsilon})$
for all $\varepsilon>0$.}

\appendix
\section{Extension to general initial data}
\label{sec:app}
We present here the technicalities to 
extend the uniqueness and continuous dependence results 
to the case of less regular initial data.

Throughout this section, we work under the following assumptions.
\begin{enumerate}[start=0,label={{(A\arabic*})}]
\item \label{A0} $\mathcal H$ is a real separable Hilbert space
with scalar product $(\cdot, \cdot)_{\mathcal H}$. 
\item \label{A1} $\mathcal A: D(\mathcal A)\subseteq \mathcal H\rightarrow \mathcal H$ is a linear self-adjoint operator with dense domain. Moreover there exists an orthonormal basis $\{\tilde e_k:k\in\N_+\}$ of ${\mathcal H}$ consisting of eigenvectors of ${\mathcal A}$ and an increasing sequence $(\lambda_k)_{k\in\N_+}$ such that $\lambda_1>0$, $\lambda_k\to+\infty$ as $k\to\infty$ and
\[
\mathcal A\tilde e_k=\lambda_k \tilde e_k \quad\forall\,k\in\N_+\,.
\]
\item \label{A2}
\lc{$\tilde\alpha,\tilde\beta\in[0,1)$,
$\tilde\delta\in(-\frac12, \frac12]$, with 
\[
  \tilde\beta+\tilde\delta\in
  \begin{cases}
      [0,\frac12] \quad&\text{if } \tilde\delta\geq0\,,\\
      [0,\frac12) \quad&\text{if } \tilde\delta<0\,.
  \end{cases}
\]}
\item \label{A3} There exists $\eta\in (0,1)$ such that ${\mathcal A}^{-(1+2 {\tilde \delta})+\eta}\in\cL^1(\mathcal H,\mathcal H)$.
\item \label{A4}
\lc{$B\in C^\theta_b(\mathcal H;D(\mathcal A^{-{\tilde \beta}}))$} for some $\theta\in (0,1)$.
\item \label{A5} 
In the regime ${\tilde \beta}+{\tilde \delta}=\frac12$ there exists $z_0\in D({\mathcal A}^{-{\tilde \beta}})$ such that 
\[
\sup_{x\in {\mathcal H}}\norm{B(x)-z_0}_{D({\mathcal A}^{-{\tilde \beta}})}<
\frac{\theta(1-\theta)(1-{\tilde \beta})}
{4M_{{\tilde \beta},{\tilde \delta},\theta}(2-\theta)}\,,
\]
where $M_{{\tilde \beta},{\tilde \delta},\vartheta}$ is the implicit constant 
appearing in Proposition~\ref{stimeHolder}.
\end{enumerate}

First, we adapt the concept of solution given in Definition~\ref{def-sol} for the equation
\begin{equation}
    \label{SDE-app}
    \d X + \mathcal AX\,\d t = B(X)\,\d t + {\mathcal A}^{-\tilde \delta}\,\d\mathcal  W\,, \qquad X(0)=x\,,
\end{equation}
to the case of less regular initial data $x\in D(\mathcal A^{-\tilde\alpha})$.

\begin{defn}\label{def-sol-app}
Assume \ref{A0}--\ref{A4} and let $x\in D({\mathcal A}^{-{\tilde \alpha}})$.
A weak mild solution to \eqref{SDE-app} 
is a pair $(X,\mathcal W)$, where $\mathcal W$ is a ${\mathcal H}$-cylindrical Wiener process on a filtered probability space 
$(\Omega,\mathcal{F},\{\mathcal{F}_t\}_{t\geq0},\mathbb{P})$ and 
$X$ is a ${\mathcal H}$-valued progressively measurable process such that 
\begin{align*}  
X \in C^0(\mathbb R_+; D({\mathcal A}^{-{\tilde \alpha}}))
\cap C^0((0,+\infty); {\mathcal H}) \quad \P\text{-a.s.}
\end{align*}
and
\begin{align}\label{mild-app}
X(t)=e^{-t{\mathcal A}}x+\int_0^te^{-(t-s){\mathcal A}}B(X(s))\,\d s+\mathcal W_{\mathcal A}(t)
\quad\forall\,t\geq0\,, \quad \mathbb{P}\text{-a.s.}
\end{align}
where $W_{\mathcal A}$ is the stochastic convolution process given by 
\begin{equation*}
\mathcal W_{\mathcal A}(t):=\int^t_0
e^{-(t-s)\mathcal A}{\mathcal A}^{-\tilde \delta}\,\d \mathcal W(s)\,, \quad t\geq0\,.
\end{equation*}
\end{defn}

\begin{rmk}
We note that the integral formulation \eqref{mild-app}
is well-defined in $D({\mathcal A}^{-{\tilde \alpha}})$. 
Indeed, since $x\in D({\mathcal A}^{-{\tilde \alpha}})$ it holds that 
$t\mapsto e^{-t{\mathcal A} }x$ is continuous with values in $D({\mathcal A}^{-{\tilde \alpha}})$.
Moreover, since $B$ is bounded by \ref{A4},
the deterministic convolution is continuous with values in 
$D({\mathcal A}^{-{\tilde \alpha}})$
since $1-{\tilde \beta}>-{\tilde \alpha}$. Eventually, 
by \ref{A3} $\mathcal W_{\mathcal A}$ is a Gaussian process with continuous trajectories in ${\mathcal H}$, see \cite[Chap.~4]{Dap-Zab14}.
\end{rmk}

Now, we need to adapt Proposition~\ref{prop:Xn}
on the convergence of approximations. For every $n\in\enne_+$, 
let $X_n$ be defined as 
\begin{align}\label{eq-approssimata-app}
X_n(t)+\int_0^t\mathcal A_nX_n(s)\,\d s
=\mathcal P_nx + \int_0^tB_n(X(s))\,\d s
+\int_0^t\mathcal A_n^{-\tilde\delta}\,\d \mathcal W(s)\,, \quad\forall\,t\geq0\,,
\quad\P\text{-a.s.},
\end{align}
where $\mathcal P_n:\mathcal H\to\mathcal H$ 
is the orthogonal projection on $\mathcal H_n:=\operatorname{span}\{\tilde e_1,\ldots,\tilde e_n\}$.

\begin{proposition}\label{prop:Xn_alpha}
Assume \ref{A0}--\ref{A4} hold  and let $x\in D({\mathcal A}^{-{\tilde \alpha}})$. Then, 
for every $T>0$ we have
\begin{equation}\label{supE_alpha}
\lim_{n \to \infty}\sup_{t\in (0,T]}t^{{\tilde \alpha}}\mathbb{E}
\|X_{n}(t)-X(t)\|_{\mathcal H}=0\,.
\end{equation} 
\end{proposition}
\begin{proof}
The proof is similar to the one of Proposition~\ref{prop:Xn}:
the only point to adapt is the convergence of the initial data.
To this end, we note that
from hypothesis \ref{A1} it follows that, for every $t>0$ and $z\in {\mathcal H}$,
\begin{align*}
\|e^{-t{\mathcal A}}z\|_{\mathcal H}^2
&=\sum_{k=1}^\infty
\left(e^{-t{\mathcal A}}z, \tilde e_k\right)^2_{{\mathcal H}}
=  \sum_{k=1}^\infty \lambda_k^{2{\tilde \alpha}}e^{-2\lambda_kt}
({\mathcal A}^{-{\tilde \alpha}}z,\tilde e_k)_{\mathcal H}^2
=\frac1{t^{2{\tilde \alpha}}}\sum_{k=1}^\infty (\lambda_k t)^{2{\tilde \alpha}}e^{-2\lambda_kt}
({\mathcal A}^{-{\tilde \alpha}}z,\tilde e_k)_{\mathcal H}^2
\\
&\leq \frac{\max_{r\geq0}(r^{2{\tilde \alpha}}e^{-2r})}
{t^{2{\tilde \alpha}}}\sum_{k=1}^\infty( {\mathcal A}^{-{\tilde \alpha}}z,\tilde e_k)_{\mathcal H}^2
= C_{\tilde \alpha} t^{-2{\tilde \alpha}}\|z\|_{D({\mathcal A}^{-{\tilde \alpha}})}^2,
\end{align*}
where $C_{\tilde \alpha}$ is a positive constant which only depends on ${\tilde \alpha}$. We infer that $t^{{\tilde \alpha}}\|e^{-t{\mathcal A}}\|_{\mathscr L(D({\mathcal A}^{-{\tilde \alpha}}),{\mathcal H})}\leq C_{\tilde \alpha}$ for every $t>0$ by density. 
Moreover, one has 
\begin{align*}
\|\mathcal P_nx-x\|_{D({\mathcal A}^{-{\tilde \alpha}})}^2=
\norm{\sum_{k=n+1}^\infty({\mathcal A}^{-{\tilde \alpha}}x,\tilde e_k)_{\mathcal H}\tilde e_k
}_{{\mathcal H}}^2
= \sum_{k=n+1}^\infty ( {\mathcal A}^{-{\tilde \alpha}}x,\tilde e_k)_{{\mathcal H}}^2\,.
\end{align*}
By taking these remarks into account, we obtain that 
\begin{align*}
\sup_{t\geq0}{t^{{\tilde \alpha}}}
\norm{e^{-t{\mathcal A}_n} \mathcal P_nx-e^{-t{\mathcal A}}x}_{\mathcal H}
&=\sup_{t\geq0}{t^{{\tilde \alpha}}}\norm{e^{-t{\mathcal A}}
(\mathcal P_nx-x)}_{\mathcal H}\\
&\leq 
\sup_{t\geq0}t^{{{\tilde \alpha}}}\norm{e^{-t{\mathcal A}}}_{\mathscr{L}({D({\mathcal A}^{-{\tilde \alpha}})},{\mathcal H})}
\norm{\mathcal P_nx-x}_{{D({\mathcal A}^{-{\tilde \alpha}})}}.
\end{align*}
It follows that
\begin{align*}
\lim_{n\to\infty}\sup_{t\geq0}{t^{{\tilde \alpha}}}
\norm{e^{-t{\mathcal A}_n} \mathcal P_nx-e^{-t{\mathcal A}}x}_{\mathcal H}=0.    
\end{align*}
The proof can be concluded then by proceeding as in 
the one of Proposition~\ref{prop:Xn}.
\end{proof}

We are now ready to present the technical argument to extend 
the pathwise uniqueness and continuous dependence to general initial data.
\begin{proposition}
\label{th:uniq_path-app}
Assume \ref{A0}--\ref{A5} and that
\begin{align}
\label{trace_app}
&\exists\,\varepsilon \in (0,1):\quad
\lc{{\mathcal A}^{-(1+\theta)+2{\tilde \beta}
+2(1-\theta)\max\{-\tilde\alpha,\tilde \delta\}+2\varepsilon}\in\cL^1({\mathcal H},{\mathcal H})\,,}\\
\label{ineq_alfa_app}
   &{\tilde \alpha}\leq {\tilde \beta}\,, \qquad
\lc{\tilde\alpha + 
(1-\theta)\max\{-\tilde\alpha,\tilde\delta\}
<\frac{\theta}{2}}\,.
\end{align}
Then, for every $T>0$ there exists a constant $\mathfrak L>0$, depending only on the structural data ${\tilde \alpha},{\tilde \beta},{\tilde \delta},\theta,T$, such that, 
for every weak mild solutions $(X,\mathcal W)$ and $(Y,\mathcal  W)$ to \eqref{SDE-app} with initial data $x,y\in D({\mathcal A}^{-{\tilde \alpha}})$, respectively, in the sense of Definition~\ref{def-sol-app}, with the same cylindrical Wiener process $\mathcal W$ and defined on the same probability space, it holds that 
\begin{equation}
\label{cont_dep-app}
\|X-Y\|_{C^0([0,T];L^2(\Omega; D({\mathcal A}^{-{\tilde \alpha}}))}\leq 
\mathfrak L\|x-y\|_{D({\mathcal A}^{-{\tilde \alpha}})}.
\end{equation}
\end{proposition}
\begin{proof}
Let $(X,\mathcal W)$ and $(Y,\mathcal W)$ be two weak mild solutions to \eqref{SDE-app} with initial data $x$ and $y$, respectively, in the sense of Definition~\ref{def-sol-app}, with the same cylindrical Wiener process $\mathcal W$ and defined on the same probability space.
Let $(X_n)_n$ and $(Y_n)_n$ be the respective approximated solutions as constructed above and satisfying the convergence 
in Proposition~\ref{prop:Xn_alpha}.
\lc{Let us note that the same arguments employed to 
obtain \eqref{est1}--\eqref{est2_bis} and
\eqref{est1'}--\eqref{est2'_bis} yield
in this setting that 
\begin{align}
\label{est1'_tilde}
\|\mathcal A_n^{\gamma} Du_{n,k}\|_{C^0_b(\mathcal H_n; \mathcal H_n)}&\leq\frac{\mathfrak L_2\lambda_k^{\tilde\beta}}{(\bar c\lambda_k)^{\frac{1+\theta}{2}-(1-\theta)\lc{\max\{0,\gamma+\tilde\delta\}-\theta\gamma}}}\,,\\
\label{est2'_tilde}
\|D(\mathcal A_n^{\gamma_2}Du_{n,k})\mathcal A_n^{\gamma_1
}\|_{C^0_b(\mathcal H_n; \cL(\mathcal H_n,\mathcal H_n))}
&\leq \lc{\frac{\mathfrak L_3\lambda_k^{\tilde\beta}}
{(\bar c\lambda_k)^{\lc{
\frac{\theta}{2}-
\max\{0,\gamma_1+\tilde\delta\}-
(1-\theta)\max\{0,\gamma_2+\tilde\delta\}
-\theta\max\{0,\gamma_2\}
}}}\,,}\\
\label{est2'_tilde_bis}
\|D(\mathcal A_n^{\gamma_1}Du_{n,k})\mathcal A_n^{\gamma_2}\|_{C^0_b(\mathcal H_n; \cL(\mathcal H_n,\mathcal H_n))}
&\leq \lc{\frac{\mathfrak L_3\lambda_k^{\tilde\beta}}
{(\bar c\lambda_k)^{\lc{
\frac{\theta}{2}-
\max\{0,\gamma_1+\tilde\delta\}-
(1-\theta)\max\{0,\gamma_2+\tilde\delta\}
-\theta\max\{0,\gamma_2\}
}}}\,,}
\end{align}
for all $\gamma\in[0,\tilde\beta]$
and for all $\gamma_1,\gamma_2\in\R$ with
$\max\{0,\gamma_1+\tilde\delta\}+
(1-\theta)\max\{0,\gamma_2+\tilde\delta\}
+\theta\max\{0,\gamma_2\}<\frac{\theta}{2}$.}

The first part of proof is analogous to {\sc Step 1} of the proof of Theorems~\ref{th:3}:
we only show how to deal with the terms $I_0,\ldots,I_7$
under the approximation results of Proposition~\ref{prop:Xn_alpha}.
We estimate the left-hand side of \eqref{diff}
in the space $L^2(\Omega; D({\mathcal A}^{-{\tilde \alpha}}))$
by analyzing all the terms $I_0,\ldots,I_7$ separately.
Let $T_0\in(0,T]$, whose value will be specified later.

For the term $I_0$ we have, 
\begin{align*}
I_0= e^{-t{\mathcal A}}\mathcal P_n(x-y)+e^{-t{\mathcal A}}\sum_{k=1}^n\int_0^1({\mathcal A}_n^{{\tilde \alpha}} Du_{n,k}(\mathcal P_ny+r\mathcal P_n(x-y)), {\mathcal A}_n^{-{\tilde \alpha}}
\mathcal P_n(x-y))_{{\mathcal H}}e_k\,\d r
\end{align*}
so that, by using \eqref{est1'_tilde} 
with $\gamma={\tilde \alpha}$
(which is admissible by \eqref{ineq_alfa_app})
and the fact that 
$e^{-t{\mathcal A}}$ and $\mathcal P_n$ are contractions in $D({\mathcal A}^{-{\tilde \alpha}})$, 
we infer that 
\begin{align}
\notag
&\|I_0\|^2_{L^2(\Omega; D({\mathcal A}^{-{\tilde \alpha}}))}\\
\label{E0}
&\leq   \left[1+
\frac{\mathfrak L_2^2}{\bar c^{1+\theta
-2(1-\theta)\lc{\max\{0,\tilde\alpha+\tilde \delta\}}
-\lc{2\theta\tilde\alpha}
}}
\sum_{k=1}^n
\frac1{\lambda_k^{1+\theta-2{\tilde \beta}
-2(1-\theta)\lc{\max\{-\tilde\alpha,\tilde \delta\}}
}}\right]\|x-y\|_{D({\mathcal A}^{-{\tilde \alpha}})}^2
\end{align}
Analogously, for the term $I_1$ we have,
\begin{align*}
I_1(t)&=\int_0^1Du_n(X_n(t)+r(Y_n(t)-X_n(t)))[Y_n(t)-X_n(t)]\,\d r\\
&=\sum_{k=1}^n\int_0^1({\mathcal A}^{{\tilde \alpha}}Du_{n,k}(X_n(t)+r(Y_n(t)-X_n(t))), {\mathcal A}^{-{\tilde \alpha}}(
Y_n(t)-X_n(t)))_{\mathcal H}e_k\,\d r\,,
\end{align*}
so that, by using \eqref{est1'_tilde} 
with $\gamma={\tilde \alpha}$, we infer that 
\begin{align*}
&\|I_1(t)\|^2_{L^2(\Omega; D({\mathcal A}^{-{\tilde \alpha}}))}\\
&\leq \frac{\mathfrak L_2^2}{\bar c^{1+\theta
-2(1-\theta)\lc{\max\{0,\tilde\alpha+\tilde \delta\}}
-\lc{2\theta\tilde\alpha}
}}
\left(\sum_{k=1}^n
\frac1{\lambda_k^{1+\theta-2\tilde\beta
-2(1-\theta)\lc{\max\{-\tilde\alpha,\tilde \delta\}}
}}\right)
\|X_n(t)-Y_n(t)\|_{L^2(\Omega; D({\mathcal A}^{-{\tilde \alpha}}))}^2\,.
\end{align*}
For the term $I_2$, by \eqref{est1'_tilde} with $\gamma={\tilde \alpha}$ and 
the H\"older inequality we have
\begin{align*}
&\|I_2(t)\|^2_{L^2(\Omega;D({\mathcal A}^{-{\tilde \alpha}}))} \notag \\
&=(\bar c +1)^2\sum_{k=1}^{n}\mathbb{E}
\left(\int_0^te^{-(t-s){\mathcal A}_n}{\mathcal A}_n
\left[u_n(X_n(s))-u_n(Y_n(s))\right]\, \d s,{\mathcal A}^{{\tilde \alpha}}e_k\right)_{D({\mathcal A}^{-{\tilde \alpha}})}^2 \notag \\
&=(\bar c +1)^2\sum_{k=1}^{n}
\mathbb{E}\left(\int_0^te^{-(t-s)\lambda_k}\lambda_k^{1-{\tilde \alpha}}\left[u_{n,k}(X_n(s))-u_{n,k}(Y_n(s))\right]\,\d s\right)_{\mathcal H}^2 \notag \\
& = (\bar c +1)^2\sum_{k=1}^{n}
\mathbb{E}\left(\int_0^te^{-(t-s)\lambda_k}\lambda_k^{1-{\tilde \alpha}}\int_0^1(Du_{n,k}(X_n(s)+r(Y_n(s)-X_n(s))),Y_n(s)-X_n(s))_{\mathcal H}\,\d r\,\d s\right)^2 \notag \\
&\leq \frac{\mathfrak{L}^2_2(\bar c +1)^2}{\bar{c}^{
1+\theta
-2(1-\theta)\lc{\max\{0,\tilde\alpha+\tilde \delta\}}
-\lc{2\theta\tilde\alpha}
}}\\
&\qquad\times
\sum_{k=1}^{n}\mathbb{E}\left(\int_0^te^{-(t-s)\lambda_k}\lambda_k\frac{1}{\lambda_k^{\frac{1+\theta}{2}-{\tilde \beta}-(1-\theta)
\lc{\max\{-\tilde\alpha, \tilde\delta\}}
}}\norm{{\mathcal A}^{-{\tilde \alpha}}(X_n(s)-Y_n(s))}_{\mathcal H}\,\d s\right)^2 \notag \\
&\leq \frac{\mathfrak{L}^2_2(\bar c +1)^2}{\bar{c}^{
1+\theta
-2(1-\theta)\lc{\max\{0,\tilde\alpha+\tilde \delta\}}
-\lc{2\theta\tilde\alpha}
}}\sum_{k=1}^{n}
\frac{1}{\lambda_k^{\theta-1-2{\tilde \beta}-2(1-\theta)\lc{\max\{-\tilde\alpha, \tilde\delta\}}}}\\
&\qquad\times\mathbb{E}
\left(\int_0^te^{-\frac{(t-s)\lambda_k}{2}}e^{-\frac{(t-s)\lambda_k}{2}}\norm{X_n(s)-Y_n(s)}_{D({\mathcal A}^{-{\tilde \alpha}})}\,\d s\right)^2 \notag \\
&\leq \frac{\mathfrak{L}^2_2(\bar c +1)^2}{\bar{c}^{
1+\theta
-2(1-\theta)\lc{\max\{0,\tilde\alpha+\tilde \delta\}}
-\lc{2\theta\tilde\alpha}
}}\sum_{k=1}^{n}
\frac{1}{\lambda_k^{\theta-1-2{\tilde \beta}-2(1-\theta)\lc{\max\{-\tilde\alpha, \tilde\delta\}}}}\\
&\qquad\times
\int_0^te^{-(t-s)\lambda_k}\,\d s
\int_0^te^{-(t-s)\lambda_k}\mathbb{E}\norm{X_n(s)-Y_n(s)}^2_{D({\mathcal A}^{-{\tilde \alpha}})}\, \d s\,, \notag 
\end{align*}
so that, for every $t\in(0,T_0]$,
\begin{align*}
 &\|I_2(t)\|^2_{L^2(\Omega;D({\mathcal A}^{-{\tilde \alpha}}))}\\
&\leq \frac{T_0^{2\varepsilon}\mathfrak{L}^2_2(\bar c +1)^2}{\bar{c}^{
1+\theta
-2(1-\theta)\lc{\max\{0,\tilde\alpha+\tilde \delta\}}
-\lc{2\theta\tilde\alpha}
}}\sum_{k=1}^{n}
\frac{1}{\lambda_k^{\theta+1-2{\tilde \beta}-2(1-\theta)\lc{\max\{-\tilde\alpha, \tilde\delta\}}-2\varepsilon}}
\sup_{s\in[0,t]}\E\norm{X_n(s)-Y_n(s)}_{D({\mathcal A}^{-{\tilde \alpha}})}^2.
\end{align*}
For the term $I_3$, by \eqref{est1'_tilde} 
with $\gamma={\tilde \beta}$ and assumption \ref{A4} we get
\begin{align*}
\nonumber
&\|I_3(t)\|^2_{L^2(\Omega; D({\mathcal A}^{-{\tilde \alpha}}))}\\
\nonumber
&=
\sum_{k=1}^n\mathbb{E}
\left(\int_0^te^{-(t-s){\mathcal A}_n}Du_n(X_n(s))\left[B_n(X(s))-B_n(X_n(s))\right]\,\d s, {\mathcal A}^{\tilde \alpha} e_k\right)_{D({\mathcal A}^{-{\tilde \alpha}})}^2\\
\nonumber
&=\sum_{k=1}^n\mathbb{E}\left(\int_0^te^{-(t-s)\lambda_k}\lambda_k^{-{\tilde \alpha}}
\left(Du_{n,k}(X_n(s)), B_n(X(s))-B_n(X_n(s))\right)_{\mathcal H}\,\d s\right)^2\\
\nonumber
&\leq \frac{C_B^2\mathfrak L_2^2}{\bar c^{1+\theta
-2{\tilde \beta}-2(1-\theta){\tilde \delta}}}\sum_{k=1}^n\frac{1}{\lambda_k^{1+\theta+2{\tilde \alpha}
-4{\tilde \beta}-2(1-\theta){\tilde \delta}}}
\left(\int_0^t e^{-(t-s)\lambda_k}\mathbb{E}\norm{X_n(s)-X(s)}_{\mathcal H}^\theta\,\d s\right)^2\\
\nonumber
& = \frac{C_B^2\mathfrak L_2^2}{\bar c^{1+\theta
-2{\tilde \beta}-2(1-\theta){\tilde \delta}}}\sum_{k=1}^n\frac{1}{\lambda_k^{1+\theta+2{\tilde \alpha}
-4{\tilde \beta}-2(1-\theta){\tilde \delta}}}
\left(\int_0^t e^{-(t-s)\lambda_k}s^{-\theta {\tilde \alpha}}s^{\theta{\tilde \alpha}}\mathbb{E}\norm{X_n(s)-X(s)}_{\mathcal H}^\theta\,\d s\right)^2\\
\nonumber
&\leq \frac{C_B^2\mathfrak L_2^2}{\bar c^{1+\theta-2{\tilde \beta}-2(1-\theta){\tilde \delta}}}\sum_{k=1}^n\frac{1}{\lambda_k^{1+\theta+2{\tilde \alpha}
-4{\tilde \beta}-2(1-\theta){\tilde \delta}}}
\sup_{s\in(0,t]}\left(s^{2\theta{\tilde \alpha}}\mathbb{E}\norm{X_n(s)-X(s)}_{\mathcal H}^{2\theta}\right)\left(\int_0^t \frac{e^{-\lambda_k(t-s)}}{s^{\tilde \alpha}}\,\d s\right)^2 \\ 
& \leq \frac{C_B^2\mathfrak L_2^2}{\bar c^{1+\theta-2{\tilde \beta}-2(1-\theta){\tilde \delta}}}\left(\sum_{k=1}^n\frac{1}{\lambda_k^{3+\theta-4{\tilde \beta}-2(1-\theta){\tilde \delta}}}\right)
\sup_{s\in(0,t]}\left(s^{2\theta{\tilde \alpha}}\mathbb{E}\norm{X_n(s)-X(s)}_{\mathcal H}^{2\theta}\right)\,.
\end{align*}
In the same way, for $I_4$ we get
\begin{align*}
\|I_4(t)\|^2_{L^2(\Omega; D({\mathcal A}^{-{\tilde \alpha}}))}\leq 
\frac{C_B^2\mathfrak L_2^2}{\bar c^{1+\theta
-2{\tilde \beta}-2(1-\theta){\tilde \delta}}}
\left(\sum_{k=1}^{n}\frac{1}{\lambda_k^{3+\theta
-4{\tilde \beta}-2(1-\theta){\tilde \delta}}}\right)
\sup_{s\in(0,t]}\left(s^{2\theta{\tilde \alpha}}\mathbb{E}\norm{X_n(s)-X(s)}_{\mathcal H}^{2\theta}\right)\,.
\end{align*}
For $I_5$, we have
\begin{align*}
I_5(t)
=\int_0^t{\mathcal A}_n^{\tilde \beta} e^{-(t-s){\mathcal A}_n}
{\mathcal A}_n^{-{\tilde \beta}}[B_n(X(s))-B_n(X_n(s))]\,\d s\,,
\end{align*}
Since ${\tilde \beta}-{\tilde \alpha}\in[0,1)$, we can fix 
$\eta\in(({\tilde \beta}-{\tilde \alpha}-\frac12)\vee0, \frac12)$:
by \eqref{semiS} with $\gamma={\tilde \beta}$, \eqref{est:Bn2} and the H\"older inequality, we get
\begin{align*}
&\mathbb{E}\norm{I_5(t)}_{D({\mathcal A}^{-{\tilde \alpha}})}^2
\leq \mathbb{E}\left[\left(\int_0^t\norm{{\mathcal A}_n^{{\tilde \beta}-{\tilde \alpha}} e^{-(t-s){\mathcal A}_n}
{\mathcal A}_n^{-{\tilde \beta}}[B_n(X(s))-B_n(X_n(s))]}_{\mathcal H}\,\d s\right)^2\right]\\
\nonumber
&\qquad\leq \mathbb{E}\left[M_{{\tilde \beta}-{\tilde \alpha}}^2C_B^2\left(\int_0^t\frac{1}{(t-s)^{{\tilde \beta}-{\tilde \alpha}-\eta}}
\frac{1}{(t-s)^{\eta}}\norm{X(s)-X_n(s)}^\theta_{\mathcal H}\,\d s\right)^2\right]\\
\nonumber
&\qquad\leq \mathbb{E}\left[M_{{\tilde \beta}-{\tilde \alpha}}^2C_B^2\left(\int_0^t\frac{1}{(t-s)^{2({\tilde \beta}-\eta-{\tilde \alpha})}}\,\d s\right)\left(\int_0^t\frac{s^{2\theta{\tilde \alpha}}}{(t-s)^{2\eta}s^{2\theta{\tilde \alpha}}}\norm{X(s)-X_n(s)}^{2\theta}_{\mathcal H}\,\d s\right)\right]\\
&\qquad\leq \frac{M_{{\tilde \beta}-{\tilde \alpha}}^2C_B^2 
t^{1-2({\tilde \beta}-\eta-{\tilde \alpha})}}{1-2({\tilde \beta}-\eta-{\tilde \alpha})}
\sup_{s\in(0,t]}\left(s^{2\theta{\tilde \alpha}}\mathbb{E}
\norm{X(s)-X_n(s)}^{2\theta}_{\mathcal H}\right)
\int_0^t\frac{1}{(t-s)^{2\eta}s^{2\theta{\tilde \alpha}}}\,\d s \nonumber \\
&\qquad = \frac{M_{{\tilde \beta}-{\tilde \alpha}}^2C_B^2 K_{\eta,\theta,{\tilde \alpha}}
t^{2-2{\tilde \beta}+2(1-\theta){\tilde \alpha}}}{1-2({\tilde \beta}-\eta-{\tilde \alpha})}
\sup_{s\in(0,t]}\left(s^{2\theta{\tilde \alpha}}\mathbb{E}
\norm{X(s)-X_n(s)}^{2\theta}_{\mathcal H}\right)\,,
\end{align*}
where the integral
\[
  K_{\eta,\theta,{\tilde \alpha}}:=\int_0^1\frac{1}{(1-s)^{2\eta}s^{2\theta{\tilde \alpha}}}\,\d s<+\infty
\]
is finite since
$2\eta\in(0,1)$ and $2\theta{\tilde \alpha}\in(0,1)$.
Analogously, we obtain for $I_6$ that 
\begin{align*}
\mathbb{E}\norm{I_6(t)}_{\mathcal H}^2\leq 
\frac{M_{{\tilde \beta}-{\tilde \alpha}}^2C_B^2K_{\eta,\theta,{\tilde \alpha}} 
t^{2-2{\tilde \beta}+2(1-\theta){\tilde \alpha}}}{1-2({\tilde \beta}-\eta-{\tilde \alpha})}
\sup_{s\in(0,t]}\left(s^{2\theta{\tilde \alpha}}\mathbb{E}
\norm{Y(s)-Y_n(s)}^{2\theta}_{\mathcal H}\right)\,.
\end{align*}
Eventually, for $I_7$, by the It\^o isometry we get
\begin{align*}
&\|I_7(t)\|^2_{L^2(\Omega; D({\mathcal A}^{-{\tilde \alpha}}))}\\
&=
\E\sum_{k=1}^n\left(\int_0^te^{-(t-s){\mathcal A}_n}\left(Du_n(X_n(s))-Du_n(Y_n(s)\right)[{\mathcal A}_n^{-{\tilde \delta}}\,\d \mathcal W(s)], {\mathcal A}^{{\tilde \alpha}}\tilde e_k\right)^2_{D({\mathcal A}^{-{\tilde \alpha}})}\\
&=\sum_{k=1}^n\lambda_k^{-2{\tilde \alpha}}\mathbb{E}\left(\int_0^te^{-(t-s)\lambda_k}\left(Du_{n,k}(X_n(s))-Du_{n,k}(Y_n(s)), {\mathcal A}_n^{-{\tilde \delta}}\,\d \mathcal W(s)\right)_{\mathcal H}\right)^2\\
&=\sum_{k=1}^n\lambda_k^{-2{\tilde \alpha}}\sum_{j=1}^{n}\int_0^te^{-2(t-s)\lambda_k}\mathbb{E}
\left(Du_{n,k}(X_n(s))-Du_{n,k}(Y_n(s)),{\mathcal A}_n^{-{\tilde \delta}}\tilde e_j\right)_{\mathcal H}^2\,\d s\\
&=\sum_{k=1}^n\lambda_k^{-2{\tilde \alpha}}\int_0^te^{-2(t-s)\lambda_k}\mathbb{E}
\norm{{\mathcal A}_n^{-{\tilde \delta}}(Du_{n,k}(X_n(s))-Du_{n,k}(Y_n(s)))}_{\mathcal H}^2\,\d s\\
&=\sum_{k=1}^n\lambda_k^{-2{\tilde \alpha}}\int_0^te^{-2(t-s)\lambda_k}\mathbb{E}
\norm{
\int_0^1
D({\mathcal A}_n^{-{\tilde \delta}}Du_{n,k})(X_n(s)+r(Y_n(s)-X_n(s)))[Y_n(s)-X_n(s)]\,\d r
}_{\mathcal H}^2\,\d s\\
&=\lc{\sum_{k=1}^n\lambda_k^{-2{\tilde \alpha}}\int_0^te^{-2(t-s)\lambda_k}}\\
&\qquad\lc{\times\mathbb{E}
\norm{
\int_0^1
D({\mathcal A}_n^{-{\tilde \delta}}Du_{n,k})(X_n(s)+r(Y_n(s)-X_n(s)))
{\mathcal A}_n^{{\tilde \alpha}}
[{\mathcal A}_n^{-{\tilde \alpha}}(Y_n(s)-X_n(s))]\,\d r
}_{\mathcal H}^2\,\d s.}
\end{align*}
\lc{By the estimate \eqref{est2'_tilde_bis} with 
$\gamma_1=-{\tilde \delta}$ and 
$\gamma_2=\tilde \alpha$
(which we can apply thanks to \eqref{ineq_alfa_app}),
we obtain
\begin{align}
\label{E7}\notag
&\|I_7(t)\|^2_{L^2(\Omega; D({\mathcal A}^{-{\tilde \alpha}}))}\\
\notag
&\leq\frac{\mathfrak L_3^2}
{\bar c^{\theta
-2(1-\theta)\max\{0,\tilde\alpha+\tilde\delta\}
-2\theta\tilde\alpha
}}\sum_{k=1}^n
\frac{1}{\lambda_k^{\theta
-2{\tilde \beta}
-2(1-\theta)\max\{0,\tilde\alpha+\tilde\delta\}
+2(1-\theta)\tilde\alpha
}}\\
\notag
&\qquad\times\int_0^te^{-2(t-s)\lambda_k}\mathbb{E}\norm{X_n(s)-Y_n(s)}_{D({\mathcal A}^{-{\tilde \alpha}})}^2
\,\d s\\
&\leq\frac{\mathfrak L_3^2}
{2\bar c^{
\theta
-2(1-\theta)\max\{0,\tilde\alpha+\tilde\delta\}
-2\theta\tilde\alpha
}}\sum_{k=1}^n\frac{1}{
\lambda_k^{
1+\theta
-2{\tilde \beta}
-2(1-\theta)\max\{-\tilde\alpha,\tilde\delta\}
}}
\sup_{s\in[0,t]}
\mathbb{E}\norm{X_n(s)-Y_n(s)}_{D({\mathcal A}^{-{\tilde \alpha}})}^2\,.
\end{align}
Recalling that, thanks to
assumption \eqref{trace_app},
\[
  \sum_{k=1}^n
    \frac1{\lambda_k^{1+\theta
      -2{\tilde \beta}-2(1-\theta)\max\{-\tilde\alpha,\tilde\delta\}-2\varepsilon}}\leq
  \sum_{k=1}^\infty
    \frac1{\lambda_k^{1+\theta
      -2{\tilde \beta}-2(1-\theta)\max\{-\tilde\alpha,\tilde\delta\}-2\varepsilon}} =: s_1 <+\infty\,,
\]
}
and 
\[
s_2:= \sum_{k=1}^\infty\frac{1}{\lambda_k^{3+\theta
-4{\tilde \beta}-2(1-\theta){\tilde \delta}}}\leq s_1<+\infty\,,
\]
from \eqref{E0}-\eqref{E7}, it follows that
\begin{align*}
&\sup_{t\in[0,T_0]}\|X_n(t)-Y_n(t)\|_{L^2(\Omega;D({\mathcal A}^{-{\tilde \alpha}}))}^2
\leq  \left[1+
\frac{\mathfrak L_2^2 s_1}{\bar c^{
1+\theta
-2(1-\theta)\lc{\max\{0,\tilde\alpha+\tilde \delta\}}
-\lc{2\theta\tilde\alpha}
}}
\right]\|x-y\|_{D({\mathcal A}^{-{\tilde \alpha}})}^2 \\
&\qquad + 
\left(\frac{\mathfrak L_2^2 s_1
+T_0^{2\varepsilon}\mathfrak{L}^2_2(\bar c +1)^2 s_1}{\bar c^{
1+\theta
-2(1-\theta)\lc{\max\{0,\tilde\alpha+\tilde \delta\}}
-\lc{2\theta\tilde\alpha}
}}
+\frac{\mathfrak L_3^2 s_1}
{2\bar c^{
\lc{\theta
-2(1-\theta)\max\{0,\tilde\alpha+\tilde\delta\}
-2\theta\tilde\alpha}
}}
\right)
\sup_{t\in[0,T_0]}\E\|X_n(t)-Y_n(t)\|_{D({\mathcal A}^{-{\tilde \alpha}})}^2 \\
&\qquad + 
\left[\frac{C_B^2\mathfrak L_2^2 s_2}{\bar c^{1+\theta-2{\tilde \beta}-2(1-\theta){\tilde \delta}}}
+\frac{M_{{\tilde \beta}-{\tilde \alpha}}^2C_B^2 K_{\eta,\theta,{\tilde \alpha}}
T^{2-2{\tilde \beta}+2(1-\theta){\tilde \alpha}}}{1-2({\tilde \beta}-\eta-{\tilde \alpha})}
\right]
\sup_{t\in(0,T_0]}\left(t^{2\theta{\tilde \alpha}}\mathbb{E}\norm{X_n(t)-X(t)}_{\mathcal H}^{2\theta}\right) \\
&\qquad + 
\left[\frac{C_B^2\mathfrak L_2^2 s_2}{\bar c^{1+\theta-2{\tilde \beta}-2(1-\theta){\tilde \delta}}}
+\frac{M_{{\tilde \beta}-{\tilde \alpha}}^2C_B^2 K_{\eta,\theta,{\tilde \alpha}}
T^{2-2{\tilde \beta}+2(1-\theta){\tilde \alpha}}}{1-2({\tilde \beta}-\eta-{\tilde \alpha})}
\right]
\sup_{t\in(0,T_0]}\left(t^{2\theta{\tilde \alpha}}\mathbb{E}\norm{Y_n(t)-Y(t)}_{\mathcal H}^{2\theta}\right).
\end{align*}
Now, since the exponents of the constant $\bar c$ are all positive
thanks to assumption \eqref{ineq_alfa_app},
we can fix $\bar c$ satisfying \eqref{c}--\eqref{c'} large enough
and $T_0\in(0,T]$ small enough so that 
\[
  \frac{\mathfrak L_2^2 s_1
+T_0^{2\varepsilon}\mathfrak{L}^2_2(\bar c +1)^2 s_1}{\bar c^{
1+\theta
-2(1-\theta)\lc{\max\{0,\tilde\alpha+\tilde \delta\}}
-\lc{2\theta\tilde\alpha}
}}
+\frac{\mathfrak L_3^2 s_1}
{2\bar c^{
\lc{\theta
-2(1-\theta)\max\{0,\tilde\alpha+\tilde\delta\}
-2\theta\tilde\alpha}
}}<\frac12\,.
\]
Hence, by letting $n\to\infty$ and by exploiting Proposition~\ref{prop:Xn_alpha}
we get 
\begin{align*}
    \|X-Y\|^2_{C^0([0,T_0]; L^2(\Omega; D(A^{-{\tilde \alpha}})))}
    &\leq 
    2\left[1+
\frac{\mathfrak L_2^2 s_1}{\bar c^{
1+\theta
-2(1-\theta)\lc{\max\{0,\tilde\alpha+\tilde \delta\}}
-\lc{2\theta\tilde\alpha}
}}
\right]\|x-y\|_{D(A^{-{\tilde \alpha}})}^2
\end{align*}
and the continuous dependence \eqref{cont_dep-app} is proved
via a classical patching argument.
This concludes the proof.
\end{proof}

\section*{Data availability statement}
No new data were created or analysed in this study. Data sharing is not applicable to this article.

\section*{Conflict of interest statement}
The authors have no conflicts of interest to declare.

\section*{Acknowledgments}
\lc{The authors are very grateful to the anonymous referees for their valuable suggestions, which have greatly helped to improve the quality and clarity of the manuscript.}
The authors are members of Gruppo Nazionale per l'Analisi Matematica, la Probabilit\`a e le loro Applicazioni (GNAMPA), Istituto Nazionale di Alta Matematica (INdAM). The research of L.S.~has been supported by MUR, 
grant Dipartimento di Eccellenza 2023-2027.

\bibliographystyle{siam}
\bibliography{biblio.bib}

\end{document}